\documentclass[10pt]{article}

\usepackage{fullpage}
\usepackage{amsmath, amsfonts, amssymb, amsthm, amsxtra, verbatim,graphics}
\usepackage{mathrsfs,bbm}
\usepackage{hyperref}
\usepackage[usenames,dvipsnames]{color}
\usepackage {graphicx}

\usepackage{etoolbox}
\apptocmd{\sloppy}{\hbadness 10000\relax}{}{}

\begin{document}

\theoremstyle{plain}
\newtheorem{thm}{Theorem}[section]
\newtheorem{cor}[thm]{Corollary}
\newtheorem{lem}[thm]{Lemma}
\newtheorem{prop}[thm]{Proposition}
\newtheorem{conj}[thm]{Conjecture}
\newtheorem{notation}[thm]{Notation}
\theoremstyle{definition}
\newtheorem{defn}[thm]{Definition}
\newtheorem{remark}[thm]{Remark}
\newtheorem{ack}[thm]{Acknowledgements}
\newtheorem{qn}[thm]{Question}

\newcommand{\arXiv}[1]{\href{http://arxiv.org/abs/#1}{arXiv:#1}}
\newcommand{\arxiv}[1]{\href{http://arxiv.org/abs/#1}{#1}}
\newcommand{\e}[1]{\noindent{\textbf{(#1)}}}
\newcommand{\Hqed}{\textpmhg{F}}
\newcommand{\dsb}{\begin{adjustwidth}{2.5em}{0pt}
\begin{footnotesize}}
\newcommand{\dse}{\end{footnotesize}
\end{adjustwidth}}
\newcommand{\ssb}{\begin{adjustwidth}{2.5em}{0pt}}
\newcommand{\sse}{\end{adjustwidth}}
\newcommand{\aryb}{\begin{eqnarray*}}
\newcommand{\arye}{\end{eqnarray*}}
\def\alb#1\ale{\begin{align*}#1\end{align*}}
\newcommand{\eqb}{\begin{equation}}
\newcommand{\eqe}{\end{equation}}
\newcommand{\eqbn}{\begin{equation*}}
\newcommand{\eqen}{\end{equation*}}
\newcommand{\BB}{\mathbb}
\newcommand{\ol}{\overline}
\newcommand{\ul}{\underline}
\newcommand{\op}{\operatorname}
\newcommand{\la}{\langle}
\newcommand{\ra}{\rangle}
\newcommand{\bd}{\mathbf}
\newcommand{\im}{\operatorname{Im}}
\newcommand{\re}{\operatorname{Re}}
\newcommand{\frk}{\mathfrak}
\newcommand{\eqD}{\overset{d}{=}}
\newcommand{\ep}{\epsilon}
\newcommand{\rta}{\rightarrow}
\newcommand{\xrta}{\xrightarrow}
\newcommand{\Rta}{\Rightarrow}
\newcommand{\hookrta}{\hookrightarrow}
\newcommand{\wt}{\widetilde}
\newcommand{\wh}{\widehat} 
\newcommand{\mcl}{\mathcal}
\newcommand{\pre}{{\operatorname{pre}}}
\newcommand{\lrta}{\leftrightarrow}
\newcommand{\av}{a}
\newcommand{\scr}{\mathscr}

\newcommand{\N}{\mathbb{N}}
\newcommand{\E}{\mathbb{E}}
\newcommand{\Z}{\mathbb{Z}}
\newcommand{\R}{\mathbb{R}}
\newcommand{\1}{\mathbbm{1}}
\renewcommand{\P}{\mathbb{P}}

\newcommand{\scott}[1]{{\color{red}{#1}}}
\newcommand{\omer}[1]{{\color{cyan}{#1}}}
\newcommand{\nina}[1]{{\color{blue}{#1}}}

\title{Scaling limits of the Schelling model}

\author{Nina Holden\footnote{Massachusetts Institute of Technology, Cambridge, MA, ninah@math.mit.edu} 
	\and  
	Scott Sheffield\footnote{Massachusetts Institute of Technology, Cambridge, MA, sheffield@math.mit.edu}
}

\date{\today}
\maketitle

\begin{abstract}
The Schelling model, introduced by Schelling in 1969 as a model for residential segregation in cities, describes how populations of
	multiple types self-organize to form homogeneous clusters of one type.
	In this model, vertices in an $N$-dimensional lattice are initially assigned types randomly. As time
	evolves, the type at a vertex $v$ has a tendency to be replaced with
	the most common type within distance $w$ of $v$.  We present the first
	mathematical description of the dynamical scaling limit of this model
	as $w$ tends to infinity and the lattice is correspondingly rescaled.
	We do this by deriving an integro-differential equation for the limiting
	Schelling dynamics and proving almost sure existence and uniqueness of
	the solutions when the initial conditions are described by white
	noise. The evolving fields are in some sense very ``rough'' but we are able to make rigorous sense
	of the evolution. In a key lemma, we show that for certain Gaussian fields
	$h$, the supremum of the occupation density of $h-\phi$ at zero (taken
	over all $1$-Lipschitz functions $\phi$) is almost surely finite, thereby extending a result of Bass and Burdzy. In the
	one dimensional case, we also describe the scaling limit of the
	limiting clusters obtained at time infinity, thereby resolving a conjecture of Brandt, Immorlica, Kamath, and Kleinberg.
\end{abstract}
\begin{figure}[h]
	\begin{center}
	\includegraphics[scale=1.6]{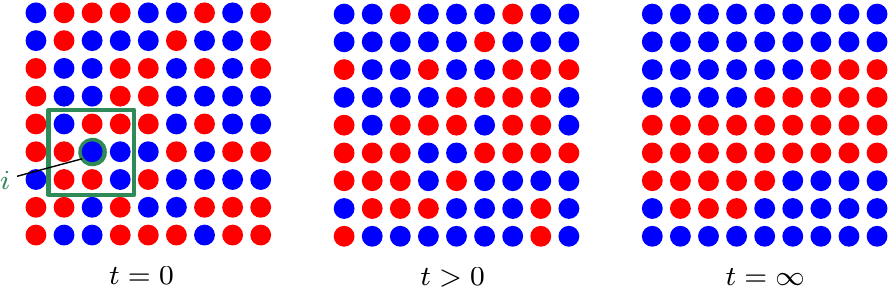}
	\caption{The Schelling model on the two-dimensional torus with two types (red and blue). At time $t=0$ (left) the types of the nodes are chosen uniformly and independently at random. Each node $i$ is associated with a neighborhood (shown in green) and an independent rate one Poisson clock. Every time the clock of a node rings, it updates its type to the most common type in its neighborhood. Eventually we reach a stable configuration (right). }
\end{center}
\end{figure}
\newpage
\tableofcontents

\section{Introduction}

The Schelling model \cite{schel69,schelling71,schel-book} was initially introduced to explain residential segregation in cities, and is one of the earliest and most influential agent-based models studied by economists. Variants of the model have been studied by thousands of researchers within a number of disciplines, e.g.\ social sciences, statistical mechanics, evolutionary game theory, and computer science, see the works referenced below and \cite{Clark91,pw01,lj03,vk07,Pancs07,ss07,dcm08,o08,Gerhold08,gvn09,gblj09,svw09} for a small and incomplete selection of these works. Until recently \cite{schel-klein1,bel-schelling-14,schel-klein2,bel-schelling-tipping,bel-schelling-2d,bel-schelling-minority} most analysis of the model was based either on simulation, non-rigorous analysis or so called ``perturbed'' versions of the model (discussed below).  We will discuss the Schelling model history and give an informal overview of the paper in Sections~\ref{subsec::history} and~\ref{subsec::overview}, and we provide a precise definition of the model in Section~\ref{sec:intro1}.

\subsection{History} \label{subsec::history}
In the original formulation of the model, individuals of two ``types'' occupy a subset of the nodes of a graph, and at random times an individual moves to a free (i.e., unoccupied) position in the graph. Individuals move to locations at which they will have more neighbors of  their own type. 
Schelling showed, using simulations he implemented manually with pennies and dimes on a ruled sheet of paper \cite{schel-book}, that segregation occurs even if the agents have only a weak preference for being in regions with a high density of their own type. His findings have been confirmed later by a huge number of simulations of other researchers, and his findings have strongly influenced debates about the causes of residential segregation \cite{cf08}. The introduction of the model also contributed to Schelling winning the Nobel Memorial Prize in Economics in 2005 \cite{N05}.

The first mathematically rigorous results on the model considered a variant where the dynamics describing the transition between states were ``perturbed'' in the sense that agents have a small probability $p>0$ of acting against their preference \cite{young-schel,zhang-schel}. The perturbed model was analyzed by studying the stationary distribution of the associated Markov chain. In particular, the stochastically stable states, which are states whose stationary
probability is bounded away from zero when $p\rta 0$, were studied. The stochastically stable states are proved to be those which minimize the length of the interface between the two types of individuals (i.e., the number of neighboring pairs containing one individual of each type) so that using the terminology of statistical physics the stochastically stable states correspond to Ising model ground states.

One interesting property of the unperturbed model is that it can be shown to {\em stabilize} almost surely in finite time, i.e., after some point in time the agents stop moving. It has been argued (see e.g.\ \cite{schel-klein2}) that these limiting stable configurations have at least some properties in common with the segregation patterns observed in real cities, e.g.\ since they tend to appear more irregular than the stochastically stable states of the unperturbed model. We will not address real world segregation patterns in this paper.

The first mathematically rigorous analysis of the unperturbed model by Brandt, Immorlica, Kamath, and Kleinberg \cite{schel-klein1} concerns a version of the model on the one-dimensional torus where the neighborhood of a node is given by its nearest $2w+1$ neighbors (including itself) for some $w\in\N$. They prove that the configuration of types in the stable limiting configuration consists of intervals of length at most polynomial in $w$. In \cite{schel-klein2} the authors consider a two-dimensional Schelling model where an agent only changes location if the fraction of his neighbors having the same type as himself is $<1/2-\ep$ for some $\ep\ll 1$. They prove that the expected diameter of the segregated region containing the origin in the final configuration grows at least exponentially in $w^2$. See also \cite{bel-schelling-14,bel-schelling-tipping,bel-schelling-2d,bel-schelling-minority} for recent rigorous results on the unperturbed Schelling model in one, two, and three dimensions. The models studied in these papers have a more general initial configuration of types and more general tolerance parameters than the models in \cite{schel-klein1,schel-klein2} and the current paper. The authors are particularly interested in parameter values which lead to either very high degree of segregation, total takeover of one type, or almost no changes relative to the initial configuration.

We will focus on a variant of the model which we call the {\em single-site-update} Schelling model, briefly explained later in this paragraph, in which individual vertices are updated one at a time. This variant of the model is also the one considered in \cite{bel-schelling-tipping,bel-schelling-2d,schel-klein2}. Some of the other papers mentioned above consider the {\em pair-swapping} Schelling model, where two individuals of different types will swap positions with each other if this leads to both nodes having more neighbors of their own type. In both variants of the model all the nodes of the considered graph are occupied, i.e., there are no free or unoccupied nodes. In the single-site-update Schelling model unsatisfied individuals {\em change} types, instead of swapping with each other. That is, one picks a random {\em individual} and allows that individual to change type if desired, instead of picking a {\em pair} of individuals and asking them to swap locations if desired.  In the single-site-update version, the number of vertices of a given type is not constant. Instead, one imagines that there is a larger ``outside world'' beyond the graph being considered, and that when a vertex changes type, it corresponds to an individual within the configuration swapping location with someone from the ``outside world.'' The single-site-update evolution is essentially equivalent to the pair-swapping model evolution in a setting with an  ``outside world'' region (disconnected from the main lattice graph under consideration) that contains a large number of unsatisfied individuals of each type.  We will focus on the single-site-update variant in this paper because it is cleaner mathematically (one only has to deal with one individual at a time when making updates), but we will explain at the end of Section \ref{sec:mainres} that our first main result (Theorem \ref{thm2}) also holds in the pair-swapping setting.

\subsection{Overview} \label{subsec::overview}
We study an unperturbed, single-site-update version of the Schelling model on an $N$-dimensional lattice with $M\geq 2$ different types. A node is unsatisfied if the most common type in its neighborhood differs from its current type, and the size of the neighborhood is described by a constant $w\in\N$. Adapting the vocabulary of majority dynamics, we call the type of the node the \emph{opinion} of the node. At time zero, each node is assigned an opinion uniformly and independently at random.  Each node is associated with an independent Poisson clock, and every time the clock of a node rings it updates its opinion to the most common opinion in its neighborhood. In other words, a node changes its opinion when its Poisson clock rings if and only if the node is currently not  satisfied.

We prove a dynamical scaling limit result for the early phase of the Schelling dynamics for any $N\in\N$ and $M\in\{2,3,\dots \}$. We define a vector-valued function $Y^w$ called the normalized bias function; as explained below, $Y^w(z)$ is the vector whose components are (a normalizing constant times) the $M$ opinion densities (minus their expectations) in the radius $w$ box centered at $z$. We prove that $Y^w$ converges in the scaling limit to the solution $Y$ of a differential equation (more precisely, an {\em integro-differential equation}) with Gaussian initial data. We call the function $Y$ the continuum bias function, and we call the associated initial value problem the continuum Schelling model. See Theorem \ref{thm2} and Proposition \ref{prop18}. Solutions of the differential equation are not unique for all choices of initial data. However, we prove existence and uniqueness of solutions for Gaussian initial data.

The basic idea of the argument is to note that even though the initial Gaussian normalized bias function is very rough, the {\em change} in its value from the initial time to a finite later time is a.s.\ a (random) Lipschitz function. Focusing on this difference, we are left with a random ODE in the (more regular) space of Lipschitz functions. To establish existence and uniqueness of the evolution within this space, we wish to apply a variant of the Picard-Lindel\"of theorem, but doing so requires some sort of continuity in the corresponding ODE, which requires us to understand whether there are situations where two coordinates of the normalized bias function are {\em very} close on a large set, so that even small perturbations lead to big changes in the direction the functions are evolving.  It turns out that one can show these situations are unlikely by establishing control on the maximal {\em occupation kernel} corresponding to the intersection of the initial Gaussian function with a Lipschitz function (where the maximum is taken over all functions with Lipschitz norm bounded by a fixed constant). Analogous results for Brownian local time, established by Bass and Burdzy in \cite{bb01}, turn out to be close to what we need, and we are able to adapt the techniques of \cite{bb01} to our higher dimensional setting with a few modifications.

In the special case when $N=1$ and $M=2$ we also prove a scaling limit result for the final configuration of opinions. This confirms a conjecture in \cite{schel-klein1}. More precisely, we show in Theorem \ref{thm1} below that the law on subsets of $\Z$ describing the limiting opinion of each node converges upon rescaling by $w$. The theorem says that if we study the model on the rescaled lattice $w^{-1}\Z$ and let $A\subset w^{-1}\Z$ be the set of nodes for which the limiting opinion is 1, then $A$ converges in law as $w\rta\infty$, viewed as an element in the space of closed subsets of $\R$ equipped with the Hausdorff distance. The scaling limit result is proved by studying the long-time behavior of the continuum bias function $Y$. This is one of the more technically interesting parts of the paper, as a number of tricks are used to rule out anomalous limiting behavior. Our conclusion is that in the limit one obtains a random collection of homogeneous neighborhoods, each of width strictly greater than one. The idea of the proof is to show that if this does not occur, then it will occur if we make a slight perturbation to the initial data, and a delicate analysis of the differential equation is required to show that this is indeed the case. 

After we describe the continuum dynamics and (for $N=1,M=2$) its limiting behavior, we will need to do some additional work to make the connection with the discrete model. We consider two phases separately: First we study the model up to time $Cw^{-N/2}$ for $C\gg1$, and then (for $N=1,M=2$) we consider times larger than $Cw^{-1/2}$. The first phase of the evolution is governed by the differential equation, and we prove that the differential equation predicts the evolution of the discrete model well by bounding the error which accumulates during a short interval of length $\Delta t$. The second phase starts when the solution of the differential equation has almost reached its limiting state with homogeneous intervals. We show that with high probability the homogeneous intervals observed at time $Cw^{-1/2}$ will continue to exist until all nodes have reached their final opinion. Nodes near the boundary between two intervals at time $Cw^{-1/2}$ have approximately half of their neighbors of each opinion, which makes it hard to control the evolution of the bias for these nodes; however, we do manage to show that with high probability each interval does not shrink too much before all nodes have reached their final opinion.

In the remainder of the introduction we will give a precise definition of the Schelling model and state our main results. In Section \ref{sec:welldef} we show existence and uniqueness of solutions of the continuum Schelling model by using results from Section \ref{sec:suploctime}, and in Section \ref{sec:propertiescontinuum} we prove that for the one-dimensional model with $M=2$ the sign of the solution converges almost surely at almost every point. In Section \ref{sec:discrete1} we prove that the continuum Schelling model describes the discrete Schelling model well for small times and large $w$. In Section \ref{sec:discrete2} we conclude the proof of the scaling limit result for the one-dimensional Schelling model. We also prove (proceeding similarly as in \cite{tt14}) that the opinion of each node converges a.s.\ for any $N\in\N$ and $M\geq 2$, and we include a lemma which might be related to the typical cluster size for the limiting opinions in higher dimensions. We conclude the paper with a list of open questions in Section \ref{sec:open}.

\subsection{The Schelling model}
\label{sec:intro1}
We will start by defining the Schelling model on a general simple graph $G$ with vertex set $V(G)$ and edge set $E(G)$. The vertex set $V(G)$ may be infinite, but we assume $G$ is locally finite. Let $M\in\{2,3,4,\dots \}$. Each $i\in V(G)$ is associated with an opinion in $\{1,...,M\}$ and an independent unit rate Poisson clock, i.e., each node is associated with a clock such that the times between two consecutive rings of the clock are distributed as i.i.d.\ unit rate exponential random variables. Let $X(i,t)\in\{1,\dots,M\}$ denote the opinion of node $i$ at time $t\geq 0$, and let $\frk N(i):=\{j\in V(G)\,:\, (i,j)\in E(G) \}\cup\{i\}$ be the neighborhood of $i$. Every time the Poisson clock of a node rings, the node updates its opinion according to the following rules:
\begin{enumerate}
\item[(i)] The node chooses the most common opinion in its neighborhood if this is unique. In other words, we set $X(i,t)=m$ for $m\in\{1,\dots,M \}$ if for all $m'\in\{1,\dots,M \}\setminus\{m \}$ we have $|\{ j\in\mcl N(i)\,:\,X(j,t)=m \}|> |\{ j\in\mcl N(i)\,:\,X(j,t)=m' \}|$.
\item[(ii)] If there is a draw between different opinions, and the current opinion of the node is one of these opinions, the node keeps its current opinion. 
\item[(iii)] If there is a draw between different opinions, and none of these opinions are equal to the current opinion of the node, the new opinion of the node will be chosen uniformly at random from the set of most common opinions in its neighborhood.
\end{enumerate}
Note that almost surely no two Poisson clocks will ring simultaneously, even when $|V(G)|$ is infinite. Furthermore, one can show that for graphs with bounded degree, and for any fixed vertex $i\in V(G)$ and a fixed time $t\in\R_+$, the set of vertices whose initial opinion may have influenced the opinion of $i$ at time $t$, given the times at which the various Poisson clocks were ringing, is almost surely finite, see e.g.\ \cite[Claim 3.5]{tt14}. These two observations imply that the configuration at each time $t$ is a.s.\ determined by the initial configuration and the ring times, along with knowledge about how the draws described in (iii) are resolved.

In this paper, we will consider the Schelling model on a lattice, and we consider scaling limits as the neighborhood size tends to infinity, but we will be fairly flexible about the ``shape'' (circle, square, etc.) of the neighborhood. Let $N\in\N$ be a parameter describing the dimension of the graph, and let $w\in\N$ be a parameter we call the window size. Let $\mcl N=(-1,1)^N$ (or, alternatively, let $\mcl N\subset(-1,1)^N$ be a sphere or some other shape; precise conditions on $\mcl N$ appear below). Define the neighborhood $\frk N(i)$ of an element $i$ of $\Z^N$ by 
\eqbn
\frk N(i) = \{j\in\Z^N\,:\,w^{-1}(j-i)\in\mcl N \}.
\label{eq78}
\eqen 
When the dimension $N$ is greater than one, we will mostly work on a torus whose size is a large constant times the neighborhood size; precisely, for a fixed constant $R\in\{3,4,\dots \}$ and defining the one-dimensional torus $\frk S=\Z / (Rw \Z)$ we will work on the torus $\frk S^N$. We use the torus instead of $\Z^N$ mainly because we do not establish existence and uniqueness of solutions of a particular differential equation for the model on $\Z^N$ when $N\geq 2$. In the remainder of this section we describe the Schelling model in terms of $\frk S^N$ rather than $\BB Z^N$, but we obtain the model on $\BB Z^N$ by repeating the description with $\BB Z^N$ instead of $\frk S^N$. To simplify notation when considering the Schelling model on $\frk S^N$, we identify an element $i\in\BB Z^N$ with its equivalence class in $\frk S^N$.

Assume the initial opinions of the nodes are i.i.d.\ random variables satisfying $\BB P[X(i,0)=m]=M^{-1}$ for all $i\in\frk S^N$ and $m\in\{1,...,M\}$. Throughout the paper we let $\mcl R$ denote the set of rings of the Poisson clocks
\eqb
\mcl R = \{(i,t)\in\frk S^N\times\R_+ \,:\, \text{the clock of node $i\in\frk S^N$ rings at time $t$}\}.
\label{eq47}
\eqe

We define the \emph{bias} of node $i\in\frk S^N$ towards opinion $m\in\{1,...,M\}$ at time $t>0$, to be the sum $\sum_{j\in \frk N(i)} \bd 1_{X(j,t)=m}$. When the clock of a node rings, the node updates its opinion to the opinion towards which it has the strongest bias, with draws resolved as described in (ii)-(iii) above. We say that $i\in\frk S^N$ agrees with the most common opinion in its neighborhood at time $t$ if $\sum_{j\in \frk N(i)} \bd 1_{X(j,t)=X(i,t)}\geq \sum_{j\in \frk N(i)} \bd 1_{X(j,t)=m}$ for any $m=\{1,...,M\}$. In other words, $i$ agrees with the most common opinion in its neighborhood if and only if it would not update its opinion if its Poisson clock were ringing.

In the description of the Schelling model above we considered (for simplicity of the description) the neighborhood $\mcl N=(-1,1)^N$, but our results are proved for more general neighborhoods, since the more general case is not significantly more difficult to analyze. Unless otherwise stated, we will assume that $\mcl N\subset(-1,1)^N$ is an arbitrary open set containing 0, such that $\partial \mcl N$ has upper Minkowski dimension strictly smaller than $N$, and such that the following technical condition is satisfied. If $x_0$ is a unit vector in an arbitrary direction, $\lambda$ denotes Lebesgue measure, and  we let $\mcl N+s x_0$ denote the set $\{ x+sx_0\,:\,x\in\mcl N \}$ for some $s\in\R$, then
\eqb
\liminf_{t\rta 0^+} \frac{1}{t} \lambda\left( (\mcl N+tx_0) \setminus \bigcup_{s\leq 0} (\mcl N +sx_0) \right) >0.
\label{eq65}
\eqe 
The condition \eqref{eq65} will be used in the proof of Lemmas \ref{prop14} and \ref{prop25}. It is obviously satisfied by most of the smooth-boundary regions one would be inclined to consider.  We may for example let $\mcl N=\mcl N_p$, where $\mcl N_p:=\{x\in\R^N\,:\,\|x\|_{p}<1 \}$ is a metric ball for the $\ell^p$ norm for some $p\in[1,\infty]$.

Finally we remark that the first main result of the next section holds for the general $N$-dimensional Schelling model on the torus with $M\in\{2,3,\dots\}$ opinions, while the second main result holds for the torus or the real line for $N=1$, $M=2$, and $\mcl N=\mcl N_\infty$.

\subsection{Main results}
\label{sec:mainres}

Our first main result is a dynamical scaling limit result for the Schelling model. We prove that a function $Y^w$ describing the opinions of the nodes in the early phase (times up to order $w^{-N/2}$) of the Schelling model on the torus, converges in law as $w\rta\infty$. Denote the $N$-dimensional torus of side length $R$ by $\mcl S^N=\mcl S\times\dots\times\mcl S$. Let $\mcl C(\mcl S^N\times\R_+)$ denote the set of continuous real-valued functions on $\mcl S^N\times\R_+$, and let $\mcl C_M(\mcl S^N\times\R_+)$ denote the set of functions which can be written in the form $f=(f_1,\dots,f_M)$ for $f_m\in\mcl C(\mcl S^N\times \R_+)$ and $m=1,\dots,M$. Equip $\mcl C_M(\mcl S^N\times\R_+)$ with the topology of uniform convergence on compact sets. Define the unscaled bias function $\mcl Y_m$ by 
\eqbn
\mcl Y_m(i,t) = \sum_{j\in\frk N(i)}\left(\1_{X(j,t)=m} - \frac{1}{M}\right).
\eqen
Then define the (normalized) bias function $Y^w=(Y^{w}_1,\dots, Y^{w}_M)\in\mcl C_M(\mcl S^N\times\R_+)$ by
\eqb 
Y^{w}_m(x,t) := \frac{1}{w^{N/2}} \mcl Y_m( xw,tw^{-N/2} )
\qquad m\in\{1,...,M\},\quad x\in w^{-1}\Z^N,\quad t\geq 0, 
\label{eq24} 
\eqe
and for $x\not\in w^{-1}\Z^N$ let $Y^{w}_m(x,t)$ be a weighted average of the $\leq 2^N$ points $\wt x\in w^{-1}\Z^N$ which satisfy $\|x-\wt x\|_{\infty}<w^{-1}$
\eqb
Y^{w}_m(x,t) = \sum_{\wt x\in w^{-1}\Z^N\,:\,\|x-\wt x\|_{\infty}<w^{-1}} \left(\prod_{k=1}^{N} (1-w|x_k-\wt x_k|)\right) Y^{w}_m(\wt x,t).
\label{eq60}
\eqe
See Figure \ref{fig-1d}. Note that $Y^w$ encodes the bias of each node towards each opinion $1,\dots,M$. Also note that when defining $Y^w$ we do not only scale space; we also scale time by $w^{-N/2}$. This implies that the convergence result of the following theorem only describes the evolution of the bias in the very beginning (more precisely, up to times of order $O(w^{-N/2})$) of the Schelling model.
\begin{figure}[ht!]
	\begin{center}
		\includegraphics[scale=0.52]{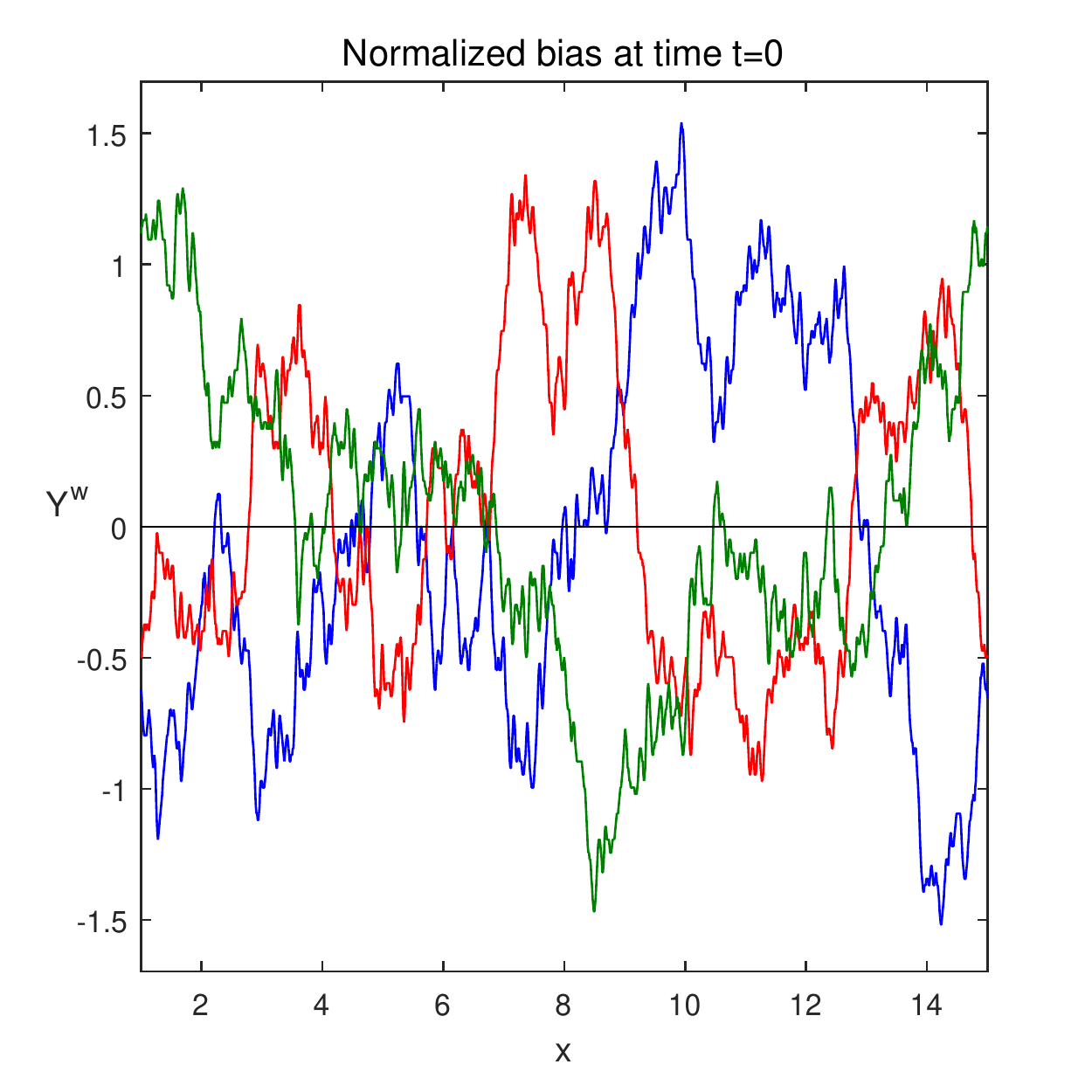} 
		\includegraphics[scale=0.52]{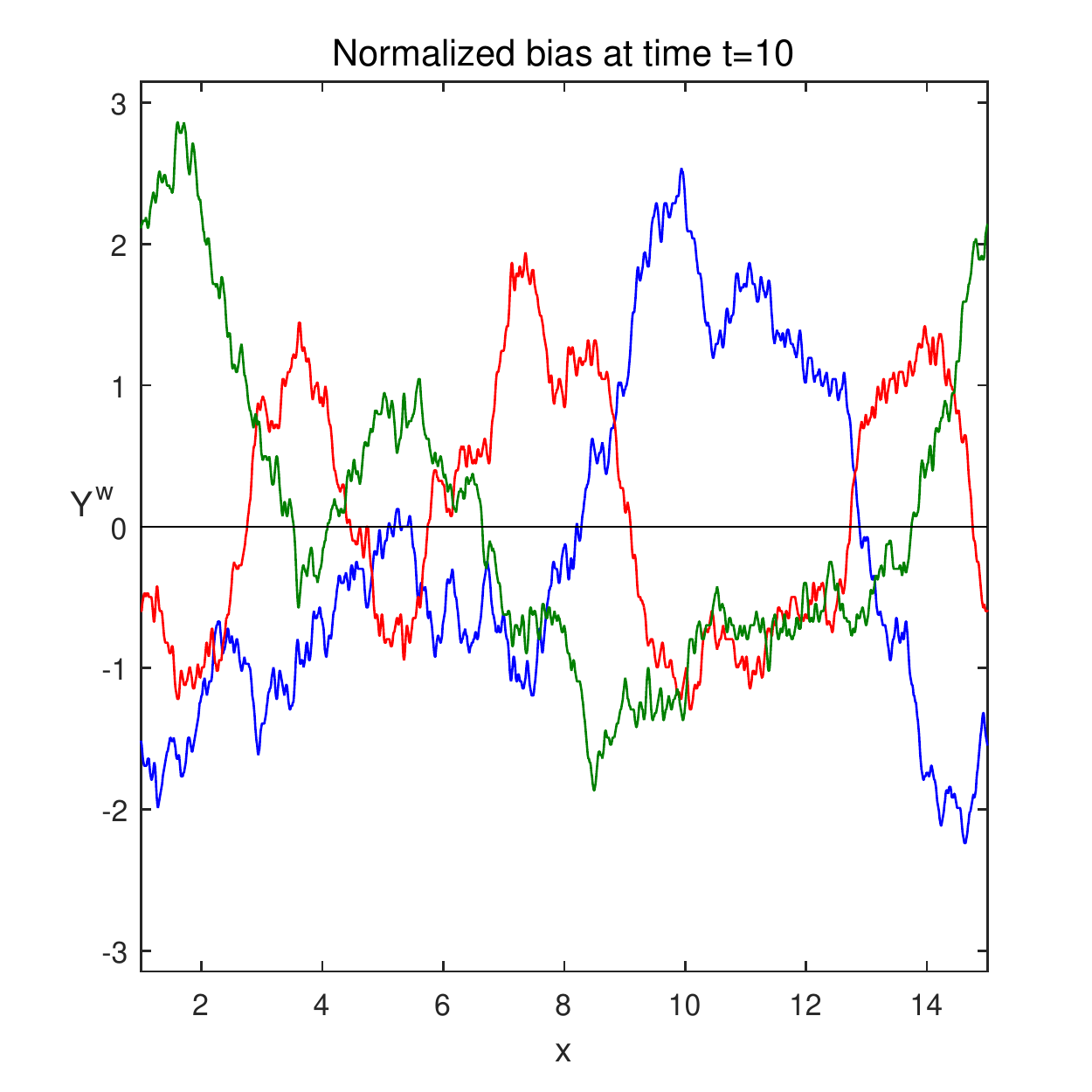} 
		\includegraphics[scale=0.52]{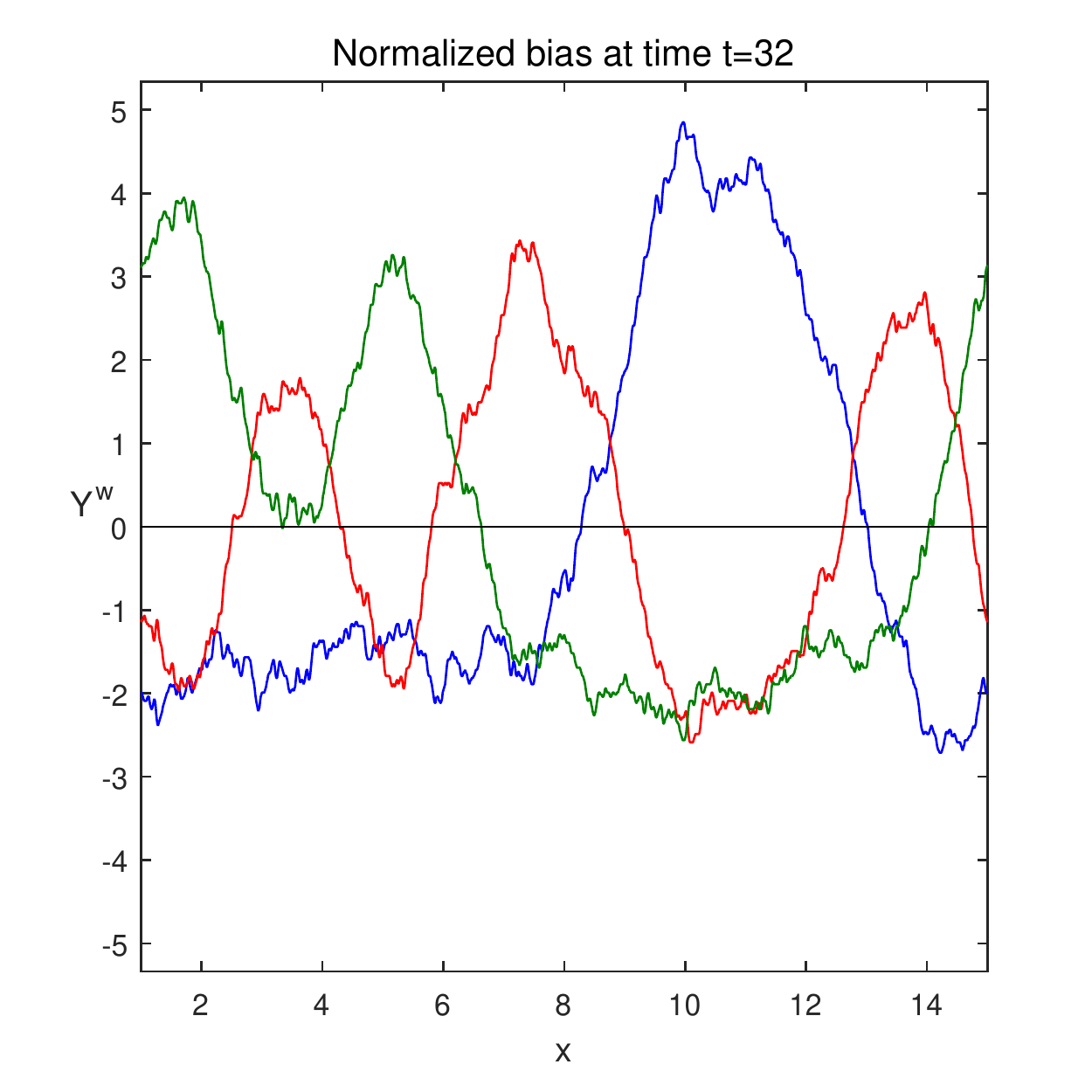} 
		\includegraphics[scale=0.52]{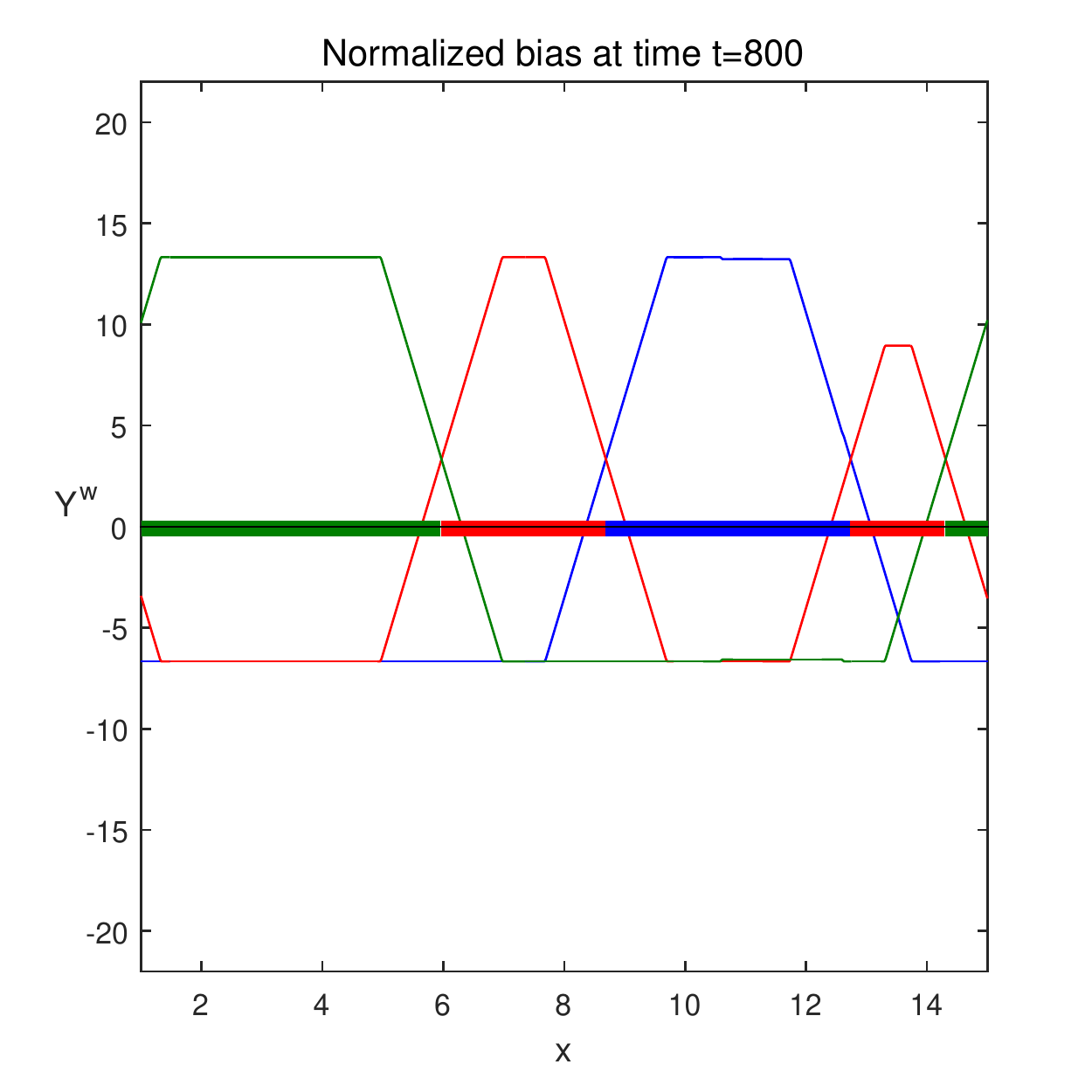} 
		\caption{The graphs show the bias function $Y^w$ at four different times for the Schelling model on the torus of dimension $N=1$ with $M=3$ opinions. Each color in the figure corresponds to one of the $M$ opinions. The initial data of $Y^w$ (upper left figure) converges in law to the Gaussian field $B$ defined in Section \ref{sec:cont} when $w\rta\infty$. The evolution of $Y^w$ can be well approximated by the differential equation \eqref{eq1} on compact time intervals, see Theorem \ref{thm2} and Proposition \ref{prop18}. When we have reached the final configuration of opinions, each function $Y^w_m$, $m=1,2,3$, is piecewise linear with slopes $\pm 2w^{1/2}$ and 0, and values in the interval $[-2w^{1/2}\frac{1}{M},2w^{1/2}(1-\frac{1}{M})]$. Regions in which $Y_m^w$ equals $2w^{1/2}(1-\frac{1}{M})$ correspond to long intervals where the nodes have limiting opinion $m$. For the continuum analog $Y$ of $Y^w$, each function $Y_m$ converges to $\pm\infty$ at almost every point. The thick line at the $x$ axis in the lower right figure represents the limiting opinion of the nodes. The plots are made with window size $w=100$ and torus width $R=14$. } 
		\label{fig-1d}
	\end{center}
\end{figure}
\begin{figure}[ht!]
	\begin{center}
		\includegraphics[scale=0.52]{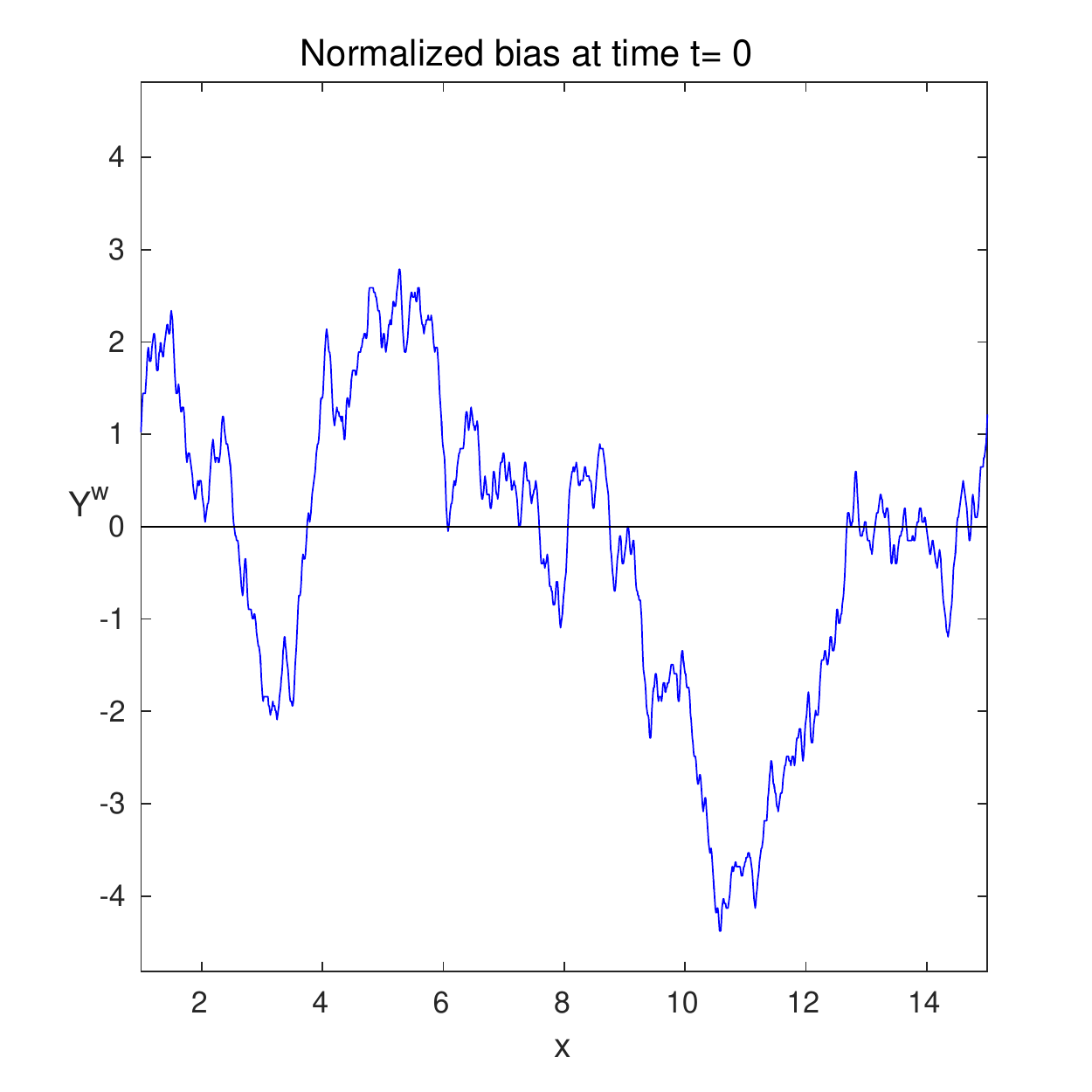} 
		\includegraphics[scale=0.52]{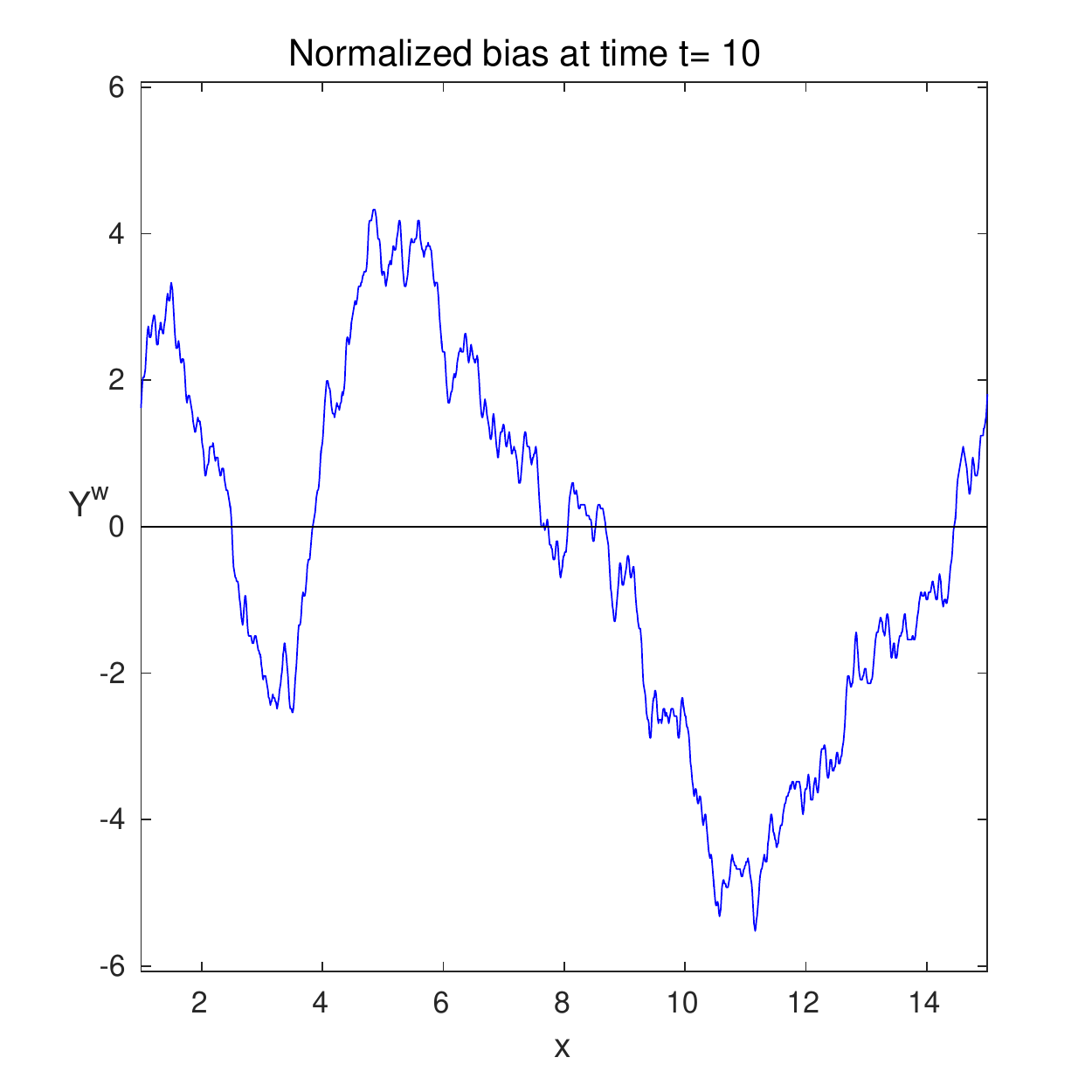} 
		\includegraphics[scale=0.52]{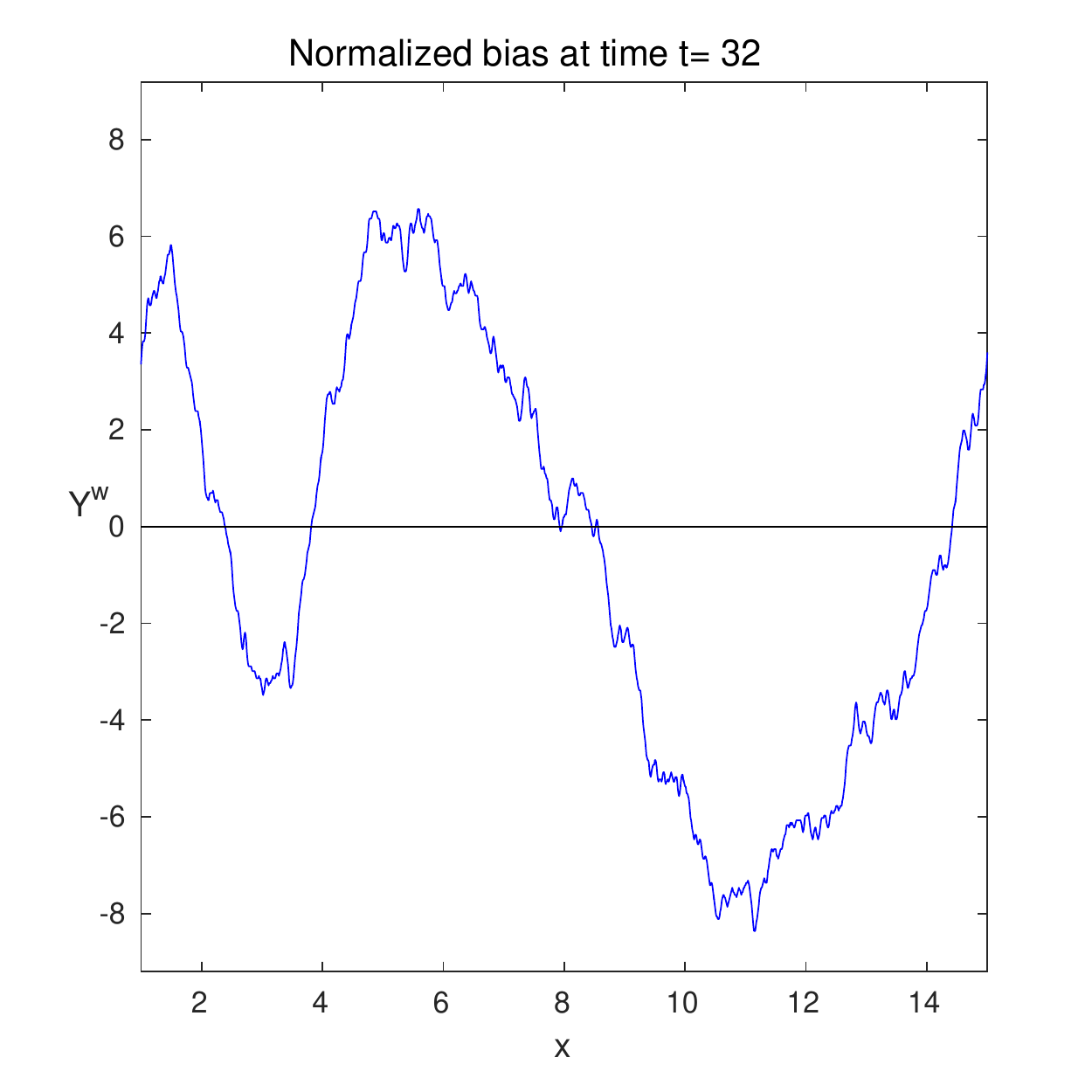} 
		\includegraphics[scale=0.52]{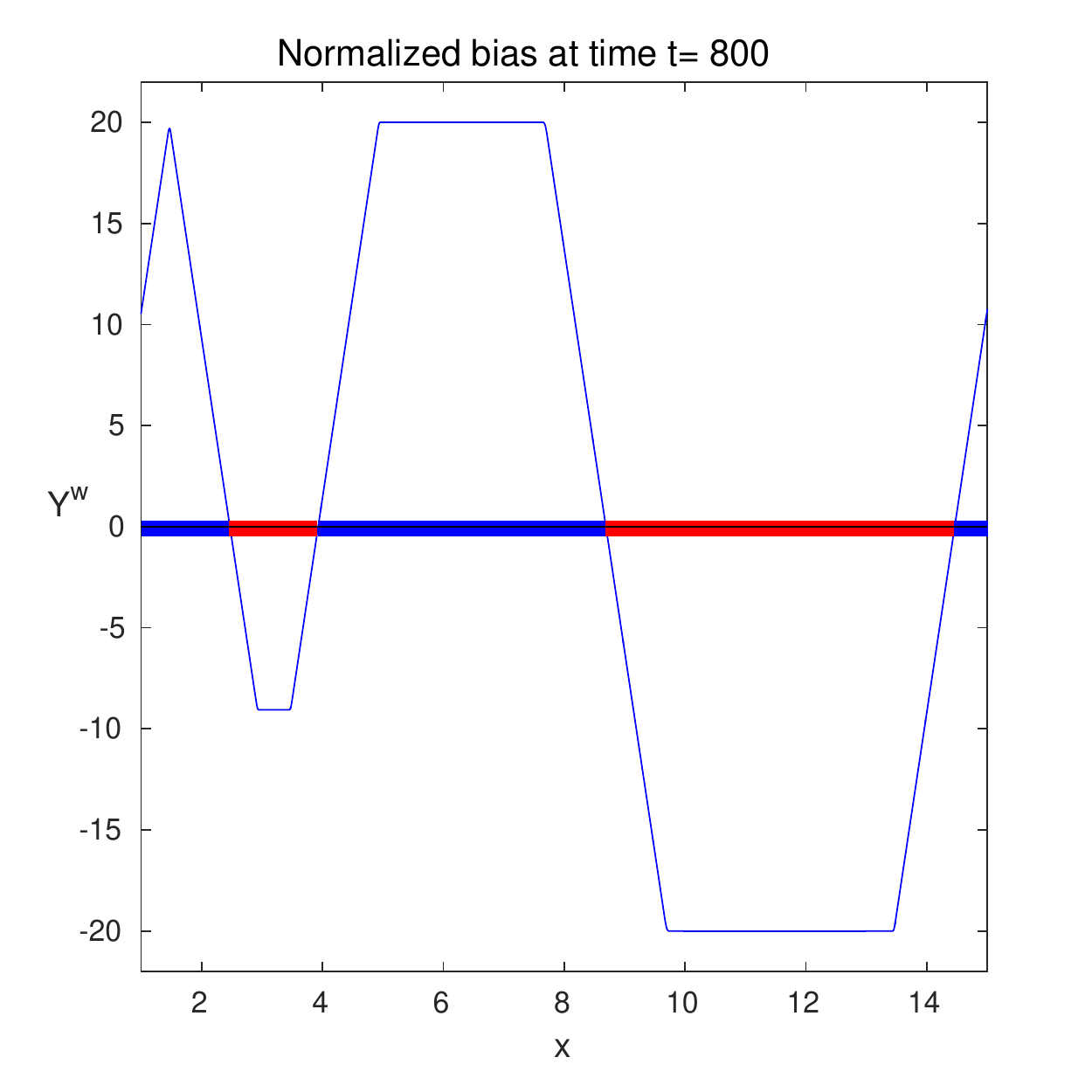} 
		\caption{The graphs show the function $\wh Y^w:=Y^w_1-Y^w_2$ for the Schelling model on the torus of dimension $N=1$ with $M=2$ opinions, which evolves approximately as described by \eqref{eq71}. The plots are made with window size $w=100$ and torus width $R=14$.} 
		\label{fig-1d-M2}
	\end{center}
\end{figure}
\begin{thm}
	In the setting described above, $Y^w$ converges in law in $\mcl C_M(\mcl S^N\times\R_+)$ to a random function $Y$ as $w\rta\infty$.
	\label{thm2}
\end{thm}
The theorem is an immediate consequence of Proposition \ref{prop18} in Section \ref{sec:discrete1}, which identifies $Y$ as the solution of a particular differential equation with Gaussian initial data.

Our second main result is a scaling limit result for the final configuration of opinions in the one-dimensional Schelling model. We consider the model on either the torus $\frk S$ or on $\Z$, and have $M=2$ opinions and neighborhood $\mcl N=\mcl N_\infty$. For $R\in\{ 3,4,\dots \}$ let $\mcl S$ be the torus of width $R$, i.e., $\mcl S=\R/\sim$, where $\sim$ is the equivalence relation on $\R$ defined by $x\sim y$ iff $x-y$ is an integer multiple of $R$. By \cite{tt14} (see Proposition \ref{prop30} below for the analogous result for general $M$) the opinion of each node converges almost surely as time goes to infinity, hence each node is associated with a unique opinion in $\{1,2\}$ describing its limiting opinion. 

For $V=\mcl S$ and $V=\R$ define 
\eqbn
\mcl D(V):=\{\mcl A=(A_1,A_2)\,:\,A_m\subset V \text{\,\,is closed for }m=1,2 \}.
\eqen
Equip $\mcl D(V)$ with the topology of convergence of $A_1$ and $A_2$ for the Hausdorff distance on compact sets. The appropriately normalized limiting distribution of opinions in the Schelling model, is a random variable in $\mcl D(V)$. The following theorem says that this random variable converges in law in $\mcl D(V)$ as the window size $w$ converges to $\infty$. In the theorem below we identify $\frk S$ and $\mcl S$ with $\{ 0,\dots,Rw-1 \}$ and $[0,R)$, respectively.
\begin{thm}
	Let $\mcl V=\BB Z$ and $V=\R$, or let $\mcl V=\frk S$ and $V=\mcl S$. Consider the one-dimensional Schelling model on $\mcl V$ as described in Section \ref{sec:intro1} with window size $w\in\BB N$, $M=2$ opinions, and $\mcl N=\mcl N_\infty$. Define $\mcl A^w\in\mcl D(V)$ by
	\eqbn
	\mcl A^w:=(A^w_1,A^w_2),\qquad
	A^w_m := \{jw^{-1}\in V\,:\,j\in\mcl V,\, \lim_{t\rta\infty}X(j,t)=m\}\text{\,\,for\,\,}m=1,2.
	\eqen
	Then $A^w_1\cup A_2^w=\{jw^{-1}\,:\,j\in \mcl V \} $ a.s., and $\mcl A^w$ converges in law as a random variable in $\mcl D(V)$ to a limiting random variable $\mcl A=(A_1,A_2)$. The sets $A_1$ and $A_2$ have disjoint interior and union $V$ a.s., and each set $A_m$ for $m=1,2$ is almost surely the union of at most countably many closed intervals each of length strictly larger than 1.
	\label{thm1}
\end{thm}
When we prove the theorem in Section \ref{sec:discrete2} we will describe the limiting random variable $\mcl A$ in terms of the solution $Y$ of the initial value problem mentioned above. 
We will not describe the law of this random variable further, but we remark that if $I$ is defined to be the maximal interval satisfying either $0\in I\subset A_1$ or $0\in I\subset A_2$, then the length of $I$ decays at least exponentially; this holds by Lemma \ref{prop21}, and since the event considered in this lemma holds independently and with uniformly positive probability on each interval $[10k,10k+5]$, $k\in\Z$ (see the proof of \cite[Theorem 1]{schel-klein1} for a similar argument).

\begin{figure}[ht!]
	\begin{center}
		\includegraphics[scale=1.1]{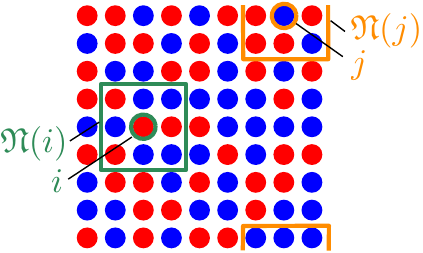} \quad
		\includegraphics[scale=1.1]{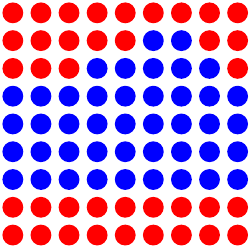} \qquad
		\includegraphics[scale=1.1]{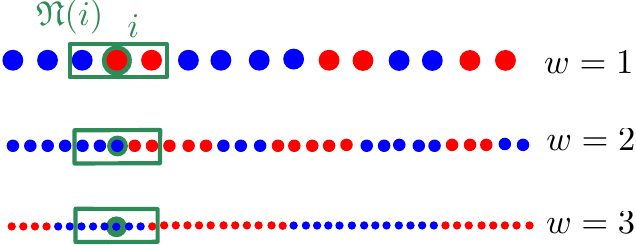} 
		\caption{Left: initial configuration of opinions in the Schelling model on the torus for $N=2$, $M=2$, $R=9$, and $w=1$. Middle: one possible final configuration of opinions with initial data as on the left figure. Observe that all nodes agree with the most common opinion in their neighborhood. Right: final configuration of opinions in the Schelling model on the torus for $N=1$, $M=2$, $R=15$, and $w=1,2,3$.} 
		\label{fig-limit0}
	\end{center}
\end{figure}

\begin{figure}[ht!]
	\begin{center}
		\includegraphics[scale=0.161]{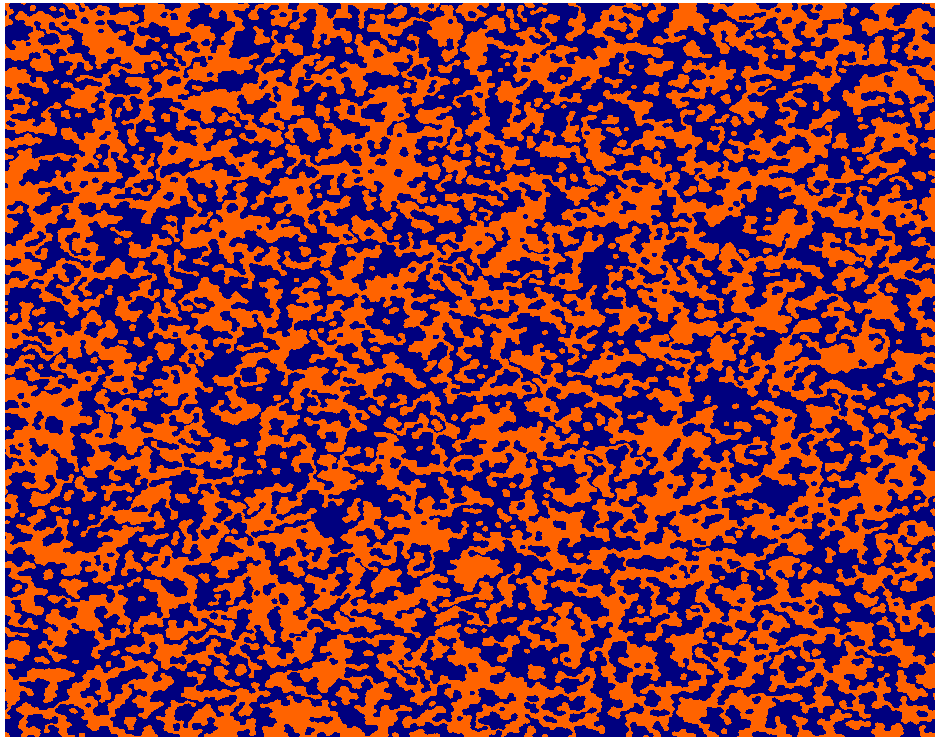} 
		\includegraphics[scale=0.165]{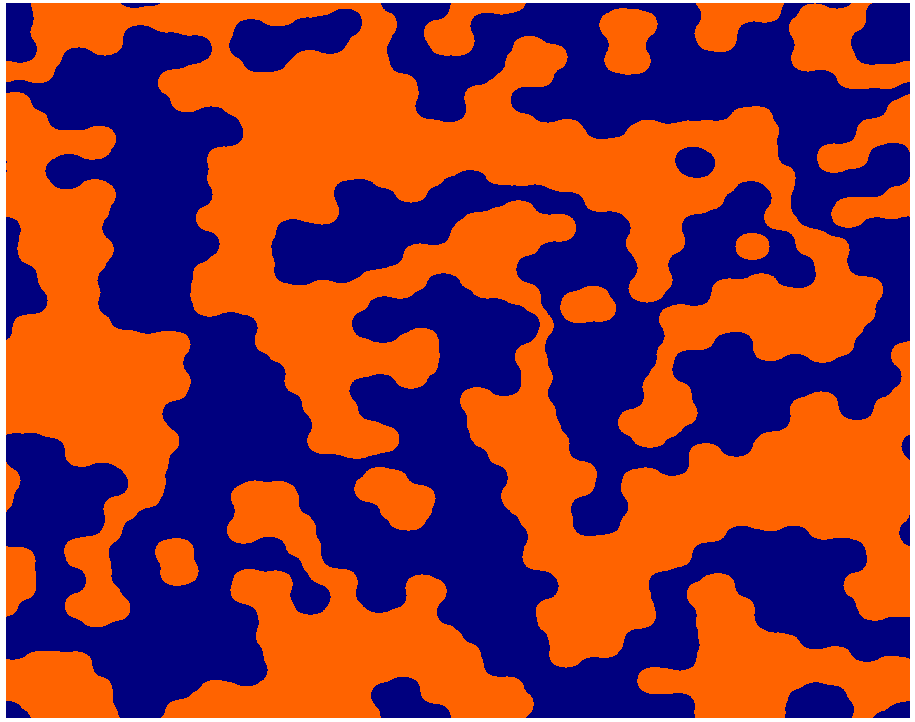} 
		\includegraphics[scale=0.161]{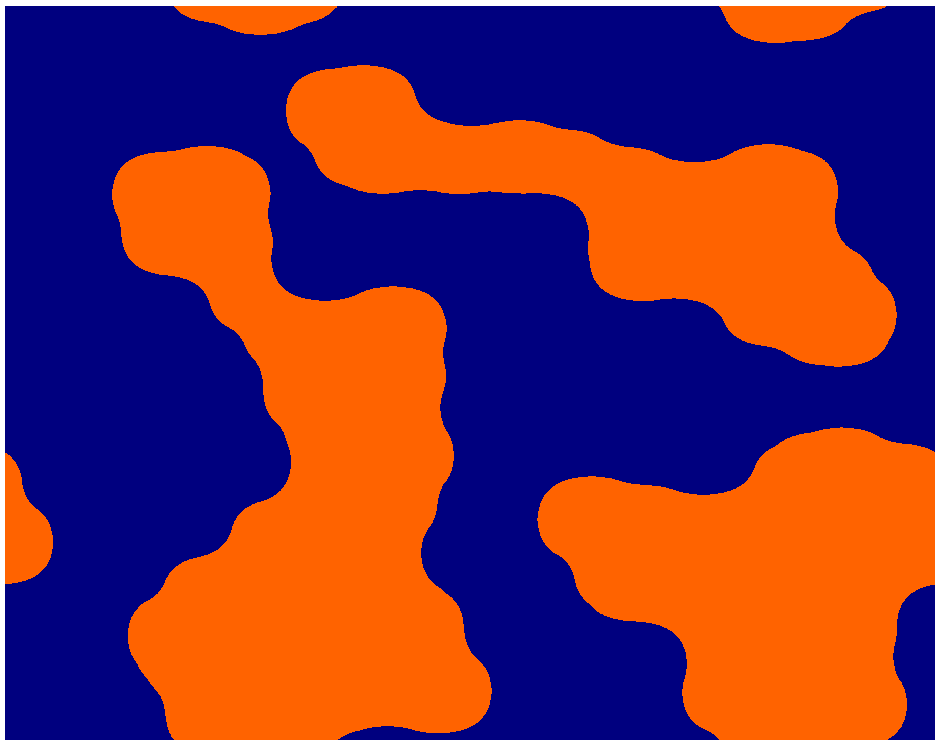} 
		\caption{Final configuration of opinions in the Schelling model on the torus for $N=2$, $M=2$, and a torus width of $4000$ nodes. Left: $w=4$, middle: $w=8$, and right: $w=12$. Simulations by Omer Tamuz. } 
		\label{fig-limit}
	\end{center}
\end{figure}

Other results in the paper of independent interest include Theorem \ref{prop3}, which establishes  existence and uniqueness of the solution of the differential equation mentioned above, and Theorem \ref{prop6}, which says that for a certain family of Gaussian fields the supremum of its occupation kernel on Lipschitz functions is finite.

Finally we remark that our methods can be adapted easily to certain other variants of the Schelling model. For example, we may consider a perturbed variant of the model, where each node acts against its own preference with probability $p\in(0,1)$ every time its Poisson clocks rings, e.g.\ it chooses some opinion uniformly at random from $\{1,\dots,M \}$ instead of changing its opinion to the most common opinion in its neighborhood. In this perturbed model a variant of Theorem \ref{thm2} still holds, but the continuum bias function would evolve slower than with the unperturbed dynamics.

The results above are stated for the single-site-update variant of the Schelling model. In the version of the Schelling model studied in certain other papers, however, the nodes swap opinions rather than changing opinions; equivalently, the nodes have a fixed opinion and they change locations in order to be surrounded by nodes of a similar opinion to themselves. 
In this formulation of the model we would consider a finite grid (e.g.\ the torus), and each time step could consist of choosing two nodes $i,j$ uniformly at random and swapping their opinions if the opinion of node $i$ (resp.\ $j$) equals the most common opinion in the neighborhood of node $j$ (resp.\ $i$). Defining $Y^w$ using \eqref{eq24} and \eqref{eq60}, the continuum approximation $Y$ to $Y^w$ would evolve as described by a particular differential equation with random initial data, i.e., a variant of Theorem \ref{thm2} still holds in this setting. This differential equation also describes certain variants closely related to Schelling's original model, where some nodes are unoccupied and individuals may move to unoccupied sites if they are not satisfied. See the introduction of Section \ref{sec:cont} for more details.

Our results above also extend easily to other lattices than $\Z^N$ and $\frk S^N$ (assuming $\frk N(i)$ is still defined by \eqref{eq78} for each node $i$).

\subsection{Notation}
We will use the following notation:
\begin{itemize}
\item If $a$ and $b$ are two quantities whose values depend on some parameters, we write $a\preceq b$ (resp. $a \succeq b$) if there is a constant $C$ independent of the parameters such that $a \leq C b$ (resp. $a \geq C b$). We write $a \asymp b$ if $a\preceq b$ and $a \succeq b$.  (We will sometimes abuse notation and use the same terminology when $C$ depends on some parameters but not others, but this will be made clear in context.)
\item For any $N\in\BB N$ let $\lambda$ denote the Lebesgue measure on $\BB R^N$ or the torus $\mcl S^N$.
\item For $N\in\N$ and a topological space $V$ let $\mcl C(V^N)$ denote the space of continuous real-valued functions on $V^N$, equipped with the topology of uniform convergence on compact sets. 
\item For $N\in\N$ and either $V=\mcl S$ or $V=\R$ let $\mcl L(V^N)$ denote the space of real-valued Lipschitz continuous functions on $V^N$ equipped with the topology of uniform convergence on compact sets. For $K>0$ define $\mcl L^{K}(V^N)\subset\mcl L(V^N)$ by
\end{itemize}
\eqbn
	\mcl L^{K}(V^N) = \left\{y\in\mcl L(V^N)\,:\,
	\|y\|_{L^\infty}\leq K2^N\text{ and }\forall k\in\{1,...,N\},
	\sup_{\begin{subarray}{c}
			\text{$x,x'\in V^N,x_k\neq x'_k,$}\\
			\text{$x_j=x'_j\forall j\neq k$}
		\end{subarray}}
	\frac{|y(x)-y(x')|}{|x_k-x'_k|}\leq K2^{N-1}\right\}.
\eqen
\begin{itemize}
\item For $M,N\in\N$ and either $V=\mcl S$ or $V=\R$ define $\mcl C_M(V^N)$ (resp. $\mcl L_M(V^N)$, $\mcl L^K_M(V^N)$) to be the space of functions $f=(f_1,...,f_M)$ taking values in $\BB R^M$, such that for each $m\in\{1,...,M\}$ we have $f_m\in \mcl C(V^N)$ (resp. $f_m\in\mcl L(V^N)$, $f_m\in\mcl L^K(V^N)$). 
\item For any topological space $V$ let $\mcl B(V)$ denote the Borel $\sigma$-algebra.
\end{itemize}
See Sections \ref{sec:intro1}, \ref{sec:mainres}, and \ref{sec:cont} for additional notation.

\subsection*{Acknowledgements}
We thank Omer Tamuz for introducing us to the Schelling model, for numerous helpful discussions, for comments on an earlier draft of this paper, and for allowing us to use his simulations in Figure \ref{fig-limit}. We thank Chris Burdzy for our discussion about his paper \cite{bb01}, and we thank Matan Harel for our discussions about Proposition \ref{prop33}. We thank Nicole Immorlica, Robert Kleinberg, Brendan Lucier, and Rad Niazadeh for discussions about this paper and some of their related work \cite{schel-klein1,schel-klein2}.  The first author was supported by a fellowship from the Norwegian Research Council. The second author was supported by NSF grants  DMS 1209044 and DMS 1712862.

\section{The continuum Schelling model}
\label{sec:cont}

In this section we will introduce a differential equation which describes the early phase of the Schelling dynamics when the window size $w$ is large. We call this differential equation with appropriate initial the continuum Schelling model.

The main result of this section is the existence and uniqueness of solutions of the differential equation (Theorem \ref{prop3}), along with a result on the occupation kernel of Gaussian fields (Theorem \ref{prop6}) and some properties of the continuum one-dimensional dynamics (Proposition \ref{prop20a}).

Let $N\in\N$, $M\in\{2,3,\dots \}$, and $R\in\{3,4,\dots \}$, and let $V=\mcl S$ or $V=\mcl \R$. The solution of the differential equation we will define just below is a function $Y=(Y_1,\dots,Y_M)$, $Y_m:V^N\times \R_+\to \BB R$, which is the continuum analog of the function $Y^w$ defined by \eqref{eq24} and \eqref{eq60}. Define the plurality function $p:\BB R^M\mapsto\{0,1,...,M\}$ by 
\eqbn 
p(y)=
\left\{
  \begin{array}{ll}
    m &\text{ if   } y_m>\underset{m'\in\{1,...,M\}\backslash\{m\}}{\max}y_{m'},\\
    0 &\text{ if there is no } m \text{ for which }  y_m>\underset{m'\in\{1,...,M\}\backslash\{m\}}{\max}y_{m'}.
  \end{array}
  \right.
\eqen
Letting $\mcl N\subset (-1,1)^N$ be as in Section \ref{sec:intro1} define the neighborhood of $x\in V^N$ by
\eqbn
\mcl N(x) = \{x'\in V^N\,:\,x'-x\in\mcl N\},
\eqen
where we view $x'-x$ modulo addition of an element in $R\Z^N$ if $V=\mcl S$.
Consider the following differential equation
\eqb
\frac{\partial Y_m}{\partial t}(x,t) = \int_{x'\in\mcl N(x)} (1-M^{-1}) \bd 1_{p(Y(x',t))=m}-M^{-1}\bd 1_{p(Y(x',t))\neq m}\,dx'
\qquad m\in\{1,\dots,M\},\,x\in V^N,\,t\geq 0.
\label{eq1}
\eqe
We will prove in Proposition \ref{prop18} that this differential equation approximates the early phase of the Schelling dynamics well for large window size. 

An $(N,M)$-random field is a random map from $\R^N$ to $\R^M$. Let $B$ be a multivariate Gaussian $(N,M)$-random field (see Section~\ref{sec:intro2}) on the probability space $(\Omega,\mcl F,\P)$ with mean and covariance functions given by $\frk m_m(x)=0$ and
\eqb
\frk C_{m,m'}(x,x') = 
\left\{
  \begin{array}{ll}
  	\frac{M-1}{M^2}\lambda\big(\mcl N(x)\cap\mcl N(x')\big) &\text{ for } m=m',\\
    \frac{-1}{M^2} \lambda\big(\mcl N(x)\cap\mcl N(x')\big)  &\text{ for } m\neq m'.
  \end{array}
  \right.
 \label{eq20}
\eqe
We prove in Lemma \ref{prop19} that $B$ is well-defined as a continuous field.  Although we will not need this formulation, we remark that one way to construct $B$ involves starting with $W= (W_1, W_2, \ldots, W_M)$, where the $W_i$ are i.i.d.\ instances of white noise (each rescaled by $1/\sqrt{M}$) on $N$-dimensional space, and then writing 
$$\wt W_i = W_i - \frac{1}{M} \sum_{i=1}^M W_i,$$ 
so that the $\wt W_i$ sum up to zero a.s.\ and each $\wt W_i$ describes (in a limiting sense) the ``surplus'' of individuals with opinion $i$. We can then let $B(x)$ denote the integral of $\wt W/\sqrt{2}$ over the set $\mcl N(x)$.

Let the initial data of \eqref{eq1} be given by $B$
\eqb
Y(x,0) = B(x),\qquad\forall x\in V^N.
\label{eq2}
\eqe
The following theorem will be proved in Section \ref{sec:welldef}.
\begin{thm}[Existence and uniqueness for \eqref{eq1}, \eqref{eq2}]
Let $V=\mcl S$ and $N\in\N$, or let $V=\R$ and $N=1$. Let $M\in\{2,3,\dots \}$ and $R\in\{3,4,\dots \}$, and let $\mcl N\subset(-1,1)^N$ be as defined in Section \ref{sec:intro1}. Then the initial value problem \eqref{eq1}, \eqref{eq2} a.s.\ has a solution $Y:V^N\times [0,\infty)\to \BB R^M$. This solution can be written as the sum of the function $(x,t)\mapsto B(x)$ and a function $y:V^N\times[0,\infty)\to \BB R^M$ satisfying the following properties. For any $t>0$, we have $y(\cdot,t)\in\mcl L^{t}_M(V^N)$, so that in particular $y(\cdot, 0) = 0$, and $y$ is continuously differentiable in $t$. The solution $Y$ just described is unique in the space of functions satisfying these properties. 
\label{prop3}
\end{thm}
Note that we have not proved that the initial value problem \eqref{eq1}, \eqref{eq2} is well-defined for $N\geq 2$ and $V=\R$, but we believe that the above theorem also holds in this case. As we will discuss in Section \ref{sec:welldef} there exist initial data for which \eqref{eq1} does not have a unique solution (also when $N=1$ and/or $V=\mcl S$), and solutions of \eqref{eq1} do not in general vary continuously with the initial data.

For any $\omega\in\Omega$ and $m\in\{1,...,M\}$ define the random function $\phi_m^\omega: \mcl C_M(V^N) \rightarrow \mcl L^{1}(V^N)$ by
\eqb
\phi_m^\omega(y):x \mapsto 
\int_{x'\in\mcl N(x)} 
(1-M^{-1})\bd 1_{p(B(x')+y(x'))=m}
-M^{-1}\bd 1_{p(B(x')+y(x'))\neq m}
\,dx',
\quad\forall x\in V^N,\,y\in\mcl C_M(V^N).
\label{eq8}
\eqe
Also define $\phi^\omega:\mcl C_M(V^N)\to \mcl L^{1}_M(V^N)$ by $\phi^\omega=(\phi_1^\omega,...,\phi_M^\omega)$. Note that solving \eqref{eq1}, \eqref{eq2} is equivalent to solving the following initial value problem
\eqb
\begin{split}
\frac{\partial y}{\partial t}(\cdot,t)&= \phi^\omega(y(\cdot,t)),\qquad
t\geq 0,\\
y(x,0)&=0,\qquad x\in V^N.
\end{split}
\label{eq12}
\eqe
We obtain a solution to \eqref{eq1}, \eqref{eq2} by defining $Y(x,t)=y(x,t)+B(x)$ for all $t\geq 0$ and $x\in V^N$. 

We also observe that \eqref{eq1} is equivalent to a single differential equation when $M=2$. Let $\op{sign}:\R\to\{-1,0,1 \}$ denote the sign function, which is defined to be 0 at 0. Defining $\wh Y=Y^1-Y^2$, the initial value problem \eqref{eq1}, \eqref{eq2} is equivalent to
\eqb
\begin{split}
	\frac{\partial \wh Y}{\partial t}(x,t) &= \int_{x'\in\mcl N(x)} \op{sign} \wh Y(x',t)\,dx',\qquad\forall x\in V^N,\,t\geq 0,\\
	\wh Y(x,0)&=\wh B(x),
\end{split}
\label{eq71}
\eqe
where $\wh B$ is the centered Gaussian field with covariances $\op{Cov}(\wh B(x),\wh B(x'))=\lambda(\mcl N(x)\cap\mcl N(x'))$.

Finally, we will briefly state the analog of \eqref{eq1}, \eqref{eq2} for the setting of the pair-swapping variant of the Schelling model. See the end of Section \ref{sec:mainres} for the definition of this model. In this variant of the model the initial data for the continuum approximation $Y:\mcl S^N\times\R_+\to \R^M$ to $Y^w:\mcl S^N\times\R_+\to\R^M$ is still given by \eqref{eq2}, while the differential equation describing the evolution of $Y$ is given by 
\eqb
\begin{split}
	&\frac{\partial Y_m}{\partial t}(x,t) = \int_{x'\in\mcl N(x)} 
	(\lambda(S^N)-\Lambda_m(t))\bd 1_{p(Y(x',t))=m}
	-
	\Lambda_m(t)\bd 1_{p(Y(x',t'))\neq m}
	\,dx',\\
	&m\in\{1,\dots,M\},\,\,\,
	x\in \mcl S^N,\,\,\,
	t\geq 0,\,\,\,
	\Lambda_m(t):=\lambda(\{ x\in \mcl S^N\,:\, p(Y(x,t))= m \}).
\end{split}
\label{eq69}
\eqe 
The main difference between \eqref{eq1} and \eqref{eq69} is that the rate at which $Y$ changes in \eqref{eq69} depends on the {\em overall} fraction $\Lambda_m(t)$ of points with the various biases. Also observe that the integral of $Y_m(\cdot,t)$ is constant in time, which is consistent with the fact that the number of nodes with opinion $m$ is constant. Theorems \ref{thm2} and \ref{prop3} (for the model on the torus) also hold for the pair-swapping Schelling model, and are proved exactly as before. Notice in particular that any solution of \eqref{eq69}, \eqref{eq2} can be written on the same form as the function $Y$ in Theorem \ref{prop3} (except that $\mcl L_M^t(V^N)$ is replaced by $\mcl L_M^{Ct}(V^N)$ for some constant $C>1$), and since $y(\cdot,t)\in\mcl L_M^{Ct}(V^N)$ we can still apply Theorem \ref{prop6} in this setting. The initial value problem \eqref{eq69}, \eqref{eq2} also describes the situation where each node is unoccupied with constant probability in the initial configuration, and individuals (who have a fixed type or opinion) may move to an unoccupied node if this would make them satisfied.

\subsection{Occupation measures of random fields: basic definitions}
\label{sec:intro2}
We now give a short introduction to the theory of occupation measures of random fields, which is used frequently in our study of the continuum Schelling model. We refer to \cite{gh80} for further information.

Let $N,M\in\N$. An $(N,M)$-random field on a probability space $(\Omega,\mcl F,\BB P)$ is a collection $B$ of random variables with values in $\R^M$, which are indexed by an $N$-dimensional vector space $\wh{V}$, i.e.,\ $B=\{ B(x)\,:\,x\in \wh{V}\}$ or $B=(B(x))_{x\in \wh{V}}$. If $M=1$ we say that the field is an $N$-random field. In the remainder of the paper we consider $\wh V=V^N$ for ${V}=\BB R$ or ${V}=\mcl S$, where $S$ is a one-dimensional torus. 

The field $B$ is Gaussian if $M=1$ and if, for every $k\in\BB N$ and $x^1,...,x^k\in V^N$, the random variable $( B(x^1),..., B(x^k))$ is multivariate Gaussian. We say that $B=(B_1,...,B_M)$ is a multivariate $(N,M)$-Gaussian field if, for every $\alpha\in\BB R^M$, the weighted sum $\sum_{m=1}^M \alpha_m B_m$ is a real-valued Gaussian field. By e.g.\ \cite{adlerbook}, a multivariate Gaussian field is uniquely determined by its mean $\frk m=(\frk m_1,...,\frk m_M)$ and covariance matrix ${\frk C}=({\frk C}_{m,m'})_{m,m'\in\{1,...,M\}}$, which satisfy the following relations with $^t$ denoting the transpose of a matrix
\eqbn
\frk m(x) = \BB E[B(x)],\,\,\,\,\,\,
\frk C(x,x') = \BB E\big[( B(x)-\frk m(x))^t( B(x')-\frk m(x'))\big],\qquad x,x'\in\wh V.
\eqen

White noise on $V^N$ for $V=\mcl S$ or $V=\R$ is a collection of random variables $W=(W(X)\,:\,X\in\mcl B(V^N))$ such that (i) $W(X)\sim \mcl N(0,\lambda(X))$ for any $X\in\mcl B(V^N)$, (ii) $W(X_1\cup X_2)=W(X_1)+W(X_2)$ if $X_1,X_2\in\mcl B(V^N)$ and $X_1\cap X_2=\emptyset$, and (iii) $W(X_1)$ and $W(X_2)$ are independent if $X_1,X_2\in\mcl B(V^N)$ and $X_1\cap X_2=\emptyset$. Note that for a fixed set $X\in\mcl B(V^N)$ we can define an $N$-random field $\wh B=\{\wh B(x)\,:\,x\in V^N\}$ by defining $\wh B(x) := W(x+X)$, where $x+X=\{x+x'\,:\,x'\in X\}$ for any $x\in V^N$. We call a field that can be written on this form a moving average Gaussian field. 

The following definition is from \cite[Section 21]{gh80}. See \cite[Theorem 6.3]{gh80} for a proof that the occupation kernel $\alpha$ described below is well-defined when $\mu_X$ is a.s.\ absolutely continuous with respect to Lebesgue measure $\lambda$.

\begin{defn}[Occupation measure and occupation kernel]
	Let $N,M\in\N$, let $V=\mcl S$ or $V=\BB R$, and consider an $(N,M)$-random field $ B=( B_1(x),..., B_M(x))_{x\in V^N}$ on the probability space $(\Omega,\mcl F,\BB P)$. 
	\begin{itemize}
		\item The \emph{occupation measure} $\mu=(\mu_X)_{X\in\mcl B(V^N)}$ is defined by 
		\eqbn
		\mu_X(A) := \lambda(X\cap  B^{-1}(A)),\qquad A\in \mcl B(\BB R^M).
		\eqen
		\item If $\mu_X$ is a.s.\ absolutely continuous with respect to Lebesgue measure $\lambda$ for all $X\in {\mcl B}(V^N)$, let ${\alpha}(a,X)$ denote the Radon-Nikodym derivative of $\mu_X$ with respect to $\lambda$, i.e.,
		\eqb
		\lambda(X\cap B^{-1}(A)) = \int_A \alpha(a,X)\,da,\qquad A\in \mcl B(\BB R^M). 
		\label{eq19}
		\eqe
		Let $\alpha$ be chosen such that $\alpha(\cdot,X)$ is measurable for each fixed $X\in\mcl B(V^N)$, and $\alpha(a,\cdot)$ is a $\sigma$-finite measure on $(V^N,\mcl B(V^N))$ for each $a\in\BB R^M$. We call $\alpha$ the \emph{occupation kernel} of $B$.
		\item For $f:V^N\to\R^M$ let $\alpha(f,\cdot,\cdot)$ denote the occupation kernel (provided it exists) of the field $B-f$.
		\item More generally, for $k\in\{1,...,N-1\}$, $x'=(x'_1,\dots,x'_k)\in V^k$, and $f:V^N\rta\BB R^M$, let $\alpha(x',f,\cdot,\cdot)$ denote the occupation kernel (provided it exists) of the $(N-k,M)$-random field
		\eqbn \big(B(x_1,\dots,x_{N-k},x'_1,\dots,x'_k)-f(x_1,\dots,x_{N-k},x'_1,\dots,x'_k)\big)_{ (x_1,\dots,x_{N-k})\in V^{N-k} }.
		\eqen
	\end{itemize}
	\label{def2}
\end{defn}

\subsection{Supremum of the occupation kernel on Lipschitz functions for moving average Gaussian fields}
\label{sec:suploctime}
In \cite{bb01,bb02}  the authors prove that the supremum on Lipschitz curves of Brownian local time is finite. In this section we will prove a higher-dimensional analog of this result, stated in Theorem \ref{prop6} below. We consider the supremum on Lipschitz functions of the occupation kernel of a particular centered moving average Gaussian field $B$. 

The idea of the proof is to define various $(N+1)$-dimensional boxes on different scales, and bound the number of such boxes intersected by both the graph of $B$ and a Lipschitz function $f$, uniformly over all choices of $f$. On each scale we proceed by using that $f$ is approximately constant, while $B$ fluctuates rapidly. We also prove that if $f$ and $g$ are uniformly close then the occupation kernel of $B$ on $f$ and $g$, respectively, are close with high probability. We start the section by proving existence and basic properties of the occupation kernel of $B$ on any fixed Lipschitz function. 

Theorem \ref{prop6} will imply that for any solution $Y$ of \eqref{eq1}, \eqref{eq2} on $\mcl S^N$ and any $m,m'\in\{1,...,M\}$, $m\neq m'$, the field $Y_m-Y_{m'}$ is close to 0 only for a small subset of $\mcl S^N$ simultaneously. This will help us to prove existence, uniqueness and other properties of solutions to \eqref{eq1}, \eqref{eq2}. Since $B_m-B_{m'}$ has the law of a constant multiple of $B_m$, it will be sufficient to obtain our result for the following real-valued field $B$.
\begin{remark}
We will assume throughout the section that $B=(B(x))_{x\in \mcl S^N}$ is the centered moving average Gaussian field with covariances given by
\eqb
\frk C(x,x') = \lambda\big(\mcl N(x)\cap\mcl N(x')\big),\qquad
x,x'\in\mcl S^N.
\label{eq68}
\eqe
\end{remark}
Observe that $B$ is a moving average Gaussian field as defined in Section \ref{sec:intro2}. The proof of the following theorem uses ideas from \cite{bb01,bb02}. See Section \ref{sec:intro2} for the definition and basic properties of the occupation measure of random fields.
\begin{thm} 
Let $B=(B(x))_{x\in\mcl S^N}$ be the centered Gaussian field with covariances given by \eqref{eq68}, and let $K>0$. For each fixed $f\in \mcl L^K(\mcl S^N)$ the occupation kernel $\alpha(f,\cdot,\cdot)$ of $B-f$ exists almost surely. Furthermore, almost surely there exists a random field $\wt \alpha=\{\wt \alpha(f)\,:\,f\in\mcl L^K(\mcl S^N)\}$ satisfying the following properties.
\begin{enumerate}
\item[(I)] For each fixed $f\in\mcl L^K(\mcl S^N)$, we have $\wt \alpha(f)=\alpha(f,\mcl S^N,0)$ a.s.
\item[(II)] Almost surely, $f\mapsto \wt \alpha(f)$ is continuous on $\mcl L^K(\mcl S^N)$ equipped with the supremum norm.
\item[(III)] Almost surely, $\sup_{f\in \mcl L^K(\mcl S^N)}\wt \alpha(f)<\infty$.
\end{enumerate}
\label{prop6}
\end{thm}

First we will prove existence and various properties of the occupation kernel for a certain class of Gaussian random fields.

\begin{lem}
Let $K>0$ and $f \in\mcl L^K(\mcl S^N)$. Define $\wt B=B-f$, where $B$ is the centered Gaussian field on $\mcl S^N$ with covariances given by \eqref{eq68}. Then the following holds a.s.
\begin{itemize}
\item[(I)] $\wt B$ has an occupation kernel $\alpha(f,\cdot,\cdot)$.
\item[(II)] For all $k\in\{1,...,N-1\}$ and almost all $x'=(x'_1,...,x'_k)\in\mcl S^k$ the $(N-k)$-field $$x''\mapsto \wt B(x''_1,...,x''_{ N-k},x'_1,...,x'_{k})$$ for $x''=(x''_1,...,x''_{N-k})\in {\mcl S}^{N-k}$, has an occupation kernel $\alpha(x',f,\cdot,\cdot)$. For almost all $a\in\BB R$ the following holds for all $X'\in \mcl B({\mcl S}^k)$ and $X''\in \mcl B({\mcl S}^{N-k})$
\eqb
\alpha(f,a,X''\times X') = \int_{X'} \alpha(x',f,a,X'')\,dx'.
\label{eq13}
\eqe
\end{itemize}
\label{prop14}
\end{lem}

\begin{proof}[Proof of Lemma \ref{prop14}]
By \eqref{eq68}, we have Var$(|\wt B(x')-\wt B(x)|)\asymp \lambda\big( \mcl N(x)\Delta \mcl N(x' \big)\big)$, where the implicit constant is independent of $x,x'\in\mcl S^N$ and $\Delta$ denotes symmetric difference. Since $x,x'\in \mcl S^N$, the difference $x-x'$ is only defined as an element in $\R^N$ modulo $R\Z^N$, but we will view $x-x'$ as an element in $\R^N$ by choosing the equivalence class such that $\|x-x'\|_1$ is minimized (with some arbitrary choice of equivalence class in case of draws). 
 By \eqref{eq65} it follows that Var$(|\wt B(x')-\wt B(x)|)\succeq \|x'-x\|_{1}$, where the implicit constant is again independent of $x,x'\in\mcl S^N$, but may depend on all other parameters. Since the probability density function of a standard normal random variable is bounded, for any $x\in\mcl S^N$,
\eqb
\liminf_{\ep\rightarrow 0} \ep^{-1} \int_{\mcl S^N} \BB P[|\wt B(x')-\wt B(x)|\leq\ep]\,dx'
\preceq \int_{\mcl S^N} \|x'-x\|_{1}^{-1/2} \,dx' <\infty.
\label{eq85}
\eqe
By \cite[Theorem 21.12]{gh80} we know that if $\wt B$ is a random field for which the left side of \eqref{eq85} is finite, then $\wt B$ has an occupation kernel a.s. Applying this result concludes our proof of (I).

The same theorem also implies the existence of the occupation kernel $\alpha(x',f,\cdot,\cdot)$ for a.e.\ fixed $x'=(x'_1,...,x'_k)$.
The identity \eqref{eq13} follows by \cite[Theorem 23.5]{gh80}, since for any $k\in\{1,...,N-1 \}$ and $x=(x''_1,\dots,x''_{N-k},x'_1,\dots,x'_k)\in\mcl S^N$, $x''=(x''_1,\dots,x''_{N-k})\in\mcl S^{N-k},x'=(x'_1,\dots,x'_k)\in\mcl S^{k}$,
\eqbn
\begin{split}
	\int_{\mcl S^{N-k}} \sup_{\ep>0} \ep^{-1} 
	\BB P[|\wt B((\wt x_1,\dots,\wt x_{N-k},x'_1,\dots,x'_k))
	-
	\wt B((x''_1,\dots,x''_{N-k},&x'_1,\dots,x'_k))|\leq\ep]
	\,d\wt x\\
	&\preceq 
	\int_{\mcl S^{N-k}} \|\wt x-x''\|_{1}^{-1/2} \,d\wt x<\infty.
\end{split}
\eqen
\end{proof}

For $K>1$ let $\mcl L^{K,k}(\mcl S^N)$ denote the space of functions $f\in\mcl L(\mcl S^N)$ that satisfy the following properties: 
(i) for each $x\in 2^{-k}\Z^N$, $f(x)\in 2^{-k}\Z$,
(ii) for each  $x\not\in 2^{-k}\Z^N$, $f(x)$ is a weighted average of $f$ at the $\leq 2^N$ points $\wt x\in 2^{-k}\Z^N$ which satisfy $\|x-\wt x\|_{\infty}<2^{-k}$, where the weights are defined as in \eqref{eq60},
(iii) $\|f\|_{L^\infty(\mcl S^N)}\leq K2^N+2^{-k}$, and
(iv) if $x,x'\in 2^{-k}\Z^N$ satisfy $\|x-x'\|_{1}=2^{-k}$ then $|f(x)-f(x')|\leq K2^{N-k-1}+2^{2-k}$. Define 
\eqb
\wt{\mcl L}^K(\mcl S^N)=\bigcup_{k\in\BB N} \mcl L^{K,k}(\mcl S^N).
\label{eq70}
\eqe
The following lemma implies that $\wt{\mcl L}^K(\mcl S^N)$ is dense in $\mcl L^K(\mcl S^N)\cup\wt{\mcl L}^K(\mcl S^N)$ for the supremum norm.
\begin{lem} 
For any $f\in \mcl L^K(\mcl S^N)$ and $k\in\BB N$ there is a canonically defined element $f^k\in \mcl L^{K,k}(\mcl S^N)$ such that $\|f-f^k\|_{L^\infty}\leq 2^{-k}$, where the implicit constant is independent of $k$ and $f$, but may depend on $K$ and $k$. 
\label{prop9}
\end{lem}
\begin{proof}
For each $x\in 2^{-k}\Z^N$ let $f^k(x)$ be the multiple of $2^{-k}$ which is closest to $f(x)$. For $x\not\in 2^{-k}\Z^N$ let $f^k(0)$ be defined such that condition (ii) in the definition of ${\mcl L}^{K,k}(\mcl S^N)$ is satisfied. It is immediate that the properties (i)-(iv) in the definition of ${\mcl L}^{K,k}(\mcl S^N)$ are satisfied.
\end{proof}

By \eqref{eq68} we may couple $B$ with an instance of white noise $W$ on $\mcl S^N$ such that for any $x\in\mcl S$, we have $B(x)=W(\mcl N(x))$. Define a filtration $(\mcl F_t)_{t\in[0,R-2]}$ by
\eqb
\mcl F_t=\sigma( W|_{\mcl B(D_t)} ),\qquad
D_t := \bigcup_{s\in[0,t],\, (x_2,\dots,x_N)\in\mcl S^{N-1}} \mcl N( (s,x_2,\dots,x_N) ).
\label{eq63}
\eqe
Let $0\leq s<t<R-2$ and $(x_2,\dots,x_N)\in\R^{N-1}$. Conditioned on $\mcl F_s$ the random variable $B(t,x_2,\dots,x_N)$ is a Gaussian random variable with expectation $W(\mcl N((t,x_2,\dots,x_N))\cap D_s)$ and variance depending only on $t-s$. Therefore there exists a function $p:(0,R-2)\times\R\times\R$ such that $p(t-s,W(\mcl N((t,x_2,\dots,x_N))\cap D_s),\cdot)$ is the probability density function of $B(t,x_2,\dots,x_N)$ conditioned on $\mcl F_s$. By \eqref{eq68}, and with $x=(t,0,\dots,0)$ and $\wt p(s,a,y):= (2\pi s)^{-1/2} \exp(-|y-a|^2/(2 s))$,
\eqb
p(t,a,y)= \wt p(\lambda(
\mcl N(x)\setminus D_0)
,a,y).
\label{eq56}
\eqe

\begin{lem}
	Let $p$ be given by \eqref{eq56}. For $t\in(0,R-2)$ and $a,y\in\R$,
	\eqb
	p(t,a,y) \preceq t^{-1/2},
	\label{eq58}
	\eqe
	where the implicit constant is independent of $t,a,y$, but may depend on $\mcl N$. Let $\delta\in(0,1/2)$, $a\in\R$, and assume $f,g:[0,R-2]\to\R$ satisfy $\|f-g\|_{L^\infty([0,R-2])}<\delta$. Then
	\eqb
	\int_{0}^{R-2} |p(t,a,f(t)) -p(t,a,g(t))|\,dt \preceq \delta\log(1/\delta),
	\label{eq59} 
	\eqe
	where the implicit constant may depend on $\mcl N$.
	\label{prop25}
\end{lem}
\begin{proof}
Since $p$ is continuous in $t$ by \eqref{eq56} and since $p(t,a,y)\preceq 1$ for $t>2$ and any $a,y\in\R$, in order to prove \eqref{eq58} it is sufficient to prove $\liminf_{t\rta 0} t^{1/2}p(t,a,y)\preceq 1$. By the explicit formula for $p$ this follows from \eqref{eq65}.

The estimate \eqref{eq59} follows by the exact same argument as for the case $p=\wt p$, which is considered in \cite[Lemma 3.4]{bb01}, except that we use \eqref{eq58} instead of the corresponding estimate for $\wt p$.
\end{proof}

\begin{lem}
Let $d\in (0,1\wedge(R-2)]$, and let $f,g$ be two functions defined on some $N$-dimensional cube $I\subset\mcl S^N$ with side lengths $d$. Assume that $\|f-g\|_{L^\infty(I)}\leq \delta$ for some $\delta>0$. Then for all $b\geq 1$ and with $\alpha(f,I)=\alpha(f,0,I)$ and $\alpha(g,I)=\alpha(g,0,I)$ as in Definition \ref{def2},
\eqbn
\log \BB P\Big[\alpha(f,I)-\alpha(g,I)\geq b d^{(N-3/4)} (\delta\log(\delta^{-1}))^{1/2}\Big] 
\preceq -b,
\eqen
where the implicit constant is independent of $\delta$, 
$f$ and $g$.
\label{prop10}
\end{lem} 
\begin{proof}
Assume without loss of generality that $I=[0,d]^N$. Couple $B$ with an instance of white noise $W$ as described above the statement of Lemma \ref{prop25}, and recall the filtration $(\mcl F_t)_t$ and the sets $D_t\subset\R^N$ defined by \eqref{eq63}. For any $t\in[0,d]$ define $I_t=[0,t]\times [0,d]^{N-1}\subset\mcl S^N$, and let $A_t^f=\alpha(f,I_t)$, $A_t^g=\alpha(g,I_t)$, and $\wt A_t = A_t^f-A_t^g$. Then $(A_t^f)_{t\in[0,d]}$, $(A_t^g)_{t\in[0,d]}$, and $(\wt A_t)_{t\in[0,d]}$ are stochastic processes adapted to the filtration $(\mcl F_t)_{t\in[0,d]}$. To simplify notation we write $f(t,x')$ instead of $f(t,x'_1,\dots,x'_{N-1})$ for $t\in\R$ and $x'=(x'_1,\dots,x'_{N-1})\in\R^{N-1}$. Let $t_1,t_2\in [0,d]$ satisfy $t_1<t_2$. By Lemma \ref{prop14} (II) and  \eqref{eq58} of Lemma \ref{prop25},
\eqbn
\begin{split}
\BB E\left[A^f_{t_2}-A^f_{t_1}\,|\,\mcl F_{t_1}\right] 
&=\BB E\left[\int_{[0,d]^{N-1}} \alpha(x', f(\cdot,x'),[t_1,t_2])\,dx'\,\Big|\,\mcl F_{t_1}\right]\\
&=\BB \int_{[0,d]^{N-1}}\int_0^{t_2-t_1} p\Big(s,
\mcl N((t_2,x'_1,\dots,x'_{N-1}))\cap D_{t_1},
f(t_1+s,x')\Big)\,ds\,dx'\\
& \preceq d^{N-1}\int_0^{t_2-t_1} s^{-1/2}\,ds\\
&\preceq d^{N-1/2}.
\end{split}
\eqen
By a similar argument $\BB E[A^g_{t_2}-A^g_{t_1}\,|\,\mcl F_{t_1}] \preceq d^{N-1/2}$. By \eqref{eq59} we get further
\eqbn
\begin{split}
|\BB E[\wt A_{t_2}-\wt A_{t_1}\,|\,\mcl F_{t_1}]| 
&= \left| \BB E\left[\int_{[0,d]^{N-1}}
\alpha\big(x',\wt f(\cdot,x'),[t_1,t_2]\big)-
\alpha\big(x',\wt g(\cdot,x'),[t_1,t_2]\big)\,
dx'\,\,\Big|\,\,\mcl F_{t_1}  \right] \right| \\
&\leq \int_{[0,d]^{N-1}}
\int_0^{t_2-t_1} \Big|
p\Big(s,\mcl N((t_2,x'_1,\dots,x'_{N-1}))\cap D_{t_1},f(t_1+s,x')\Big)\\
&\quad-p\Big(s,\mcl N((t_2,x'_1,\dots,x'_{N-1}))\cap D_{t_1},g(t_1+s,x')\Big)\Big|
\,ds\,dx'\\
& \preceq d^{N-1} \delta\log\delta.
\end{split}
\eqen
The lemma now follows by \cite[Lemma 2.1]{bb01}.
\end{proof}

\begin{lem}
Let $L>1$, $a\in\BB R$, and define $\wt L:=\lceil L^{1/2} \rceil$. For $m\in\{1,...,\wt L\}$ define the $(N+1)$-dimensional rectangle $J^m$ by
\eqbn
J^m = [(m-1)L^{-1}, mL^{-1}]
\times[0,L^{-1}]^{N-1}
\times[a,a+\wt L^{-1}].
\eqen
Let $A$ be the number of rectangles $J^m$ for $m\in\{1,...,\wt L\}$ which intersect the graph $\{(x,B(x))\in\mcl S^N\times\R\,:\,x\in\mcl S^N\}$ of $B$. Then the following estimate holds for all $\xi>1$ and $\ep\in(0,1]$
\eqbn
\BB P[A \geq \wt L^{1/2+\ep}] \preceq L^{-\xi},
\eqen
where the implicit constant is independent of $a$ and $L$, but depends on $\ep$ and $\xi$.
\label{prop15}
\end{lem}
\begin{proof}
Let $E_m$ be the event that the graph of $B$ intersects $J^m$, let $\wt J^m=[(m-1)L^{-1}, mL^{-1}]
\times[0,L^{-1}]^{N-1}$, let $x^m=\big(mL^{-1},0,...,0\big)\in \wt J^m$ and define the event $E'_m$ by $E'_m =\{ |B(x^m)-a|<\wt L^{-1+\ep/10} 
\}$. 
Define the random variable $A^m$ by $A^m=\sum_{m'\leq m} \bd 1_{E'_{m'}}$. By \eqref{eq58}, for $m_2>m_1$,
\eqbn
\BB P[E'_{m_2}\,|\,\mcl F_{m_1L^{-1}}]
\preceq ((m_2-m_1) \cdot L^{-1})^{-1/2} \wt L^{-1+\ep/10} \asymp (m_2-m_1)^{-1/2} \wt L^{\ep/10}.
\eqen
Therefore, for any $m\in\{1,\dots,\wt L\}$,
\eqbn
\BB E[A^{\wt L}-A^{m}\,|\,\mcl F_{mL^{-1}}] \preceq
\sum_{d=1}^{\wt L} d^{-1/2} \wt L^{\ep/10} \preceq \wt L^{1/2+\ep/10}.
\eqen
By applying \cite[Corollary I.6.12]{bass95} to a constant multiple of the sequence $\{A^m/\wt L^{1/2+\ep/10}\}_{1\leq m\leq \wt L}$, we get 
\eqbn
\log \P(A^{\wt L}> \wt L^{1/2+\ep}) \preceq -\wt L^{\ep/10}.
\eqen

For any $x,x'\in\mcl S$ and $b>1$ we have $\E[|B(x)-B(x')|^b]\preceq \|x-x'\|^{b/2}$ for an implicit constant depending on $b$. By a quantitative version of the Kolmogorov-Chentsov theorem as in e.g.\ \cite[Proposition 2.3]{qle3}, the function $B$ is $\gamma$-H\"older continuous with (random) constant $C(\gamma)$ for any $\gamma<1/2$, and $\P[C(\gamma)>C]$ decays faster than any power of $C$. In particular, $\P[C(1/2-\ep/100)>L^{\ep/100}]\preceq L^{-\xi}$ for any $\xi$. Observe that if 
$C(1/2-\ep/100)\leq L^{\ep/100}$
and
$E_m$ occurs, and if $L$ is sufficiently large, then $E'_m$ also occurs since for some $x\in\wt J^m$,
\eqbn
|B(x^m)-a|
\leq |B(x^m)-B(x)|+|B(x)-a|
\leq NL^{\epsilon/100} (L^{-1})^{1/2-\epsilon/100}+\wt L^{-1}
<\wt L^{-1+\epsilon/10}.
\eqen
Therefore 
\eqbn
E_m \subset E'_m \cup \{ C(1/2-\ep/100)>L^{\ep/100} \},
\eqen
so
\eqbn
\{A\geq \wt L^{1/2+\ep} \}\subset \{A^{\wt L}\geq \wt L^{1/2+\ep} \} \cup \{ C(1/2-\ep/100)>L^{\ep/100} \},
\eqen
and the lemma follows by a union bound.
\end{proof}

\begin{lem}
For any $k\in\BB N$ divide $\mcl S^N$ into $N$-dimensional cubes $I_j$, $j=1,...,(R2^{k})^N$, of side length $2^{-k}$, such that the cubes have pairwise disjoint interior. Also divide $\mcl S^N\times\BB R$ into  $(N+1)$-dimensional cubes $I'_{ij}$ of side length $2^{-k}$ for $j=1,...,(R2^{k})^N$ and $i\in \BB Z$, such that the cubes have pairwise disjoint interior. Let $\ep\in(0,1/100)$ and define the three events $G_k^1,G_k^2,G_k^3$ as follows
\begin{itemize}
\item $G_{k}^1$ is the event that for any $g\in\mcl L^K(\mcl S^N)$ the number of cubes $I'_{ij}$ intersecting the graph of both $g$ and $B$, is bounded by ${2}^{k(N-1/2+\ep)}$ for any $g$.
\item $G_k^2$ is the event that for any $f\in \mcl {L}^K(\mcl S^N)$ and $j\in\{1,...,2^{k}\}$, the approximations $f^k$ and $f^{k+1}$ to $f$ defined in Lemma \ref{prop9} satisfy
\eqbn
|\alpha(f^k,I_j)-\alpha(f^{k+1},I_j)| \leq 2^{k(-N+1/4+\ep)}.
\eqen
\item $G_k^3$ is the event that for any $j\in\{1,...,2^k\}$ and any two $f,g\in \mcl L^K(\mcl S^N)$ satisfying $\|f-g\|_{L^\infty}\leq 2^{-k}$, we have
\eqbn
|\alpha(f^k,I_j)-\alpha(g^{k},I_j)| \leq 2^{k(-N+1/4+\ep)}.
\eqen
\end{itemize}
Finally define $G_k$ by 
\eqbn
G_k:=\left(\bigcap_{k'\geq k} G_{k'}^1\right)\cap 
\left(\bigcap_{k'\geq k} G_{k'}^2\right)\cap 
\left(\bigcap_{k'\geq k} G_{k'}^3\right).
\eqen
Then $\BB P[G_k] \rightarrow 1$ as $k\rightarrow\infty$.
\label{prop16}
\end{lem}
\begin{proof}
It is sufficient to prove that for $b=1,2,3$ it holds that $\lim_{k\rta\infty} \BB P\big[\cap_{k'\geq k} G^b_{k'}\big]=1$. We consider the three cases $b=1,2,3$ separately. All implicit constants may depend on $R$ and $K$, but not on $k$.

Case $b=1$: For any $L\in\BB N$ divide $\mcl S^N\times\BB R$ into $(N+1)$-dimensional cubes $I'_{ij}(L)$ for $i\in\BB Z$ and $j\in \{1,...,(LR)^N\}$, such that the interior of the cubes are disjoint, and each cube has side length $L^{-1}$. Observe that $I'_{ij}=I'_{ij}(2^k)$ with $I'_{ij}$ as in the statement of the lemma.  Define $\wt L:=\lceil L^{1/2} \rceil\in\BB N$. Divide $\mcl S^N\times\BB R$ into $(N+1)$-dimensional rectangles $J_{ij}=J_{ij}(L)$ for $i\in\BB Z$ and $j\in \{1,...,R^N L^{N-1}\wt L\}$, such their interiors are pairwise disjoint, and such that each rectangle is a translation of $J_{i1}:=[0,\wt L^{-1}]\times[0,L^{-1}]^{N-1}\times[0,\wt L^{-1}]$. Assume $I'_{1,1}(L)=[0,L^{-1}]^{N+1}$, and that the projection of $I'_{ij}(L)$ (resp. $J_{ij}$) onto the last coordinate is given by $L^{-1}[i-1,i]$ (resp. $\wt L^{-1}[i-1,i]$).

Consider one of the rectangles $J_{ij}$. Find a cover $\{J_{ij}^m\}_{m=1}^{\wt L}$ of $J_{ij}$ of $(N+1)$-dimensional rectangles whose interiors are disjoint, and each of which is a translation of $[0,L^{-1}]^N\times[0,\wt L^{-1}]$ (note that unless $\wt L^2=L$, a small fraction of the rectangles $J_{ij}^m$ have a non-empty intersection with the complement of $J_{ij}$). Let $A_{ij}=A_{ij}(L)$ denote the number of such rectangles that contain a point of the graph of $B$. By Lemma \ref{prop15} the following holds for all $i,j$ and any $\xi>1$
\eqb
\BB P[A_{ij}\geq \wt L^{1/2+\ep}] \preceq \wt L^{-\xi},
\label{eq17}
\eqe
where the implicit constant depends on $\ep$ and $\xi$. We will now prove that on the event
\eqbn 
\wt E_L :=
\bigcap_{1\leq j\leq R^{N}L^{N-1}\wt L,-KL\leq i< KL} \{A_{ij}<\wt L^{1/2+\ep}\}
\eqen
the number of cubes $I'_{ij}(L)$ intersecting the graph of both $f$ and $B$ is $\preceq L^{(N-1/4+\ep/2)}$. Since $f\in\mcl L^K(\mcl S^N)$ the number of rectangles $J_{i'j'}$ intersecting the graph of $f$ is $\preceq L^{N-1/2}$. Assuming the event $\wt E_L$ occurs, for each such $i',j'$ there are $< \wt L^{1/2+\ep}$ rectangles $J^m_{i'j'}$ intersecting both $J_{i'j'}$ and the graph of $B$. By using $f\in\mcl L^K(\mcl S^N)$ again it follows that, for each $i',j'$, the number of cubes $I'_{ij}(L)$ intersecting both $J_{i'j'}$ and the graph of $B$ and $f$, is $\preceq \wt L^{1/2+\ep}$. Hence the total number of cubes $I'_{ij}(L)$ intersecting the graph of both $B$ and $f$ is $\preceq L^{N-1/2}\times\wt L^{1/2+\ep} \preceq L^{N-1/4+\ep/2}$.

By \eqref{eq17} and  a union bound, in order to complete the proof for the case $b=1$ it is sufficient to prove that, conditioned on the event $\cap_{L\in\N\,:\,L\geq L_0}\wt E_{L}$ for some $L_0\in\BB N$, the number of rectangles $I'_{ij}(L)$ intersecting the graph of both $f$ and $B$ is bounded by $L^{N-1/2+\ep}$ for all sufficiently large $L$. We will proceed by iterations as in the proof of \cite[Proposition 3.3]{bb01}. 

Condition on the event $\cap_{L\in N\,:\,L\geq L_0} \wt E_{L}$, and  choose $L\geq (2L_0)^2$. Since $\wt L>L_0$ the number of cubes $I'_{ij}(\wt L)$ intersecting the graph of both $f$ and $B$ is $<\wt L^{N-1/4+\ep}$. Each of these cubes $I'_{ij}(\wt L)$ is contained in the union of $\wt L^{N-1}$ rectangles $J_{i'j'}(L)$. By the definition of $\wt E_L$ and by $f\in\mcl L^K(\mcl S^N)$, for each fixed $i',j'$, the number of cubes $I'_{i''j''}(L)$ intersecting both $J_{i'j'}(L)$, the graph of $f$ and the graph of $B$, is bounded by $\wt L^{1/2+\ep}$. It follows that the number of cubes $I'_{i''j''}(L)$ intersecting both $I'_{ij}(\wt L)$, the graph of $f$ and the graph of $B$, is $\preceq \wt L^{N-1/2+\ep}$. Further we get that the number of cubes $I'_{i''j''}(L)$ intersecting the graph of both $f$ and $B$ is $\preceq \wt L^{N-1/4+\ep}\times \wt L^{N-1/2+\ep}\asymp L^{N-3/8+\ep}$.

Now choose $L\geq (4L_0)^4$. Since $\wt L\geq (2L_0)^2$ the number of cubes $I'_{ij}(\wt L)$ intersecting the graph of both $f$ and $B$ is $\preceq \wt L^{N-3/8+\ep}$. For fixed $i,j$ the number of cubes  $I'_{i'j'}(L)$ intersecting both $I'_{ij}(\wt L)$, the graph of $f$ and the graph of $B$, is $\preceq \wt L^{N-1/2+\ep}$, so the number of cubes $I'_{i'j'}(L)$ intersecting the graph of both $f$ and $B$ is $\preceq \wt L^{N-3/8+\ep}\times\wt L^{N-1/2+\ep} \asymp L^{N-7/16+\ep}$. By continued iterations it follows that, for sufficiently large $L$, the number of cubes $I'_{i''j''}(L)$ intersecting the graph of both $f$ and $B$ is $\preceq L^{N-1/2+2\ep}$. This completes the proof.

Case $b=2$: By Lemma \ref{prop9} we have $\|f^k-f^{k+1}\|_{L^\infty(\mcl S^N)}\preceq 2^{-k}$.  Therefore, for any $f\in\mcl L^K(\mcl S^N)$ and $j\in\{1,...,2^k\}$, Lemma \ref{prop10} implies that
\eqbn
\log\BB P\big[ |\alpha(f^k,I_j)-\alpha(f^{k+1},I_j)| \geq 2^{k(-N+1/4+\ep)}
\big] \preceq-2^{k\ep/2}.
\eqen
For each fixed $x\in I_{ij}\cap( 2^k\Z^N )$, $f^k(x)$ can take $\preceq 2^{k}$ different values. Conditioned on $f^k(x)$ the number of possible realizations of $f^k|_{I_j}$ and $f^{k+1}|_{I_j}$ is bounded by a constant. It follows that there are $\preceq 2^k$ possibilities for $f^k|_{I_j}$ and $f^{k+1}|_{I_j}$. Since the number of cubes $I_j$ is $\asymp 2^{kN}$, a union bound implies that
\eqbn
\log (1-\BB P[G_k^2]) \preceq - 2^{k\ep/2}. 
\eqen
We conclude by a union bound.

Case $b=3$: We proceed exactly as in the case $b=2$. The result follows by a union bound, Lemma \ref{prop10}, and by observing that the number of possible realizations of $f^k|_{I_j}$ and $g^{k}|_{I_j}$ for each fixed $j$ is $\asymp 2^{k}$.  
\end{proof}

\begin{proof}[Proof of Theorem \ref{prop6}]
We start by proving uniform continuity of $g\mapsto \alpha(g)$ on $\wt{\mcl L}^K(\mcl S^N)$, where $\wt{\mcl L}^K(\mcl S^N)$ is defined by \eqref{eq70}. Let $\eta,\beta>0$, and the choose $\wt k\in\BB N$ sufficiently large such that the event $G_{\wt k}$ of Lemma \ref{prop16} holds with probability at least $1-\eta$. Condition on the event $G_{\wt k}$. Consider any $f,g\in\wt{\mcl L}^K(\mcl S^N)$ such that $\|f-g\|_{L^\infty}\leq 2^{-\wt k}$, and let $f^k,g^k\in\mcl L^{K,k}(\mcl S^N)$ denote the approximations to $f,g$, respectively, defined in Lemma \ref{prop9}. By the definition of the events $G_k^1$ and $G_k^2$,
\eqbn
\begin{split}
|\alpha(f^{k+1})-\alpha(f^{k})| \preceq 2^{k(N-1/2+\ep)}\cdot 2^{k(-N+1/4+\ep)} = 2^{k(-1/4+2\ep)}
\end{split}
\eqen
for all $k\geq \wt k$ and a universal implicit constant. Since $f=f^k$ for sufficiently large $k$ the triangle inequality implies that, after we increase $\wt k$ if necessary, we have $|\alpha(f)-\alpha(f^{\wt k})|\leq\beta/3$. By the same argument $|\alpha(g)-\alpha(g^{\wt k})|\leq\beta/3$. By the definition of $G_{k}^1$ and $G_{k}^3$,
\eqbn
\begin{split}
|\alpha(f^{\wt k})-\alpha(g^{\wt k})| \preceq 2^{\wt k(N-1/2+\ep)}\cdot 2^{\wt k(-N+1/4+\ep)} = 2^{\wt k(-1/4+2\ep)}.
\end{split}
\eqen
Increasing $\wt k$ if necessary, it follows by the triangle inequality that $|\alpha(f)-\alpha(g)|\leq\beta$ with probability at least $1-\eta$ for all $f,g\in\wt{\mcl L}(\mcl S^N)$ satisfying $\|f-g\|_{L^\infty}\leq 2^{\wt k}$. Note that $\wt k$ is a function of $\beta$ and $\eta$, i.e., $\wt k=\wt k(\beta,\eta)$. 

Fix some $\eta>0$ and a sequence $(\beta_m)_{m\in\BB N}$ converging to 0. For any $m\in\BB N$ and any $f,g\in\wt{\mcl L}^K(\mcl S^N)$ satisfying $\|f-g\|_{L^\infty} \leq 2^{-\wt k(\beta_m,\eta 2^{-m})}$ we have $|\alpha(f)-\alpha(g)|\leq\beta_m$ with probability at least $1-\eta 2^{-m}$. By a union bound it holds with probability at least $1-\eta$ that $|\alpha(f)-\alpha(g)|\leq \beta_m$ for all $m\in\BB N$ and all $f,g\in\wt{\mcl L}^K(\mcl S^N)$ satisfying $\|f-g\|_{L^\infty} \leq 2^{-\wt k(\beta_m,\eta 2^{-m})}$. Since the choice of $\eta$ was arbitrary this implies uniform continuity of $g\mapsto\alpha(g)$.

Define $\wt\alpha$ to be the restriction to $\mcl L^K(\mcl S^N)$ of the continuous extension of $\alpha$ from $\wt{\mcl L}^K(\mcl S^N)$ to $\mcl L^K(\mcl S^N)\cup \wt{\mcl L}^K(\mcl S^N)$. Part (II) of the theorem follows by the definition of $\wt{\alpha}$ and uniform continuity of $\alpha$ on $\wt{\mcl L}^K(\mcl S^N)$. Part (III) of the theorem follows by using part (II) and that $\mcl L^K(\mcl S^N)$ is compact for the supremum norm.

Finally we will prove part (I). Let $f\in\mcl L^K(\mcl S^N)$, and for each $k\in\N$ let $f^k$ be as in Lemma \ref{prop9}. By Lemma \ref{prop10} and the Borel-Cantelli lemma we can find an increasing sequence $\{k_l\}_{l\in\BB N}$, $k_l\in\BB N$, such that $\alpha(f^{k_l})\rightarrow \alpha(f)$ a.s.\ as $l\rightarrow\infty$. By part (II) we have $\alpha(f^k)\rightarrow \wt\alpha(f)$ a.s.\ as $k\rightarrow\infty$. Part (I) now follows by the triangle inequality.
\end{proof}

\subsection{Existence and uniqueness of solutions of the continuum Schelling model}
\label{sec:welldef}
In this section we will prove Theorem \ref{prop3}, i.e., we will prove existence and uniqueness of solutions of the initial value problem \eqref{eq1}, \eqref{eq2}. 

First we will see that the theorem does not hold for all choices of initial data, i.e., there exist initial data for which \eqref{eq1} does not have a unique solution. Furthermore, solutions of \eqref{eq1} do not in general vary continuously with the initial data. We will illustrate these properties of \eqref{eq1} by considering the model for $N=1$, $M=2$, $\mcl N=\mcl N_\infty$, and $V=\mcl S$. Let $\wh Y$ be as in \eqref{eq71}. The initial data $\wh Y(\cdot,0)\equiv\ep$ and $\wh Y(\cdot,0)\equiv-\ep$ for $\ep>0$, give solutions $\wh Y(x,t)=\ep+2t$ and $\wh Y(x,t)=-\ep-2t$, respectively, so the solution does not vary continuously with the initial data. As an example of initial data for which Theorem \ref{prop3} does not hold, assume $n:=3R/4\in\BB N$, and define periodic initial data by
\eqbn
\wh Y(x,0)
= \left\{
 \begin{array}{ll}
 -1+3(x-4k/3) &\text{ for } x \in \frac 23 [2k,2k+1),\, k\in\{0,...,n-1\},\\
 1-3(x-4k/3-2/3)  &\text{ for } x \in \frac 23 [2k+1,2k+2),\, k\in\{0,...,n-1\}.
 \end{array}
  \right.
\eqen
Then we have for all $t\in [0,3/2)$,
\eqbn
\frac{\partial \wh Y}{\partial t}(x,t)
= \left\{
 \begin{array}{ll}
 2/3-2(x-4k/3) &\text{ for } x \in \frac 23 [2k,2k+1),\, k\in\{0,...,n-1\},\\
 -2/3+2(x-4k/3-2/3)  &\text{ for } x \in \frac 23 [2k+1,2k+2),\, k\in\{0,...,n-1\}.
 \end{array}
  \right.
\eqen
For $t=3/2$ we have $\wh Y(\cdot,t)\equiv 0$. If we only allow for solutions satisfying \eqref{eq1} for all $t\geq 0$, we have no solutions since \eqref{eq1} is not satisfied at $t=3/2$. If we allow the time derivative not to exist for the single time $t=3/2$, solutions are not unique, e.g.\ $\wh Y(t,x)=2(t-3/2)$ and $\wh Y(t,x)=-2(t-3/2)$ are both solutions for $t\geq 3/2$. We do not encounter these problems for the Gaussian initial data \eqref{eq2}. The problems in the examples above arise since $\wh Y(x,t)$ is close to 0 for many $x\in\mcl S^N$ simultaneously, and Theorem \ref{prop6} implies that this is not the case for the Gaussian initial data.

\begin{lem}
For any $\omega\in\Omega$ let $\phi^\omega$ be defined by \eqref{eq8}. For any $K>0$ the map $\phi^\omega|_{\mcl L_M^{K}(\mcl S^N)}:\mcl L_M^{K}(\mcl S^N)\to \mcl L_M^{1}(\mcl S^N)$ is a.s.\ continuous for the supremum norm.
\label{prop7}
\end{lem}
\begin{proof}
By the definition of $\phi^\omega$, for any $\ep>0$ and $y\in \mcl L^{K}(\mcl S^N)$ it holds a.s.\ that 
\eqb
\begin{split}
	&\sup_{\wt y\in \mcl L(\mcl S^N),\,\|\wt y-y\|_{L^\infty}<\ep} \|\phi^\omega(y)-\phi^\omega(\wt y)\|_{L^\infty}\\
	&\qquad\leq \sup_{\wt y\in \mcl L(\mcl S^N),\,\|\wt y-y\|_{L^\infty}<\ep}
	\sum_{\substack{1\leq m,m'\leq M,\\ m\neq m'}} 
	\int_{\mcl S^N} \bd 1_{
	B_{m}(x)+y_{m}(x) \geq B_{m'}(x)+y_{m'}(x);
	B_{m}(x)+\wt y_{m}(x) \leq B_{m'}(x)+\wt y_{m'}(x)
	}\,dx
	 \\
	&\qquad\leq \sum_{\substack{1\leq m,m'\leq M,\\ m\neq m'}} \int_{\mcl S^N} \bd 1_{|(B_{m'}(x)-B_m(x))-(y_{m}(x)-y_{m'}(x))|\leq 2\ep}\,dx.
\end{split}
\label{eq22}
\eqe 
We want to show that a.s., for all $y\in \mcl L^{K}(\mcl S^N)$ the right side of \eqref{eq22} converges to 0 as $\ep\rta 0$. Fix $m,m'\in\{1,...,M\}$, $m\neq m'$. The random field $\wt B(x):=B_{m'}(x)-B_m(x)$ has the law of a constant multiple of $B_m$. Also note that $y_{m}-y_{m'}\in\mcl L^{2K}(\mcl S^N)$. Assume $\ep=2^{-k}$ for $k\in\N$. On the event $G_k^1$ of Lemma \ref{prop16} (with $2K$ instead of $K$, $\wt B$ instead of $B$, and $\ep=1/1000$),
$$\int_{\mcl S^N} \bd 1_{|(B_{m'}(x)-B_m(x))-(y_{m}(x)-y_{m'}(x))|\leq 2\ep}\,dx\preceq 2^{k(N-1/2+1/1000)}\times 2^{-kN}= 2^{-k/2+k/1000}$$ 
for all $y\in \mcl L^{K}(\mcl S^N)$. The lemma now follows by Lemma \ref{prop16}.
\end{proof}

\begin{prop}
For any $\omega\in\Omega$ let $\phi^\omega$ be defined by \eqref{eq8}. For any $K>0$ the map $\phi^\omega|_{\mcl L^K_M(\mcl S^N)}:\mcl L^K_M(\mcl S^N)\to \mcl L^1_M(\mcl S^N)$ is a.s.\ Lipschitz continuous for the supremum norm.
\label{prop1}
\end{prop}
\begin{proof}
Recall the set $\wt{\mcl L}^K(\mcl S^N)$ defined by \eqref{eq70}. By Lemmas \ref{prop9} and \ref{prop7} it is sufficient to prove a.s.\ Lipschitz continuity on $\wt{\mcl L}^K(\mcl S^N)$. By \eqref{eq22} it is sufficient to prove that for any $m,m'\in\{1,...,M\}$, $m\neq m'$, and $\wt B=B_m-B_{m'}$,
\eqb
\sup_{f\in \wt{\mcl L}^{2K}(\mcl S^N)}\int_{\mcl S^N} \bd 1_{|\wt B(x)-f(x)|\leq\ep}\,dx \preceq \ep,
\label{eq23}
\eqe
where the implicit constant is independent of $\ep$, but depends on $\wt B$. The occupation times formula \eqref{eq19} implies that a.s.\ for fixed $f$ and with $\alpha(f,\cdot,\cdot)$ denoting the occupation kernel of $\wt B-f$,
\eqbn
\int_{\mcl S^N}\bd 1_{|\wt B(x)-f(x)|\leq\ep}\,dx
=\int_{-\ep}^\ep \alpha(f,a,\mcl S^N)\,da.
\eqen
For fixed $a\in\R$, by the definition of $\alpha$ it holds a.s.\ that $\alpha(f,a,\mcl S^N)=\alpha(f+a,0,\mcl S^N)$. By Theorem \ref{prop6} (I) and with $\wt\alpha$ as in this theorem, we have $\alpha(f+a,0,\mcl S^N)=\wt\alpha(f+a)$ a.s. Therefore, 
\eqbn
\int_{\mcl S^N}\bd 1_{|\wt B(x)-f(x)|\leq\ep}\,dx
=\int_{-\ep}^\ep \wt\alpha(f+a)\,da \leq 
2\ep \sup_{f\in \mcl L^{K+\ep}(\mcl S^N)}\wt \alpha(f),
\eqen
which completes the proof of the lemma upon an application Theorem \ref{prop6} (III).
\end{proof}

We will deduce Theorem \ref{prop3} from the following Banach space version of the theorem known as the Picard-Lindel\"of theorem in the theory of differential equations. The theorem is proved in the same manner as the Picard-Lindel\"of theorem, i.e., by defining a contraction mapping from the integral version of \eqref{eq12}, showing that the Picard iterates converge to a solution, and deducing uniqueness from the contraction property, see e.g.\ \cite[Lemma 4.1.6]{mra02}. The integral in (iii) is the Bochner integral.
\begin{thm}
Let $(\wh{\mcl C},\|\cdot \|)$ be a Banach space, $\wh{\mcl L}\subset\wh{\mcl C}$, $y_0\in\wh{\mcl L}$, $t_0\in\BB R$, $\Delta t>0$, and $I=[t_0-\Delta t,t_0+\Delta t]$. Let $\phi:\wh{\mcl C} \rightarrow \wh{\mcl C}$ be a map satisfying the following properties: 
\begin{itemize}
\item[(i)] $\phi$ is uniformly Lipschitz continuous on the closure of $\wh{\mcl L}$,
\item[(ii)] $\sup_{y\in\wh{\mcl L}}\|\phi(y(s) )\|<\infty$, and
\item[(iii)] $ y_0+\int_{t_0}^t\phi(y)\,ds \in \wh{\mcl L}$ for any $t\in I$ and any continuous curve $(y(s))_{t_0\leq s\leq t}$ with values in $\wh{\mcl L}$.
\end{itemize}
Then there is a unique curve $(y(t))_{t\in I}$, such that $y(t_0)=y_0$, $\frac{\partial y}{\partial t}(t)=\phi(y(t))$, and $y(t)\in \wh{\mcl L}$ for all $t\in I$.
\label{prop4}
\end{thm}
\begin{proof}
The conditions $\frac{\partial y}{\partial t}=\phi(y(t))$ and $y(t_0)=y_0$ are equivalent to the following 
\eqb
y(t) = y_0 + \int_{t_0}^{t} \phi(y(s))\,ds.
\label{eq21}
\eqe
Define $y_0(t)=y_0$ for all $t\in I$, and for $n\in\BB N$ and $t\in I$ define $y_n(t)$ by induction:
\eqb
y_{n+1}(t) = y_0 + \int_{t_0}^{t} \phi(y_n(s))\,ds.
\label{eq86}
\eqe
By assumption (iii) we have $y_n(t)\in\wh{\mcl L}$ for all $n\in \BB N$ and $t\in I$. Let $C_1\geq 0$ be the Lipschitz constant of $\phi$ on $\wh{\mcl L}$, and let $C_2\in[0,\infty)$ be defined by $C_2:= \sup_{y\in\wh{\mcl L}}\|\phi(y)\|$. Then $\|y_1(t)-y_0\|\leq C_2|t-t_0|$ for all $t\in I$ by \eqref{eq86}. By assumption (i) and induction on $n$, we get further
\eqbn
\|y_{n+1}(t)-y_n(t)\| 
\leq \int_{t_0\wedge t }^{t_0\vee t}\|\phi(y_n(s))-\phi(y_{n-1}(s))\|\,ds
\leq C_2C_1^{n} |t-t_0|^{n+1}/(n+1)!. 
\eqen
Since $\wh{\mcl C}$ is complete it follows that $(y_n(t))_{t\in I}$ converges uniformly to a curve $(y(t))_{t\in I}$, such that $y(t)\in \wh{\mcl C}$ for any $t\in I$. Further, by sending $n\rta\infty$ in \eqref{eq86} it follows by continuity of $\phi$ on the closure of $\wh{\mcl L}$ and by the dominated convergence theorem for the Bochner integral that $y$ satisfies the integral equation \eqref{eq21}. We have $y(t)\in\wh{\mcl L}$ for all $t\in I$ by assumption (iii). This concludes the proof of existence of solutions.

To obtain uniqueness of solutions let $(\wt y(s) )_{s\in I}$ for $\wt y(s)\in\wh{\mcl L}$ be another solution. By \eqref{eq21} we have $\|y_0(t)-\wt y(t)\|\leq C_2|t-t_0|$ for all $t\in I$. By assumption (i) and induction we get further that for any $n\in\BB N$,
\eqbn
\|y_{n}(t)-\wt y(t)\| 
\leq \int_{t_0\wedge t }^{t_0\vee t}\|\phi(y_n(s))-\phi(\wt y(s))\|\,ds
\leq C_2C_1^{n} |t-t_0|^{n+1}/(n+1)!.
\eqen
Letting $n\rightarrow\infty$ it follows that $y=\wt y$.
\end{proof}

Theorem \ref{prop3} follows from Theorem \ref{prop4} applied with $\wh{\mcl C}=\mcl C_M(\mcl S^N)$, $\wh{\mcl L}=\mcl L^K_M(\mcl S^N)$ for some $K>0$, and $\phi=\phi^\omega$ as defined by \eqref{eq8}. 
\begin{proof}[Proof of Theorem \ref{prop3}]
First consider the case when $V=\mcl S$. Assume the assertion of the theorem is not true, and let $\wt t\in[0,\infty)$ be the supremum of times $t\geq 0$ for which \eqref{eq12} has a unique solution in $[0,t]$. Choose an arbitrary $\Delta t>0$, let $K>\wt t+ \Delta t$, and define $t_0=\wt t-(\Delta t)/2$. We will prove that the assumptions of Theorem \ref{prop4} are a.s.\ satisfied with $\phi=\phi^\omega$, $y_0=y(t_0)$, $\wh{\mcl L}=\mcl L_M^{K}(\mcl S^N)$, $\wh{\mcl C}=\mcl C_M(\mcl S^N)$ and $I=[t_0-\Delta t,t_0+\Delta t]$. 
We equip $\wh{\mcl C}$ with the norm $\|f\|=\max_{1\leq m\leq M}\|f_m\|_{L^\infty(\mcl S^N)}$ for $f=(f_1,\dots,f_M)\in \wh{\mcl C}$. 
Condition (i) holds by Proposition \ref{prop1}. 
Condition (ii) is satisfied since $\|\phi^\omega(y)\|\leq 2^N$ for any $y\in \mcl L_M^{K}(\mcl S^N)$ (in fact, this bound holds even for $y\in \mcl C_M(\mcl S^N)$). Finally observe that (iii) holds since for any $t\in I$,
\eqbn
\left\|y(t_0)+\int_{t_0}^t\phi^\omega(y(\cdot,s) )\,ds \right\|
\leq 2^N t_0 + \Delta t 2^N,
\eqen
and since the function $y(t_0)+\int_{t_0}^t\phi^\omega(y(\cdot,s))\,ds$ is Lipschitz continuous with Lipschitz constant at most $2^{N-1}\wt t+2^{N-1}\Delta t$ in each coordinate. Theorem \ref{prop4} now implies that \eqref{eq12} has a unique solution on $I$, which is a contradiction. This completes the proof of the theorem for the case $V=\mcl S$.

Now consider the case $N=1$ and $V=\R$. First we prove uniqueness. It is sufficient to show that given any $\ep>0$ solutions are unique on $[-\ep^{-1},\ep^{-1}]\times\R_+$ with probability at least $1-\ep$. Let $R\in\{3,4,\dots \}$ be sufficiently large such that with probability at least $1-\ep$ there are real numbers $x_i$ (which are random and measurable with respect to $\sigma(B)$) for $i=1,2,3,4$ such that $-R/2+1<x_1<x_2-1<-\ep^{-1}-1<\ep^{-1}+1<x_3+1<x_4<R/2-1$, and such that $x\mapsto p(B(x))$ is constant on the intervals $[x_1,x_2]$ and $[x_3,x_4]$. Consider the Schelling model on the torus $\mcl S$ of width $R$, and let $(\wt B(x))_{x\in[-R/2,R/2]}$ be the initial values. Couple $\wt B$ and $B$ such that $\wt B|_{[-R/2+1,R/2-1]}=B|_{[-R/2+1,R/2-1]}$ a.s., and observe that if $Y:\R\times\R_+$ solves the Schelling model \eqref{eq1}, \eqref{eq2} on $\R$, then $Y|_{[x_2,x_3]}$ is a solution to the Schelling model on $\mcl S$ restricted to $[x_2,x_3]$. Here we use that if $x\mapsto p(B(x))$ is constant on an interval of length $>1$ then $Y$ (resp.\ $\wh Y$) evolves independently to the left and to the right of this interval. By uniqueness of solutions to the Schelling model on $\mcl S$, we obtain uniqueness of solutions to the Schelling model on $\R$.

Existence follows by a similar argument. Let $\ep>0$, $R\in\{3,4,\dots \}$, $\wt B$, and $x_i\in\R$ for $i=1,2,3,4$ be as in the previous paragraph. It is sufficient to prove existence of $Y$ restricted to $[x_2,x_3]$, since the real line a.s.\ can be divided into countably many disjoint intervals, such that each interval either (i) has length $>1$ and is such that $x\mapsto p(B(x))$ is constant on the interval, or (ii) is between two intervals of type (i). If we find a solution on each interval of type (ii) we can get a global solution by concatenating the solution from the different intervals, since $p(Y(x,t))$ is constant for all $t\geq 0$ and all $x$ in an interval of type (i). By existence of solutions to the Schelling model on the torus, we define $Y|_{[x_2,x_3]}$ to be equal to the solution of the Schelling model on the torus restricted to $[x_2,x_3]$, which concludes the proof.
\end{proof}

\subsection{Long-time behavior of the one-dimensional continuum Schelling model}
\label{sec:propertiescontinuum}

The main result in this section is the following proposition. 
\begin{prop} \label{prop::unionofintervals}
Let $V=\R$ or let $V=\mcl S$. Let $Y$ be the solution of the initial value problem \eqref{eq1}, \eqref{eq2} on $V$ with $M=2$, $N=1$ and $\mcl N=\mcl N_\infty$, and for $m=1,2$ define 
\eqbn
A_m := \left\{x\in\BB R\,:\,\lim_{t\rta\infty} p(Y(x,t))=m \right\}.
\eqen
Then $\R=\ol{A_1\cup A_2}$ a.s.,
the boundary $\partial A_1$ is a.s.\ equal to the boundary $\partial A_2$, and this boundary a.s.\ consists of a countable collection of points such that the distance between any two of these points is strictly greater than one $1$. 
\label{prop20a}
\end{prop} 

The proposition implies that the complement of $\partial A_1 = \partial A_2$ is a sequence of open intervals, each of length greater than $1$, which alternately belong to $A_1$ and $A_2$. We make no statement about whether the limit does or does not exist at the boundary points themselves.
\begin{remark}
	The reason the proposition is only stated for $M=2$ is that a particular form of monotonicity of \eqref{eq1} holds only for $M=2$. More precisely, if $M=2$, $Y$ solves \eqref{eq1}, \eqref{eq2}, and $\wt Y$ solves \eqref{eq1} with initial data $\wt Y^1(\cdot,0)=B+f$ and $\wt Y^2(\cdot,0)=B-f$ for a strictly positive function $f$ (with $f$ chosen such that $\wt Y$ is well-defined), then $(\wt Y^1-\wt Y^2) - (Y^1-Y^2)$ is a strictly positive function which is increasing in $t$. However, we do believe that the proposition also holds for $M>2$, and if we had established the proposition for all $M$, then Theorem \ref{thm1} would hold for all $M\in\{2,3,\dots\}$. 
	\label{remark1}
\end{remark}

We briefly outline the proof of Proposition~\ref{prop::unionofintervals} before we proceed, and we begin with some notation. We say that an opinion $m$ \emph{dominates} in an interval $J$ in the limit as $t\rta\infty$ if the fraction of $(x,t)\in J\times[0,T]$ for which $p(Y(x,t))=m$ converges to 1 as $T\rta\infty$ (equivalently, $r(m,J,T)\rta 1$ with the notation introduced below).  Intuitively, this means that individuals in the interval $J$ are (regardless of how they started out) increasingly tending to switch their opinions to $m$ in the large $T$ limit.

The first part of the argument is to show that if a certain opinion $m$ dominates in an interval $J$ of length $\geq 1$ in the limit as $t\rta\infty$, and $J$ is not contained in a larger interval satisfying this property, then the interval of length 1 immediately to the right (or left) of $J$ is dominated by some other opinion $m'$ as $t\rta\infty$. 

This result is stated in Lemma \ref{prop24} (which is in turn immediate from Lemmas~\ref{lem::interval1} and~\ref{lem::interval2} below). As explained right after Lemma \ref{prop24}, from this lemma we can deduce a weak variant of Proposition \ref{prop20a} which holds for all $M$.

We conclude the proof of Proposition \ref{prop20a} by a perturbative approach. We show that if the result of Proposition \ref{prop20a} does not hold, then it will hold for a slight perturbation of the initial data which favors one opinion more. We deduce from this that the set of initial data on which the proposition does not hold is exceptional. If we increase the initial bias towards opinion (say) 1 uniformly by $\ep$, then we can find some positive (random) number $\wt\ell(t)$ such that for any $x$ the measure of the set $\{ x'\in\mcl N(x)\,:\, p(Y(x',t))=1 \}$ increases by at least $\wt\ell(t)$. By a detailed analysis of the differential equation we can show that $\inf_{t\geq 0}\wt\ell(t)>0$, which we use to show that there exists at least one interval of length $>1$ on which the bias converges, and further (using Lemma \ref{prop24}) that this property must hold everywhere.

The following lemma is immediate from \eqref{eq1}, and will be used throughout the proof of the proposition.
\begin{lem}
Let $V=\R$ or let $V=\mcl S$. Let $Y$ be a solution of the initial value problem \eqref{eq1}, \eqref{eq2} on $V$ with $N=1$ and $\mcl N=\mcl N_\infty$. Let $t\geq 0$, let $m\in\{1,...,M\}$, and let $J\subset V$ be an interval of length $\geq 1$ such that $p(Y(x,t))=m$ for all $x\in J$. Then $p(Y(x,t'))=m$ for all $t'\geq t$ and $x\in J$.
\label{prop21c}
\end{lem}

\begin{remark} For the discrete Schelling model on $\Z$ the final configuration of opinions will always consist of intervals of length at least $w+1$ in which all nodes have the same limiting opinion. This can be seen by the following argument. If there is an interval of length $w+1$ where all nodes have the same opinion $m$, then no nodes in this interval will ever change their opinion. Furthermore, if $i$ is the first node to the right of this interval for which $m$ is not the limiting opinion, then nodes $i,i+1,\dots,i+w$ must all have the same limiting opinion; otherwise node $i$ would not be satisfied in the final configuration. Since we consider the model on $\Z$, there will always be some interval of length $w+1$ where all nodes have the same opinion in the initial configuration. By induction on the nodes to the left and right, respectively, of this interval, it follows that all nodes are contained in an interval of length at least $w+1$ in which all nodes have the same limiting opinion. The two lemmas we will prove next (which imply Lemma \ref{prop24} when combined) are continuum analogs of this result. 
\end{remark}

Lemma \ref{lem::interval1} says that if some opinion $m$ dominates in an interval $I=(a-1,a)$ of length exactly 1, but $p\circ Y$ does not converge pointwise to $m$ in $I$, then there is some opinion $m'\neq m$ which dominates in the interval $\wt I=(a,a+1)$. We give a brief outline of the proof in the simplified setting where $M=2$ and $m=1$.  Observe that for $x\in I$ and $t\gg 1$, we see from \eqref{eq71} that $\wh Y(x,t)$ is approximately equal to $t\big(r(1,I,t)-r(2,\mcl N(x)\setminus I,t)\big)$. If $r(1,I,t)$ is very close to 1, but $\wh Y(x,t)<0$, then we must have $r(2,\mcl N(x)\setminus I,t)$ very close to 1. We can use this to argue existence of $x'$ satisfying $0<x'-a\ll 1$ such that $r(2,x',t)$ is close to 1. By \eqref{eq71} and since $\mcl N(x')$ is approximately equal to $I\cup\wt I$, we see that $\wh Y(x',t)$ is approximately equal to $t\big(r(1,I,t)-r(2,\wt I,t)\big)$. We can deduce from this that $r(2,\wt I,t)$ is close to 1, so opinion 2 dominates in the interval $\wt I$.
\begin{lem}\label{lem::interval1}
Let $Y$ be a solution of \eqref{eq1} with continuous initial data (chosen such that we have existence of solutions of \eqref{eq1}), $N=1$, $\mcl N=\mcl N_\infty$, and either $V=\mcl S$ or $V=\R$. For any $m\in\{1,...,M\}$, an interval $J\subset\R$, and $t\geq 0$, define
\eqbn
r(m,J,t):=\frac{1}{|J| t} \int_0^t \int_J \1_{p(Y(x,t'))=m} \,dx\,dt'.
\eqen
For $a\in\R$ define $I:=(a-1,a)$ and $\wt I:=(a,a+1)$. Assume there exists an $m\in\{1,...,M\}$ such that $\lim_{t\rta\infty}r(m,I,t)=1$, and that there exists $x\in I $ for which the limit $\lim_{t\rta\infty} p(Y(x,t))$ either does not exist or takes a value different from $m$. Then there exists an $m'\in\{1,...,M\}$, $m'\neq m$, such that $\lim_{t\rta\infty} r(m',\wt I,t) = 1$.
\label{prop23}
\end{lem}

\begin{figure}
	\centering
	\includegraphics[scale=1]{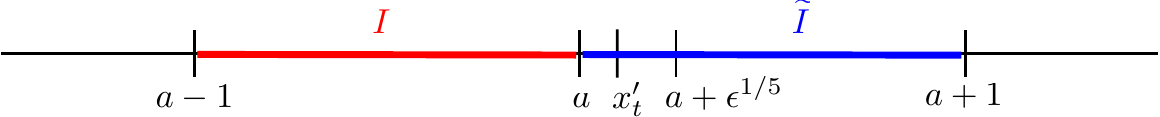}
	\caption{Illustration of objects in the statement and proof of Lemma \ref{prop23}. We assume that opinion $m$ dominates in the interval $I$, in the sense that $r(m,I,t)\rta 1$ as $t\rta\infty$. We also assume that there exists an $x \in I$ for which $p(Y(x,t))$ does not converge to $m$.
Using the latter assumption, we prove that if we pick some sufficiently small $\epsilon>0$ then for {\em every} sufficiently large $t$ (where in particular $t$ is large enough so that $r(m,I,t) > 1 - \epsilon$) there exists an $x_t $ in the slightly translated interval $(a-1+\sqrt{\epsilon}, a +\sqrt{\epsilon})$ such that $p(Y(x_t, t)) \not = m$. Using this, we will then prove existence of $m'\in\{1,\dots,M \}\setminus\{m \}$ and $x'_t\in(a,a+\epsilon^{1/5})$ for all sufficiently large $t>0$, such that $r(m',x'_t,t)>1-\epsilon^{1/5}$. Then we use the existence of $x'_t$ to prove that $\lim_{t\rightarrow\infty}r(m',\wt I,t)=1$. Note that by symmetry, we could have made the analogous argument with $\wt I$ on the left side of $I$, instead of the right side.}
\end{figure}

\begin{proof}
Define $K:=\sup_{x\in I\cup\wt I} \|Y(x,0)\|_{1}$. Let $\ep\in (0,1/10)$, and define $d_\ep:=\ep^{1/2}$ and $d'_\ep:=\ep^{1/5}$. For each $t\geq 0$ we define $x_t\in \R$ by
\eqbn
x_t := \inf\{x\geq a-1+d_\ep\,:\, p(Y(x,t)) \neq m \},
\eqen
and let $m'_t\in\{1,...,M\}\backslash\{m \}$ be such that $Y_m(x_t,t)\leq Y_{m'_t}(x_t,t)$. Note that we can find such an $m'_t$ since $Y$ is continuous.

For $m\in\{1,...,M\}$, $x\in\R$ and $t\geq 0$ define 
\eqb
r(m,x,t) :=\frac 1t \int_0^t \1_{p(Y(x,t'))=m }\,dt'.
\label{eq79}
\eqe
We abuse notation slightly by letting $r$ denote both this function and the function in the statement of the lemma.

We will prove that for all sufficiently large $t\geq 0$ there exists $x'_t\in (a,a+d'_\ep)$ satisfying $r(m'_t,x'_t,t)>1-d'_\ep$. The following relation, which follows directly from \eqref{eq1} and holds for any $m_1,m_2\in\{1,...,M\}$, will be used multiple times throughout the proof of this result
\eqb
\frac 1t\big(Y_{m_1}(x_t,t) - Y_{m_2}(x_t,t)\big)
-\frac 1t\big(Y_{m_1}(x_t,0) - Y_{m_2}(x_t,0)\big)
= r(m_1,\mcl N(x_t),t) - r(m_2,\mcl N(x_t),t).
\label{eq35}
\eqe

We consider the following three cases separately: (I) $a-1+d_\ep\leq x_t< a-1+d'_\ep$, (II) $a-1+d'_\ep\leq x_t< a$, (III) $a\leq x_t\leq a+d_\ep$. One of the cases (I)-(III) must occur by the following argument, i.e., we cannot have $x_t> a+d_\ep$. If $x_t> a+d_\ep$, we would have an interval of length $>1$, such that $p(Y(x,t))=m$ for all $x$ in the interval. Therefore, by Lemma \ref{prop21c} we would have $p(Y(x,t'))=m$ for all $t'\geq t$ and all $x$ in the interval. Since $\lim_{t\rta\infty} r(m,I,t)= 1$, and by the differential equation \eqref{eq1} and Lemma \ref{prop21c}, it would follow that $\lim_{t\rta\infty}p(Y(x,t))=m$ for all $x\in I$, which is a contradiction to the assumptions of the lemma.

First consider cases (I) and (II) defined above. Define  $I_L:=(x_t-1,a-1)$, $I_R:=(a,x_t+1)$ and $d_t:=x_t-(a-1)\geq d_\ep$. Note that $\mcl N(x_t)=\ol{I_L\cup I\cup I_R}$, $|I_R|=d_t$ and $|I_L|=1-d_t$. By \eqref{eq35}, for all sufficiently large $t$ (chosen such that $r(m,I,t)>1-\ep$, which implies $r(m'_t,I,t)<\ep$)
\eqbn
\begin{split}
|I_R|r(m'_t,I_R,t) 
=&\,
\frac 1t \Big(Y_{m'_t}(x_t,t)-Y_{m}(x_t,t)\Big) 
- \frac 1t \Big(Y_{m'_t}(x_t,0)-Y_{m}(x_t,0)\Big)
- r(m'_t,I,t) + r(m,I,t) \\
& - |I_L| r(m'_t,I_L,t) + |I_L|r(m,I_L,t)
  + |I_R| r(m,I_R,t)  \\
>&\, 0 - 2K/t - \ep + (1-\ep) - (1-d_t) + 0 + 0\\
=&\, -2K/t - 2\ep + d_t.
\end{split}
\eqen
If there is no appropriate $x'_t$ and case (I) occurs,
\eqbn
d_t(1- d'_\ep) 
\geq d_t \sup_{x\in[a,a+d_t]} r(m'_t,x,t)
\geq |I_R| r(m'_t,I_R,t)
> -2K/t - 2\ep + d_t,
\eqen
which is a contradiction for sufficiently large $t$, since $d_t d'_\ep \geq 2K/t + 2\ep$ for all large $t$. If there is no appropriate $x'_t$ and case (II) occurs, 
\eqbn
	d_t - (d'_\ep)^2
	\geq d'_\ep \sup_{x\in[a,a+d'_\ep]} r(m'_t,x,t) 
	+ (d_t-d'_\ep) \sup_{x\in[a+d'_\ep ,a+d'_t]} r(m'_t,x,t)
	\geq |I_R| r(m'_t,I_R,t) 
	\geq 2 K/t -2 \ep + d_t,
\eqen
which is a contradiction for sufficiently large $t$, since $d'_\ep d'_\ep \geq K/t + \ep$ for all large $t$. 

In case (III) define $d_t:=x_t-a<d_\ep$, $I_L:=I\cap \mcl N(x_t)=(x_t-1,a)$ and $I_R:=\mcl N(x_t)\backslash I_L=(a,x_t+1)$. Note that $|I_L|=1-d_t$ and $|I_R|=1+d_t$. By \eqref{eq35}, for any sufficiently large $t$ (such that $|I_L|r(m,I_L,t)>1-d_t-\ep$)
\eqbn
\begin{split}
|I_R| r(m'_t,I_R,t) 
=&\,\frac 1t \Big(Y_{m'_t}(x_t,t)-Y_{m}(x_t,t)\Big) 
- \frac 1t \Big(Y_{m'_t}(x_t,0)-Y_{m}(x_t,0)\Big)\\
&- |I_L| r(m'_t,I_L,t) + |I_L|r(m,I_L,t)
+ |I_R| r(m,I_R,t)  \\
\geq&\, -2K/t + 0 - \ep + (1-d_t-\ep) + 0,
\end{split}
\eqen
so if there is no appropriate $x'_t$,
\eqbn
\begin{split}
1 + d_t - (d'_\ep)^2
&= d'_\ep(1-d'_\ep) + (1+d_t-d'_\ep) \\ 
&\geq d'_\ep \sup_{x\in [a,a+d'_\ep]} r(m'_t,x,t) 
+ (1+d_t-d'_\ep) \inf_{x\in I_R\backslash [a,a+d'_\ep]} r(m'_t,x,t) \\
&\geq |I_R| r(m'_t,I_R,t) \\
&\geq -2K/t - 2\ep - d_t + 1,
\end{split}
\eqen
which implies $(d'_\ep)^2 \leq 2K/t+2\ep+ 2d_t$. This is a contradiction for sufficiently large $t$. We conclude that an appropriate $x'_t$ exists for all large $t$ in all cases (I)-(III).

For any $t\geq 0$ define
\eqbn
S^\ep_{t} := \{s\in (K/\ep,\infty)\,:\, r(m'_t,\wt I,s)< 1-3\ep-3d'_\ep,\, r(m,I,s)>1-\ep \}.
\eqen
We will prove that $|S^\ep_{t}\cap[0,t]|/t < d'_\ep$ for all sufficiently large $t>0$. For any $t>0$ sufficiently large such that $x'_t$ exists and $r(m,I,s)>1-\ep$, and for any $s\in S^\ep_{t}$ it follows from \eqref{eq35} that
\eqbn
\begin{split}
\frac 1s \big( Y_m(x'_t,s) - Y_{m'_t}(x'_t,s )\big)
=&\, \frac 1s \big( Y_m(x'_t,0) - Y_{m'_t}(x'_t,0 ) \big)
+ (x'_t-a)r(m,[a+1,x'_t+1],s )\\ 
&- (x'_t-a)r(m'_t,[a+1,x'_t+1],s )
+ r(m,\wt I,s ) - r(m'_t,\wt I,s )\\
&+ (1+a-x'_t) r(m,[x'_t-1,a],s ) - (1+a-x'_t) r(m'_t,[x'_t-1,a],s ) \\
\geq&\, -2K/s + 0 - d'_\ep + 0 - (1-3\ep-3d'_\ep) + (1-\ep-d'_\ep) - \ep \\
>&\, 0,
\end{split}
\eqen
where we used the following estimates to obtain the first inequality 
\eqbn
\begin{split}
	(1+a-x'_t) r(m,[x'_t-1,a],s)
	=  r(m,I,s) - (x'_t-a)r(m,[a-1,x'_t-1],s)
	\geq (1-\ep)-d'_\ep,\\
	(1+a-x'_t) r(m'_t,[x'_t-1,a],s)
	= r(m_t,I,s) - (x'_t-a)r(m_t,[a-1,x'_t-1],s)
	\leq (1 - r(m,I,s))-0
	\leq \ep.
\end{split}
\eqen
Therefore $Y_m(x'_t,s) > Y_{m'_t}(x'_t,s )$ for all sufficiently large $t$ and $s\in S^\ep_{t}$, and it follows from the definition of $r$ that $r(m'_t,x'_t,t) \leq 1- |S^\ep_{t}\cap[0,t]|/t$ for all sufficiently large $t$. Since $1-d'_\ep < r(m'_t,x'_t,t)$ by definition of $x'_t$ it follows that $|S^\ep_{t}\cap[0,t]|/t < d'_\ep$ for all sufficiently large $t$. 

Note that if $\ep$ is sufficiently small, and $t,t'\geq 0$ are such that $m'_t\neq m'_{t'}$, then it follows from the definition of $S^\epsilon_t$ and $S^\epsilon_{t'}$ that $(S^\ep_{t})^c\cap (S^\ep_{t'})^c\subset \{ s\geq 0\,:\,r(m,I,s)\leq 1-\ep \}$. Therefore the estimate $|S^\ep_{t}\cap[0,t]|/t < d'_\ep$ for all sufficiently large $t$ and the assumption $\lim_{s\rta\infty}r(m,I,s)=1$ imply that there is an $m'\in\{1,...,M\}$ such that $m'_t=m'$ for all sufficiently large $t$. Define 
\eqbn
S'_\ep:= \{s\in (K/\ep,\infty)\,:\, r(m',\wt I,s)< 1-3\ep-3d'_\ep \}, 
\eqen
and note that $|S'_{\ep}\cap[0,t]|/t < d'_\ep$ for all sufficiently large $t$.

Let $\wt\ep>0$. In order to complete the proof of the lemma it is sufficient to show that the set of $t'>0$ such that $r(m',\wt I,t')<1-\wt\ep$ is bounded from above. Let $t'>0$ be such that $r(m',\wt I,t')<1-\wt\ep$. Choose $\ep>0$ such that $\wt\ep>10d'_\ep$. By definition of $r$, for any $t\in[(1-\frac{\wt\ep}{10})t',t']$ we have $r(m',\wt I,t)<1-\frac 12 \wt\ep$. By definition of $\ep$ we have $3\ep+3d'_\ep< \frac 12 \wt\ep$, so $r(m',\wt I,t)<1-3\ep-3d'_\ep$ for all $t\in[(1-\frac{\wt\ep}{10})t',t']$. By definition of $S'_\ep$ this implies that $|S'_\ep \cap [0,t']|/t'>\wt\ep/10$. On the other hand we know from the preceding paragraph that $|S'_{\ep}\cap[0,t]|/t < d'_\ep<\wt\ep/10$ for all sufficiently large $t$, which completes the proof of the lemma.
\end{proof}

\begin{lem}\label{lem::interval2}
Let $Y$ be the solution of the initial value problem \eqref{eq1}, \eqref{eq2} for $N=1$, $\mcl N=\mcl N_\infty$, and either $V=\mcl S$ or $V=\R$. Assume $a\in\R$ and $m\in\{1,\dots,M\}$ are such that $\lim_{t\rta\infty}p(Y(x,t))=m$ for all $x\in I:=(a-1,a)$, and that for $x>a$ arbitrarily close to $a$, either the limit $\lim_{t\rta\infty}p(Y(x,t))$ does not exist or $\lim_{t\rta\infty}p(Y(x,t))\neq m$. Then there is an $m'\in\{1,\dots,M\}$, $m'\neq m$, such that, in the notation of Lemma \ref{prop23}, we have $\lim_{t\rta\infty} r(m',\wt I,t)=1$ for $\wt I:=(a,a+1)$.
\label{prop29} 
\end{lem}
\begin{proof}
For each $t\geq 0$ we can find an $x_t\in \R$ such that $\lim_{t\rta\infty}x_t= a$ and $m_t:=p(Y(x_t,t))\neq m$. By the identity \eqref{eq35} and letting $o_t(1)$ denote a term which converges to 0 as $t\rta\infty$,
\eqbn
\begin{split}
	0 \leq&\, \frac 1t(Y_{m_t}(x_t,t)- Y_{m}(x_t,t))\\
	=&\, \frac 1t(Y_{m_t}(x_t,0)- Y_{m}(x_t,0)) 
	+ |\mcl N(x_t) \cap I|\cdot r(m_t,\mcl N(x_t) \cap I,t)
	+ |\mcl N(x_t) \cap \wt I|\cdot r(m_t,\mcl N(x_t) \cap \wt I,t)\\
	&- |\mcl N(x_t) \cap I|\cdot r(m,\mcl N(x_t) \cap I,t)
	- |\mcl N(x_t) \cap \wt I|\cdot r(m,\mcl N(x_t) \cap \wt I,t)+o_t(1)\\
	= &\, r(m_t,I,t)+r(m_t,\wt I,t)-r(m,I,t)-r(m,\wt I,t)+o_t(1).
\end{split}
\eqen
Since $\lim_{t\rta\infty}r(m,I,t)=1$, which implies $\lim_{t\rta\infty}r(m_t,I,t)=0$, it follows that $\lim_{t\rta\infty}r(m_t,\wt I,t)=1$ and $\lim_{t\rta\infty}r(m,\wt I,t)=0$. Since $\sum_{k=1}^M r(k,\wt I,t)=1$ this implies further that there is an $m'\in\{1,...,M\}$, $m'\neq m$, such that $m_t=m'$ for all sufficiently large $t>0$. It follows that $\lim_{t\rta\infty} r(m',\wt I,t)=1$.
\end{proof}

The following lemma is immediate from Lemmas \ref{prop23} and \ref{prop29}. 
\begin{lem}
	Let $V=\mcl S$ or $V=\R$, and consider the initial value problem \eqref{eq1}, \eqref{eq2} for $N=1$ and $\mcl N=\mcl N_\infty$. Assume $m\in\{1,\dots,M \}$ and $I=[a_1,a_2]\subset\R$ is an interval of length $\geq 1$ such that $\lim_{t\rta\infty} r(m,I,t)=1$, and such that $I$ is not contained in any larger interval satisfying this property. Then there is an $m'\in\{1,\dots,M \}\setminus \{m \}$ such that $\lim_{t\rta\infty} r(m',[a_2,a_2+1],t)=1$.
	\label{prop24}
\end{lem}

We can deduce a weak version of Proposition \ref{prop20a} for the case $V=\R$ from Lemma \ref{prop24}. This weak version of Proposition \ref{prop20a} is the last result of the section which holds also for $M>2$. By Lemma \ref{prop21c} we know that there will be some interval $I$ satisfying the conditions of Lemma \ref{prop24}. Lemma \ref{prop24} therefore says that all $x\in\R$ will be contained in an interval $J\subset\R$ of length $\geq 1$, such that for some $m'\in\{1,\dots,M \}$, we have $\lim_{t\rta\infty} r(m',J,t)=1$. In particular, both this lemma and Proposition \ref{prop20a} say that the limiting states of the continuum Schelling model on $\R$ can be divided into intervals of length $\geq 1$ such that each interval is associated with a particular limiting opinion. The lemma is weaker than Proposition \ref{prop20a} in two ways: First, we do not prove that each interval has length strictly larger than 1, and second, instead of proving that $p\circ Y$ converges pointwise on the intervals we prove a weaker result expressed in terms of the function $r$. Both these stronger properties are needed when we apply Proposition \ref{prop20a} in our proof of Theorem \ref{thm1}.

For $x,x'\in\mcl S$ we define $|x-x'|$ to be equal to $\inf_{k\in\N}|x-x'-Rk|$ when we identify $\mcl S$ with the interval $[0,R)$.
\begin{lem}
	Consider the initial value problem \eqref{eq71} with $N=1$, $\mcl N=\mcl N_\infty$, and either $V=\mcl S$ or $V=\R$, but with the following perturbed initial data for some $K>0$
	\eqb
	\begin{split}
		\wh Y(x,0) &= \wh B(x) + \ep\lambda([-(K+1),K+1]\cap[x-1,x+1]) \qquad\text{for\,\,\,}V=\R,\\
		\wh Y(x,0) &= \wh B(x) +\ep \qquad\text{for\,\,\,}V=\mcl S.
	\end{split}
	\label{eq64}
	\eqe	
	For $\ep_1>\ep_2>0$ let $\wh Y^{\ep_1}$ and $\wh Y^{\ep_2}$ denote the solution of \eqref{eq71}, \eqref{eq64} with $\ep=\ep_1$ and $\ep=\ep_2$, respectively. For $V=\R$ let $J\subset [-K,K]$ be an interval which may depend on $\wh B$, $\ep_1$, and $\ep_2$, and for $V=\mcl S$ let $J=\mcl S$. For $t\geq 0$ define
	\eqbn
	\begin{split}
		h(t) := \inf\{ \wh Y^{\ep_1}(x,t)- \wh Y^{\ep_2}(x,t) \,:\,x\in J \},\qquad
		\ell(t) = \inf\{ |x-x'|\,:\,x,x'\in J,\,\wh Y^{\ep_1}(x,t)<0,\,\wh Y^{\ep_2}(x,t)>0 \},
	\end{split}
	\eqen
	where the infimum over the empty set is defined to be $\infty$.
	There are $c_0,c_1>0$ depending on $\wh B$, $\ep_1$, and $\ep_2$, such that
	\eqbn
	\begin{split}
		\ell(t) \geq \max\left\{ \frac{h(t)-c_1}{2t}; \min\left\{ c_0; \frac{h(t)}{2.01 t} \right\}  \right\}.
	\end{split}
	\eqen
	\label{prop32}
\end{lem}
\begin{proof}
 For all $x\in J$ we have $\wh Y^{\ep_1}(x,0)\geq\wh Y^{\ep_2}(x,0)+(\ep_1-\ep_2)$. Since the right side of the differential equation \eqref{eq71} is monotone in $Y$, this implies that the function $t\mapsto \wh Y^{\ep_1}(x,t)-\wh Y^{\ep_2}(x,t)$ is increasing for each $x\in J$. 
 
 Define $c_1:=\sup\{|\wh B(x)-\wh B(x')|\,:\,x,x'\in J \}$. Let $\wh y,\wt y:V\times\R_+\to \R$ be such that $\wh Y^{\ep_1}(x,t)=\wh Y^{\ep_1}(x,t)+\wh y(x,t)$ and $\wh Y^{\ep_2}(x,t)=\wh Y^{\ep_2}(x,t)+\wt y(x,t)$, and recall that $\wh y(\cdot,t)$ and $\wt y(\cdot,t)$ are Lipschitz continuous with constant $2t$ by Theorem \ref{prop3}. For $x,x'\in J$ and $t\geq 0$ satisfying $\wh Y^{\ep_2}(x,t)\geq 0$ and $|x-x'|\leq \frac{h(t)-c_1}{2t}$, we have
 \eqbn
 \begin{split}
 	\wh Y^{\ep_2}(x',t) &= (\wh B(x')-\wh B(x)) 
 	+ ( \wh y(x',t) - \wh y(x,t ))
 	+ ( \wh Y^{\ep_1}(x,t)-\wh Y^{\ep_2}(x,t) ) 
 	+ \wh Y^{\ep_2}(x,t)\\
 	&\geq -c_1 - 2t|x-x'| + h(t) + 0 \geq 0,
 \end{split}
 \eqen
 which implies $\ell(t) \geq  \frac{h(t)-c_1}{2t}$. 
 	
 Define $c_0:=\big(\frac{\ep_1-\ep_2}{1000c_2}\big)^3$, where $c_2$ is the $1/3$-H\"older constant for $\wh B$ on $J$, and observe that $c_0\leq \big(\frac{h(t)}{1000c_2}\big)^3$ for all $t\geq 0$, so $c_2 c_0^{1/3}\leq \frac{h(t)}{1000}$. For $x,x'\in J$ satisfying $\wh Y^{\ep_2}(x)\geq 0$ and $|x-x'|\leq \min\big\{ c_0; \frac{h(t)}{2.01 t} \big\}$, we have
 	\eqbn
 	\begin{split}
 		\wh Y^{\ep_1}(x',t) &= (\wh B(x')-\wh B(x)) 
 		+ ( \wh y(x',t) - \wh y(x,t ))
 		+ ( \wh Y^{\ep_1}(x,t)-\wh Y^{\ep_2}(x,t) ) 
 		+ \wh Y^{\ep_2}(x,t)\\
 		&\geq -c_2|x-x'|^{1/3} - 2t|x-x'| + h(t) + 0 \geq -\frac{h(t)}{1000} -2t\frac{h(t)}{2.01 t} + h(t)\geq 0,
 	\end{split}
 	\eqen
which implies $\ell(t) \geq \min\big\{ c_0; \frac{h(t)}{2.01 t} \big\}$. Combining the above two bounds for $\ell(t)$ we obtain the lemma.
\end{proof}

\begin{figure}
	\centering
	\includegraphics[scale=1]{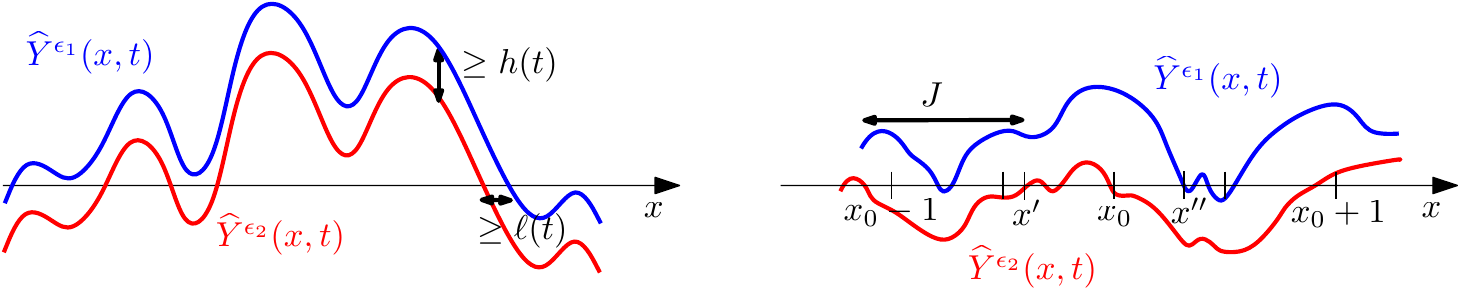}
	\caption{For $\ep_1>\ep_2>0$ we consider two solutions $\wh Y^{\ep_1}$ and $\wh Y^{\ep_2}$ of \eqref{eq71} with initial data perturbed by $\ep_1$ and $\ep_2$, respectively. We prove Proposition \ref{prop20a} by showing that the event considered in the proposition must occur for at least one of $\wh Y^{\ep_1}$ and $\wh Y^{\ep_2}$. Left: Illustration of $h$ and $\ell$ defined in Lemma \ref{prop32}. Right: Illustration of the proof of Proposition \ref{prop20a}, case (b).}
	\label{fig2}
\end{figure}

\begin{lem}
	Consider the initial value problem \eqref{eq71} with $N=1$, $\mcl N=\mcl N_\infty$, and either $V=\mcl S$ or $V=\R$. Let $E$ be the event that (in the notation of Proposition \ref{prop20a}), $\R\setminus (A_1\cup A_2)$ has measure zero, and each set $A_1$ and $A_2$ can be written as the union of intervals of length $>1$. Let $E'$ be the event that at least one of the sets $A_1$ and $A_2$ contains an interval of length $> 1$. Then $\P[E'\cap E^c]=0$.
	\label{prop31}
\end{lem}
\begin{proof}
	For $\ep>0$ and fixed $K>1$ consider the initial value problem \eqref{eq71} with perturbed initial data \eqref{eq64}. Let $A_1(\ep)$ and $A_2(\ep)$ be defined just as the sets $A_1$ and $A_2$, respectively, in the statement of Proposition \ref{prop20a}. Let $\wt E(\ep)$ be the event that the following two properties hold: (i) if $V=\mcl S$ then at least one of the sets $A_1$ and $A_2$ contains an interval of length $>1$, and if $V=\R$ then at least one of the sets $A_1\cap[-K,K]$ and $A_2\cap[-K,K]$ contains an interval of length $>1$, and (ii) the origin is not contained in an interval $I$ of length $>1$ such that $\lim_{t\rta\infty} \op{sign} Y(x,t)$ exists and is equal for all $x\in I$.
	Since $K$ was arbitrary and by translation invariance in law in the space variable, in order to complete the proof of the lemma it is sufficient to show that $\P[\wt E(0)]=0$.
	
	First we will reduce the lemma to proving 
	\eqb
	\P[ \wt E(\ep_1)\cap \wt E(\ep_2) ]=0\qquad\forall \ep_1>\ep_2>0.
	\label{eq80} 
	\eqe
	Assume $\P[\wt E(0)]>0$. Observe that the initial data \eqref{eq64} are absolutely continuous with respect to the initial data with $\ep=0$.
	By absolute continuity, we can find a sequence $(\ep_n)_{n\in\N}$ and $p>0$ such that $\ep_{n_1}\neq\ep_{n_2}$ for $n_1\neq n_2$ and $\P[\wt E(\ep_n)]>p$ for all $n\in\N$. If we assume \eqref{eq80} holds this leads to a contradiction, since for any $n\in\N$,
	\eqbn
	\P\left[ \bigcup_{k=1}^n \wt E(\ep_k) \right] = \sum_{k=1}^{n} \P[\wt E(\ep_n)] \geq np,
	\eqen
	which converges to $\infty$ as $n\rta\infty$. We conclude that the lemma will follow once we have established \eqref{eq80}.
	
	Let $\wh Y^{\ep_1}$ (resp.\ $\wh Y^{\ep_2}$) denote the solution of \eqref{eq71} with perturbed initial data for $\ep=\ep_1$ (resp.\ $\ep=\ep_2$). We will assume that both events $\wt E(\ep_1)$ and $\wt E(\ep_2)$ occur, and want to derive a contradiction. For an interval $I\subset V$ and $t\geq 0$ define 
	\eqbn
	r(I,t) := \frac{1}{\lambda(I)t} \int_{0}^{t}\int_I \op{sign}(\wh Y^{\ep_2}(x,t'))\,dx\,dt'.
	\eqen
	By Lemma \ref{prop24} and by absolute continuity of the initial data, we know that on the event $\wt E(\ep_2)$ there is almost surely an interval $I$ of length $\geq 1$ containing the origin, such that either $\lim_{t\rta\infty}r(I,t)=1$ or $\lim_{t\rta\infty}r(I,t)=-1$.

	Without loss of generality we assume $\lim_{t\rta\infty}r(I,t)=-1$; the case $\lim_{t\rta\infty}r(I,t)=1$ can be treated similarly. We also assume that $I$ is the maximal open interval satisfying this property, and let $a$ be the left end-point of the interval. By Lemma \ref{prop24}, $\lim_{t\rta\infty}r([a-1,a],t)=1$.
	
	If $\lambda(I)>1$, it is immediate from the definition of $r$ and \eqref{eq71} that $\lim_{t\rta\infty}\op{sign} \wh Y^{\ep_2}(x,t)=-1$ for all $x\in I$, which contradicts (ii) in the definition of $E(\ep_2)$. Therefore we will assume $\lambda(I)=1$. To conclude the proof of the lemma it is sufficient to derive a contradiction to the occurrence of $E(\ep_1)$.
	
	Let $J=[a-1/2,a+1/2]$, and let $h$ and $\ell$ be as in Lemma \ref{prop32}. Let $\mcl T\subset\R_+$ be the set of times $t\geq 0$ for we can find a $x'\in[a-1/10,a+1/10]$ such that $\wh Y^{\ep_2}$ takes both positive and negative values arbitrarily close to $x'$. Since $\lim_{t\rta\infty}r([a,a+1/2],t)=-1$ and $\lim_{t\rta\infty}r([a-1/2,a],t)=1$, we have $\lim_{t\rta\infty}\lambda(\mcl T\cap[0,t])/t\rta 1$. Also observe that for $t\in\mcl T$,
	\eqbn 
	\lambda(\{x\in J\,:\,\wh Y^{\ep_1}(x,t)>0,\,\wh Y^{\ep_2}(x,t)<0 \})\geq\ell(t)\wedge (1/4).
	\eqen
	By \eqref{eq71} we get further that for all $t\in\mcl T$,
	\eqbn 
	\frac{dh}{dt}(t)\geq 
	2\lambda(\{x\in J\,:\,\wh Y^{\ep_1}(x,t)>0,\,\wh Y^{\ep_2}(x,t)<0 \})\geq
	(2\ell(t))\wedge (1/2).
	\eqen
	 Since $\ell(t)\geq \min\left\{ c_0; \frac{h(t)}{2.01 t} \right\}$, this implies that $\lim_{t\rta\infty}h(t)=\infty$.
	
	By the lower bound for $\ell$ in Lemma \ref{prop32} we also have $\frac{dh}{dt}(t)\geq 2\ell(t)\geq \frac{h(t)-c_1}{t}$ for all $t\in\mcl T$, and since $\lim_{t\rta\infty}h(t)=\infty$ this implies that $h(t)\geq ct$ for some random constant $c>0$ and all $t\geq 0$. Lemma \ref{prop24} implies that $\lim_{t\rta\infty}r([a-1,a]\cup[a+1,a+2],t)=1$. By this result, \eqref{eq71}, and  $\lim_{t\rta\infty}r(I,t)= -1$,
	\eqbn
	\inf_{x\in I} \frac{1}{t} \wh Y^{\ep_2}(x,t) 
	\geq \inf_{x\in I} \frac{1}{t} \wh Y^{\ep_2}(x,0)
	+ r(I,t) + \inf_{I'\subset [a-1,a]\cup[a+1,a+2]\,:\,\lambda(I')=1} r(I',t)
	\rta 0\text{\,\,\,\,\,as\,\,\,}t\rta\infty.
	\eqen
	It follows that $\inf_{x\in I}\wh Y^{\ep_1}(x,t)\geq \inf_{x\in I}\wh Y^{\ep_1}(x,t)+h(t)>0$ for all sufficiently large $t$. This implies that $\inf_{x\in [a-1,a+2]}\wh Y^{\ep_1}(x,t)>0$ for all sufficiently large $t$, which is a contradiction to the condition (ii) in the definition of $E'(\ep_1)$.
\end{proof}

\begin{proof}[Proof of Proposition \ref{prop20a}]
	For the case when $V=\R$ the proposition follows immediately from Lemma \ref{prop31}, since  $\P[E']=1$ by Lemma \ref{prop21c}.
	
	Let $V=\mcl S$. For $\ep>0$ let $E'(\ep)$ be the event of Lemma \ref{prop31}, but for the perturbed initial data. It is sufficient to prove that for any $\ep_1>\ep_2>0$, we have $\P[E'(\ep_1)^c\cap E'(\ep_2)^c]=0$, since this implies that $\P[(E')^c]=0$ (see the argument in the second paragraph in the proof of Lemma \ref{prop31} for a similar argument). We assume that both $E'(\ep_2)^c$ and $E'(\ep_1)^c$ occur, and will derive a contradiction.
	
	First we will argue that the following inequality holds
	\eqb
	\frac{dh}{dt} \geq 4\ell(t).
	\label{eq82}
	\eqe
	By \eqref{eq71} we have
	\eqb
	\frac{dh}{dt} \geq 
	2 \inf_{x\in\mcl S} \lambda( \{ x'\in\mcl N(x)\,:\,\wh Y^{\ep_1}(x',t)>0,\,\wh Y^{\ep_2}(x',t)<0 \} ).
	\label{eq83}
	\eqe
	Since $E'(\ep_1)^c$ occurs, each closed interval $I'\subset\mcl S$ of length $\geq 1$ must intersects some interval of length $\geq 2\ell(t)$ on which $\wh Y^{\ep_2}(x,t)<0$; otherwise we would have $\wh Y^{\ep_1}(x,t)>0$ on $I'$. For some fixed $x_0\in\mcl S$ the interval $[x_0-1/2,x_0+1/2]$ therefore intersect some interval of length $\geq 2\ell(t)$ on which $\wh Y^{\ep_2}(x,t)<0$. Let $J$ be an interval satisfying this property and which is maximal in the sense that $J$ is not contained inside any other interval satisfying the property. We have either (a) $J\subset (x_0-1,x_0+1)$, (b) $\{x_0-1,x_0+1\}\cap J\neq\emptyset$. If case (a) occurs, then the following holds by the definition of $\ell$
	\eqb
	\lambda( \{ x'\in \mcl N(x_0)\,:\,\wh Y^{\ep_1}(x',t)>0,\,\wh Y^{\ep_2}(x',t)<0 \} )\geq 2\ell(t).
	\label{eq84}
	\eqe
	Next assume case (b) occurs, and without loss of generality assume  $(x_0-1)\in J$. See Figure \ref{fig2} for an illustration. We have $\ell(t)\leq 1/2$; otherwise $E'(\ep_2)^c$ and $E'(\ep_1)^c$ cannot both occur, since $\wh Y^{\ep_1}$ would be positive on any maximal interval of length $<1$ on which $\wh Y^{\ep_2}$ is negative. Let $x'\in[x_0-1/2,x_0)$ be the right end-point of $J$, and observe that $\wh Y^{\ep_1}$ is positive and $\wh Y^{\ep_2}$ negative on the interval $(x'-\ell(t),x')$. Also observe that $(x'-\ell(t),x')\subset\mcl N(x_0)$ since $x'\in[x_0-1/2,x_0)$ and $\ell(t)\leq 1/2$. 
	
	Define $x'':=\inf\{ x>x'\,:\,\wh Y^{\ep_1}(x,t)<0 \}$. We have $x''<x'+1-\ell(t)<x_0+1-\ell(t)$, where the first inequality holds since $\wh Y^{\ep_1}(x',t)>0$ and $E'(\ep_1)^c$ occurs. We also have $x''>x'+2\ell(t)$; otherwise we would have $\wh Y^{\ep_2}<0$ on $(x'-\ell(t),x'+\ell(t))$ by the definition of $\ell(t)$, which contradicts the definition of $x'$. It follows that $\wh Y^{\ep_1}$ is positive and $\wh Y^{\ep_2}$ negative on the interval $(x''-\ell(t),x'')\subset\mcl N(x_0)$. Since the intervals $(x'-\ell(t),x')$ and $(x''-\ell(t),x'')$ are disjoint, we see that \eqref{eq84} holds also in case (b). We obtain \eqref{eq82} by combining \eqref{eq83} and \eqref{eq84}.
	
	Next we argue that $h(t)\geq 2c_0 t$ for all $t\geq 0$. Let $\tau=\inf\{ t\geq 0\,:\, h(t)\leq 2c_0t \}$. We see that $\tau=\infty$, since $h(0)>0$, and since \eqref{eq82} and Lemma \ref{prop32} imply that for $t\in[0,\tau]$,
	\eqbn
	\frac{dh}{dt} \geq 4\min\left\{c_0;\,\frac{h(t)}{2.01t} \right\} \geq \frac{8}{2.01}c_0.
	\eqen
	We conclude that $h(t)\geq 2c_0 t$ for all $t\geq 0$. In particular, we have $\lim_{t\rta\infty} h(t) = \infty$, so by Lemma \ref{prop32} we have $\ell(t)\geq \frac{h(t)-c_1}{2t}\geq \frac{h(t)}{2.01t}$ for all sufficiently large $t$, which implies by \eqref{eq82} that $\frac{dh}{dt}\geq \frac{4h(t)}{2.01t}$ for all sufficiently large $t$. Further we get $h(t)\geq ct^{4/2.01}$ for some random constant $c>0$. The function $h$ cannot grow superlinearly, since $\wh Y^{\ep_2}$ and $\wh Y^{\ep_1}$ grow at most linearly in time. Therefore we obtain a contradiction, which concludes the proof.
\end{proof}

\section{The discrete Schelling model}
\label{sec:discrete}

\subsection{The early phase of the discrete Schelling model}
\label{sec:discrete1}
The main purpose of this section is to prove the following proposition, which says that the solution of the initial value problem \eqref{eq1}, \eqref{eq2} describes the early phase of the discrete Schelling model well for large $w$. Recall that we rescaled time by  $w^{-N/2}$ when we defined $Y^w$. Therefore the proposition only provides information about times up to order $w^{-N/2}$ for the discrete Schelling model.
\begin{prop}
	Let $N\in\N$ and $V=\mcl S$, or let $N=1$ and $V=\R$. Let $M\in\{2,3,\dots\}$, let $\mcl N$ be as defined in Section \ref{sec:intro1}, let $Y$ be the solution of the initial value problem \eqref{eq1}, \eqref{eq2}, and let $Y^w$ be given by \eqref{eq24}. There is a coupling of $Y$ and $Y^w$ for $w\in\N$ such that a.s.\ $Y^w$ converges uniformly to $Y$ on compact subsets of $V^N\times\R_+$ as $w\rta\infty$.	
\label{prop18}
\end{prop}
First we observe that the initial values for the discrete and continuum Schelling model can be coupled. 
\begin{lem}
	Let $V=\mcl S$ or $V=\R$. There is a uniquely defined continuous centered Gaussian $(N,M)$-random field $B$ on $V^N$ with covariances given by \eqref{eq20}. The initial data $Y^w(\cdot,0)$ of the normalized discrete bias function defined by \eqref{eq24} converges in distribution to $B$.
\label{prop19}
\end{lem}
\begin{proof} 
We first define a smoothed version $\check Y^w\in \mcl C_M(V^N)$ of $Y^w(\cdot,0)$, and prove convergence of $\check Y^w$ to $B$.  For any $i\in \Z^N$ let $A(i)$ denote the square of side length 1 centered at $i$, and for $i\in\Z^N$ let $w\mcl N+i=\{x\in\R^N\,:\, w^{-1}(x-i)\in\mcl N \}$. If $V=\mcl S$ (resp.\ $V=\R$) define $\mcl V=\frk S$ (resp.\ $\mcl V=\Z$). For $m\in\{1,\dots,M \}$ define the smoothed unscaled bias function $\check{\mcl Y}_m:\mcl V\times\R_+\to\R$ by
	\eqbn
	\check{\mcl Y}_m(i,t) = \sum_{j\in\mcl V} \lambda(A(j)\cap (w\mcl N+i)) \left(\1_{X(j,t)=m} - \frac{1}{M}\right),\quad i\in\mcl V,\,\,t\geq 0.
	\eqen
	Then define $Y^w_m$ by \eqref{eq24} and \eqref{eq60}, but using $\check{\mcl Y}$ instead of $\mcl Y$. 
	It follows by e.g.\ \cite{ap-clt} that for each $m\in\{1,\dots,M\}$, $\check Y^w_{m}$ converges in law to $B_{m}$ in $\mcl C(V^N)$. 
	Note in particular that the entropy integral considered in \cite{ap-clt} is finite as required, since $\partial\mcl N$ has an upper Minkowski dimension strictly smaller than 2.
	By this convergence result, we see that the law of $Y^w(\cdot,0)$ is tight in $\mcl C_M(V^N)$. 
	By convergence of the finite dimensional distributions and hence uniqueness of the limit, we get that $(B(x))_{x\in V^N}$ exists and that $Y^w(\cdot,0)$ converges in law to $B$ in $\mcl C_M(V^N)$. Continuity of $B$ follows e.g.\ by applying the Kolmogorov-Chentsov theorem as in the construction of Brownian motion, see the proof of Lemma \ref{prop15}.

To conclude the proof it is sufficient to show that $Y^w_m(\cdot,0)-\check Y^w_m$ converges in law to 0 as $w\rta\infty$ for any $m\in\{1,\dots,M \}$. Let $q_w\in\N$ denote the following measure for the number of lattice points which are near the boundary of $w\mcl N$
\eqbn
q_w = | \{ j\in \Z^N\,:\,  0<\lambda(A(j)\cap (w\mcl N))<1 \} |.
\eqen
Observe that for each $i$, $Y^w_m(i,0)-\check Y^w_m(i)$ is the weighted sum of $q_w$ i.i.d.\ centered random variables with values in $\{-\frac{1}{M}, \frac{M-1}{M} \}$, divided by $w^{N/2}$, where the weights are in $[0,1]$. Since the upper Minkowski dimension of $\partial\mcl N$ is smaller than $N$, it holds that $q_w=O(w^{N-\ep})$ for some $\ep>0$. Therefore, for any $i\in\mcl V$, we have $\P[ |Y^w_m(i,0)-\check Y^w_m(i)|>w^{-\ep/3} ]\preceq \exp(-w^{-\ep/10})$. A union bound now gives that $Y^w_m(\cdot,0)-\check Y^w_m$ converges in law to 0 as $w\rta\infty$.
\end{proof}

The evolution of $Y^w$ is random, while the evolution of $Y$ (given
its random initial conditions) is deterministic. The next lemma states
that if $Y^w$ and $Y$ are coupled in such a way that they are likely
to be close at time $t_0$, then it is likely that they remain
close at time $t_0 + \Delta t$. Informally, this means that the
evolution of $Y^w$ is {\em approximately} deterministic and {\em approximately} follows the same evolution rule as $Y$. One reason the lemma is challenging to prove is that the evolution of $Y^w$ (resp.\ $Y$) may be very sensitive to small perturbations when the bias is approximately as strong towards two different opinions. To bound the effect of this we will use Theorem \ref{prop6}, which will imply that the measure of the set of points at which this happens is not too large, uniformly for $t$ in a compact set.
\begin{lem}
Let $M$, $N$, $R$, $\mcl R$, $w$ and $\mcl N$ be as in Section \ref{sec:intro1}, and consider the Schelling model on $V=\mcl S$. Let $T>0$, $\Delta t>0$ and $t_0\in\{0,\Delta t,2\Delta t,\dots, \lceil T/\Delta t \rceil\Delta t \}$. Consider an arbitrary coupling of $Y^w$ (which is defined by \eqref{eq24}) and $Y$ (which solves \eqref{eq1}, \eqref{eq2}). Let $(\mcl F_t)_{t\geq 0}$ denote the filtration 
\eqbn
\mcl F_t = \sigma\Big(
Y^w|_{\mcl S^N\times[0,t]},
X|_{{\frk S}^N \times[0,tw^{-N/2}]},
\{(j,s)\in\mcl R\,:\,s \in[0,tw^{-N/2}] \}
\Big).
\eqen
For any $\delta>0$ define the event $E^w_\delta:=\{\|Y^w(\cdot,t_0)-Y(\cdot,t_0)\|_{L^\infty(\mcl S^N)}<\delta\}$.
There exists a random function $v:\BB N\rta[0,1]$ and a random constant $c>1$, such that $\lim_{w\rta\infty}v(w)=1$, and such that for all $w\in\BB N$,
\eqbn
\1_{E_\delta^w}\BB P\left[ E(t_0+\Delta t) \leq E(t_0) (1+c\Delta t)+c\Delta t^2 \,|\,\mcl F_{t_0}\right] \geq \1_{E_\delta^w} v(w),\quad
E(t):= \|Y^w(\cdot,t)-Y(\cdot,t)
\|_{L^\infty({\mcl S^N})}.
\eqen
The constant $c$ depends on $T$ and the $\sigma$-algebra generated by $(B(x))_{x\in\mcl S^N}$, and the function $v$ depends on $T,\delta,\Delta t$ and the $\sigma$-algebra generated by $(B(x))_{x\in\mcl S^N}$.
\label{prop21b}
\end{lem}

Lemma \ref{prop21b} will follow almost immediately from the following lemma. For any $f\in \mcl C_M(\mcl S^N)$ and $w\in\BB N$ define $\|f\|_{L^{\infty}(\frk S^N)}:=\sup_{m\in\{1,...,M\}}\sup_{i\in\frk S^N} |f_m(w^{-1}i)|$.
\begin{lem}
	The result of Lemma \ref{prop21b}  holds if we replace $L^\infty(\mcl S^N)$ by $L^\infty(\frk S^N)$ in the definition of the event $E_\delta^w$ and in the second indented equation.
	\label{prop21a}
\end{lem}

The next lemma, which is a discrete version of the estimate \eqref{eq23}, will help us to prove Lemma \ref{prop21a}.
\begin{lem}
	For any $T>0$ and $\Delta t>0$ there are random constants $c$ and $w_0$ such that for any $\delta\in(0,1)$, if
	\eqbn
	\wt A_{t,\delta,w}:=\{i\in\frk S^N\,:\,\exists m,m'\in\{1,...,M\},m\neq m',\text{ such\,\,that }|Y_m(iw^{-1},t)-Y_{m'}(iw^{-1},t)|<\delta\},\quad t\geq 0,
	\eqen
	then $|\wt A_{t,\delta,w}|< c\delta w^N$ for all $t\in \{ 0,\Delta t,2\Delta t,\dots,\lceil T/\Delta t \rceil \Delta t \}$ and all sufficiently large $w\geq w_0$. The constant $c$ satisfies the same properties as in Lemma \ref{prop21b}, and the constant $w_0$ depends on $T$, $\Delta t$, and the $\sigma$-algebra generated by $(B(x))_{x\in\mcl S^N}$.
	\label{prop22}
\end{lem}

\begin{proof}
	Note that for any $m,m'\in\{1,\dots,M\}$ satisfying $m\neq m'$ the field $Y_m(\cdot,t)-Y_{m'}(\cdot,t)$ has the law of a constant multiple of $B_m$ plus some element of $\mcl L^{2t}(\mcl S^N)$. The estimate \eqref{eq23} implies that 
	\eqb
	\lambda(A_{t,\delta})<\frac 12 c\delta,\qquad A_{t,\delta}:=\{x\in\mcl S^N \,:\, |Y_m(x,t)-Y_{m'}(x,t)|<\delta\}.
	\eqe
	By e.g.\ \cite[Lemma 4.2.6]{stroock-measure} and since $A_{t,\delta}$ is open,
	\eqbn
	\lambda(A_{t,\delta}) = \lim_{w\rta\infty} \sum_{i\in\frk S^N} w^{-N} \bd 1_{|Y_m(iw^{-1},t)-Y_{m'}(iw^{-1},t)|<\delta} = \lim_{w\rta\infty} w^{-N}|\wt A_{t,\delta,w}|.
	\eqen
	Therefore $|\wt A_{t,\delta,w}|<c\delta w^N$ for all sufficiently large values of $w$ for fixed $t\in[0,T]$, so $|\wt A_{t,\delta,w}|<c\delta w^N$ for all sufficiently large values of $w$ and all $t\in \{0,\Delta t,2\Delta t,\dots,\lceil T/\Delta t \rceil\Delta t \}$.
\end{proof}

\begin{proof}[Proof of Lemma \ref{prop21a}]
Throughout the proof the implicit constant of $\preceq$ will depend only on $T$ and $(B(x))_{x\in\mcl S^N}$. The function $v$ will change throughout the proof, but will always satisfy the properties of the function $v$ in the statement of the lemma. The expression "for all sufficiently large $w$" means that a statement is true for all $w\succeq 1$. Fix $m\in\{1,...,M\}$. For any $i\in\wt S$ let $U_{i,t}=\bd 1_{X(i,w^{-N/2}t^-)\neq m}$, and let $V_i$ be i.i.d.\ Bernoulli random variables with $\BB P[V_i=0]=M^{-1}$ and $\BB P[V_i=1]=1-M^{-1}$. Recall the definition \eqref{eq47} of $\mcl R$, and for $i\in\frk S^N$ define the following sets $\mcl R',\mcl R'_0,\mcl R'_i,\mcl R'_{0,i}\subset\mcl R$ 
\eqbn
\begin{split}
\mcl R'&:=\{(j,tw^{-N/2})\in\mcl R\,:\, t\in(t_0,t_0+\Delta t]\},\qquad
\mcl R'_0:=\{(j,t)\in \mcl R'\,:\,\not\exists s\in[t_0,t)\text{ such\,\,that }(j,s)\in \mcl R'\},\\
\mcl R'_i &:= \{(j,t)\in\mcl R'\,:\,j\in\frk N(i)\},\qquad\qquad\,\,\,\,\qquad\quad
\mcl R'_{0,i} := \{(j,t)\in\mcl R'_0\,:\,j\in\frk N(i)\}=\mcl R'_0\cap \mcl R'_i.
\end{split}
\eqen
Note that several of the random variables or sets we have defined above depend on $w$, but we have chosen not to indicate the $w$ dependence in order to simplify notation. We have
\eqbn
\begin{split}
\|(Y^{w}_m(\cdot&,t_0+\Delta t)-Y_m(\cdot,t_0+\Delta t))
- (Y^{w}_m(\cdot,t_0)-Y_m(\cdot,t_0))\|_{L^\infty(\frk S^N)} \\
=&\,\sup_{j\in \frk S^N} \bigg|
w^{-N/2} \sum_{(i,t)\in\mcl R'_{j}} \bd 1_{X(i,t^-)\neq X(i,t)=m} 
- w^{-N/2} \sum_{(i,t)\in\mcl R'_{j}} \bd 1_{m=X(i,t^-)\neq X(i,t)}\\
&-(1-M^{-1})\int_{t_0}^{t_0+\Delta t}
\int_{x\in \mcl N(jw^{-1})} \bd 1_{ p(Y(x,t))=m}\,dx\, dt
+M^{-1}\int_{t_0}^{t_0+\Delta t}
\int_{x\in \mcl N(jw^{-1})} \bd 1_{ p(Y(x,t))\neq m}\,dx
\,dt\bigg|.
\end{split}
\eqen
To conclude the proof of the lemma it is sufficient to show that with probability $>v(w)$, the right side is $\preceq \Delta t(\Delta t+\delta)$. By the triangle inequality,
\eqbn
\|(Y^{w}_m(\cdot,t_0+\Delta t)-Y_m(\cdot,t_0+\Delta t))
- (Y^{w}_m(\cdot,t_0)-Y_m(\cdot,t_0)) 
\|_{L^\infty(\frk S^N)}
\leq  \sum_{v=1}^7 (A_v+ A'_v),
\eqen
where
\begin{align*}
A_1&=\sup_{j\in \frk S^N} \left| w^{-N/2}\sum_{(i,t)\in\mcl R'_{j}} \bd 1_{X(i,t^-)\neq X(i,t)=m} - 
w^{-N/2}\sum_{(i,t)\in\mcl R'_{j}} \bd 1_{ p(Y^{w}(iw^{-1},t))=m} U_{i,t}
\right|,
\displaybreak[3]\\
A_2&=\sup_{j\in \frk S^N} \left|w^{-N/2} \sum_{(i,t)\in\mcl R'_{j}} \bd 1_{p(Y^{w}(iw^{-1},t))=m} U_{i,t} -
w^{-N/2} \sum_{(i,t)\in\mcl R'_{j}} \bd 1_{p(Y^{w}(iw^{-1},t_0))=m} U_{i,t}\right|,
\displaybreak[3]\\
A_3&= \sup_{j\in \frk S^N} \left| {w^{-N/2}} 
\sum_{(i,t)\in \mcl  R'_j} \bd 1_{ p(Y^{w}(iw^{-1},t_0))=m}
U_{i,t} 
- {w^{-N/2}} \sum_{(i,t)\in\mcl R'_{0,j}}
\bd 1_{p(Y^{w}(iw^{-1},t_0))=m} U_{i,t_0}\right|,
\displaybreak[3]\\
A_4&= \sup_{j\in \frk S^N} \left|w^{-N/2} \sum_{(i,t)\in \mcl R'_{0,j}} \bd 1_{p(Y^{w}(iw^{-1},t_0))=m} U_{i,t_0} -w^{-N/2} \sum_{(i,t)\in\mcl R'_{0,j}} \bd 1_{p(Y^{w}(iw^{-1},t_0))=m}V_{i}\right|,
\displaybreak[3]\\
A_5&=\sup_{j\in \frk S^N} A_{5,j},\,A_{5,j}=\left| w^{-N/2}\sum_{(i,t)\in\mcl R'_{0,j}} \bd 1_{ p(Y^{w}(iw^{-1},t_0))=m} V_{i} - 
\Delta t(1-M^{-1}) 
\int_{x\in\mcl N(jw^{-1})} \bd 1_{p(Y^w(x,t_0))=m}\,dx\right|,
\displaybreak[3]\\
A_6&= (1-M^{-1})\Delta t\sup_{j\in \frk S^N} \left|
\int_{x\in\mcl N(jw^{-1})} \bd 1_{p(Y^w(x,t_0))=m}\,dx-
\int_{x\in\mcl N(jw^{-1})} \bd 1_{p(Y(x,t_0))=m}\,dx
\right|,
\displaybreak[3]\\
A_7&= (1-M^{-1})\sup_{j\in \frk S^N}\left|\Delta t \int_{x\in\mcl N(jw^{-1})} \bd 1_{p(Y(x,t_0))=m}\,dx 
- \int_{t_0}^{t_0+\Delta t}
\int_{x\in\mcl N(jw^{-1})} \bd 1_{p(Y(x,t))=m}\,dx
\,dt\right|,
\end{align*}
where we view $iw^{-1}$ as an element of $\mcl S^N$ by identifying $\mcl S$ (resp.\ $\frk S$) with $[0,R)$ (resp.\ $\{0,\dots,Rw-1\}$). Define $A'_v$ exactly as $A_v$, except that $=m$ is replaced by $\neq m$, $1-M^{-1}$ is replaced by $M^{-1}$, $U_{i,t}$ is replaced by $1-U_{i,t}$, and $V_{i}$ is replaced by $1-V_{i}$. 

For each $j\in\{1,...,7\}$ we will show that with probability $>v(w)$ we have $A_j\preceq \Delta t(\Delta t+\delta)$. The term $A'_j$ can be bounded exactly as $A_j$ for all $j$, and the proof of its bound will therefore be omitted.

First we show that $A_1\preceq \Delta t(\delta+\Delta t)$ with probability $>v(w)$. Since $|\mcl R'|$ is a Poisson random variable with parameter $R^N\Delta t w^{N/2}$, we may assume that $|\mcl R'|<2R^N\Delta t w^{N/2}$, since $\BB P[|\mcl R'|<2R^N\Delta t w^{N/2}]\geq v(w)$. This implies $\|Y^w(\cdot,t)-Y(\cdot,t)\|_{L^\infty(\mcl S^N)}\leq \delta + c\Delta t$ for all $t\in[t_0,t_0+\Delta t]$ and some appropriate $c$ as in the statement of the lemma. Under these assumptions,
\eqbn
A_1 
\leq w^{-N/2} \sum_{(i,t)\in\mcl R'} \1_{p(Y^w(iw^{-1},t))=0}
\leq w^{-N/2} \sum_{(i,t)\in\mcl R'} \sum_{1\leq m,m'\leq M,m\neq m'} \1_{|Y_{m'}(iw^{-1},t)-Y_m(iw^{-1},t)|\leq 2\delta+2c\Delta t}.
\eqen
By independence of $\mcl R'$ and $\mcl F_{t_0}$, conditioned on $\mcl F_{t_0}$ the right side has the law of the sum of i.i.d.\ $\{0,1\}$-valued random variables. By Lemma \ref{prop22} the probability that this random variable equals 1 is $\preceq \delta+\Delta t$. The bound for $A_1$ now follows by the assumed upper bound for $|\mcl R'|$ and a Chernoff bound.

We will bound $A_2$ by using \eqref{eq23} and the assumption $\|Y^w(\cdot,t_0)-Y(\cdot,t_0)\|_{L^\infty(\frk S^N)}<\delta$. As above we assume  $|\mcl R'|<2R^N\Delta t w^{N/2}$. Define $\wt A\subset\frk S^N$ by 
\eqbn
\wt A=\{i\in\frk S^N\,:\,\exists m'\in\{1,...,M\}\backslash\{m\}\text{ such\,\,that } |Y^{w}_{m'}(w^{-1}i,t_0)-Y^{w}_m(w^{-1}i,t_0)|\leq 2w^{-N/2}|\mcl R'|\}.
\eqen
We claim that $|\wt A|\preceq (\Delta t+\delta)w^N$ for all sufficiently large $w$. If $i\in\wt A$ the assumption $\|Y(\cdot,t_0)-Y^w(\cdot,t_0)\|_{L^\infty(\frk S^N)}<\delta$ implies that there exists $m'\in\{1,...,M\}\backslash\{m\}$, such that
\eqb
|Y_{m}(iw^{-1},t_0)-Y_{m'}(iw^{-1},t_0)|\leq 2 w^{-N/2}|\mcl R'|+2\delta < 4R^N\Delta t+2\delta.
\label{eq26}
\eqe
It follows by Lemma \ref{prop22} that the set of nodes $i$ satisfying \eqref{eq26} is $\preceq w^N(\Delta t+\delta)$, and our claim follows. 

If $i\in\frk S^N$ is such that there is a $t\in[t_0,t_0+\Delta t]$ for which $p(Y^w(iw^{-1},t))\neq p(Y^w(iw^{-1},t_0))$, then we must have $i\in\wt A$ by the definition of $Y^w$. Therefore
\eqbn
A_2 \leq w^{-N/2} \sum_{(i,t)\in\mcl R'} \bd 1_{i\in\wt A}.
\eqen
We conclude the bound for $A_2$ by using independence of $\mcl R'$ and $\wt A$ and proceeding exactly as in the proof of $A_1$.

Next we claim that with probability $>v(w)$ we have $A_3<(\Delta t)^2$. For each $i\in\frk S^N$ let $P_i$ be the Poisson random variable with parameter $\Delta t w^{-N/2}$ which denotes the number of rings of Poisson clock $i$ during the interval $(t_0,t_0+\Delta t]$. Then 
\eqbn
A_3 \leq w^{-N/2} \sum_{i\in\frk S^N} \bd \max\{P_i-1,0\}.
\eqen
Since $\BB E[\max\{P_i-1,0\}]\preceq (\Delta t)^2 w^{-N}$, we have $\E[A_3]\preceq(\Delta t)^2 w^{-N/2}$. By Chebyshev's inequality  $\BB P[A_3\geq (\Delta t)^2] \preceq w^{-N/2}$, which implies our claim.

We will bound $A_4$ by approximating the region where $p(Y^w(iw^{-1},t_0))=m$ by small cubes of side length $L^{-1}$, and by proving that when a node $i$ is sampled uniformly from one of these cubes and $L\ll w$, then $U_{i,t_0}$ and $V_i$ have approximately the same distribution. As in our proof for the bound of $A_3$ we can assume $|\mcl R'|<2R^N\Delta t w^{N/2}$. Let $L=\lceil w^{1/2}\rceil$ and divide $\mcl S^N$ into $(RL)^N$ disjoint cubes of side length $L^{-1}$. For any $i\in\frk S^N$ let $I^{L}_i$ denote the cube containing $iw^{-1}$. Define $A\subset\mcl{S}$ and $\wt A\subset\frk S^N$ by
\eqbn
A = \{x\in\mcl S^N\,:\,p(Y(x,t_0))=m\},\qquad \wt A = \{i\in\frk S^N\,:\,I^{L}_i\subset A\}.
\eqen
We have
\eqb
A_4 \leq 
w^{-N/2} \sum_{(i,t)\in\mcl R'} |\bd 1_{p(Y^w(iw^{-1},t_0))=m}-\bd 1_{i\in \wt{A}}| + 
\sup_{j\in\frk S^N}
w^{-N/2} \sum_{(i,t)\in\mcl R'_{0,j},i\in\wt A}(U_{i,t_0}-V_i).
\label{eq32}
\eqe
We will prove that $A_4<\delta\Delta t$ with probability $>v(w)$. Any $i\in\frk S^N$ for which $\bd 1_{p(Y^w(iw^{-1},t_0))=m}\neq \bd 1_{i\in \wt{A}}$ must satisfy one of the following conditions: 
(i) $p(Y^w(iw^{-1},t_0))\neq p(Y(iw^{-1},t_0))$, or (ii) $p(Y(iw^{-1},t_0))=m$ and $i\not\in \wt A$. 

We will prove that the number of nodes satisfying one of the conditions (i)-(ii) is $\preceq w^{N}\delta$ with probability $>v(w)$. If $i\in\frk S^N$ satisfies (i), by definition of $E_\delta^w$ there is an $m'\in\{1,...,M\}$, $m'\neq m$, such that $|Y_m(iw^{-1},t_0)-Y_{m'}(iw^{-1},t_0)|<2\delta$, and the wanted result follows by Lemma \ref{prop22}. 
If $i$ satisfies (ii), the function $Y_m(\cdot,t_0)-Y_{m'}(\cdot,t_0)$ intersects zero in $I_i^L$. By our estimates for the event $G_k^1$ in Proposition \ref{prop16}, it holds with probability $>v(w)$ that the number of such cubes is $<L^{N-1/2+1/100}$ for all $t\in[0,T]$. Using $L=\lceil w^{1/2}\rceil$, it follows that for $w>\delta^{-4}$ the number of nodes $i\in\mcl{\wt S}$ satisfying (ii) is $\preceq L^{N-1/2+1/100} (w/L)^N \preceq w^N\delta$ with probability $>v(w)$. This completes the proof that the number of nodes satisfying one of the conditions (i)-(iii) is $\preceq w^{N}\delta$ with probability $>v(w)$.

Given any $i\in\frk S^N$ the events $\{\exists t\in(t_0,t_0+\Delta t]\text{\,\,such\,\,that\,\,} (i,t)\in\mcl R'\}$ and $\bd 1_{p(Y^w(i,t_0))=m}\neq \bd 1_{i\in \wt{A}}$ are independent. Proceeding as when bounding $A_2$ and $A_3$, we see that the first term on the right side of \eqref{eq32} is $\preceq \delta\Delta t$ with probability $>v(w)$.

Next we will prove that the second term on the right side of \eqref{eq32} converges to 0 in probability as $w\rta\infty$. Since $\sup_{x\in\mcl S}|Y^{w}(x,t_0)|<\delta+\sup_{x\in\mcl S}|Y(x,t_0)|<\infty$ on $E_\delta^w$, the difference in probability between the events $\{U_{i,t_0}=1\}$ and $\{V_{i}=1\}$ is $\preceq w^{-N/2}$ when we sample $i$ uniformly from one of the cubes $I^{L}_j$. Since $\mcl R'$ is independent of $U_{i,t_0}$ and $V_i$ for all $i\in\frk N(j)$, the second term on the right side of \eqref{eq32} is stochastically dominated by $w^{-N/2}$ times the sum of $<|\mcl R'|\preceq \Delta t w^{N/2}$ i.i.d.\ random variables taking values in $\{-1,0,1\}$ and with expectation $\preceq w^{-N/2}$. Our claim follows by a Chernoff bound and a union bound. 

Next we claim that $A_5<\delta\Delta t$ with probability $>v(w)$. If $|\mcl R'_{0,j}|\geq \Delta t 2^Nw^{N/2}$ let $\wt{\mcl R'}_{0,j}$ denote the first $\lceil \Delta t2^N w^{N/2}\rceil$ rings of the Poisson clocks during the interval $[t_0,t_0+\Delta t]$, and if $|\mcl R'_{0,j}|< \Delta t 2^N w^{N/2}$ let $\wt{\mcl R'}_{0,j}$ denote the union of $\mcl R'_{0,j}$ and $(\lceil\Delta t2^N w^{N/2}\rceil-|\mcl R'_{0,j}|)$ pairs $(i,t_0+\Delta t)$, where the $i$'s are pairwise different and sampled independently and uniformly from $\frk N(j)$. By the triangle inequality and letting $\Delta$ denote symmetric difference, 
\eqb
\begin{split}
A_{5,j} 
\leq&\,\left( w^{-N/2}\sum_{(i,t)\in\mcl R'_{0,j}\Delta \wt{\mcl R'}_{0,j}}V_i \right)\\
&+ \left|w^{-N/2}\sum_{(i,t)\in\wt{\mcl R'}_{0,j}} \bd 1_{p(Y^w(iw^{-1},t_0))=m}V_i - 
\Delta t(1-M^{-1})\int_{x\in\mcl N(jw^{-1})} \bd 1_{p(Y^w(iw^{-1},t_0))=m}\,dx\right|.
\end{split}
\label{eq33}
\eqe
We will prove that $\BB P[A_{5,j}> w^{-1/100}]$ decays faster than any power of $w$ when $w\rta\infty$, which is sufficient to complete the proof of our bound for $A_5$. We see immediately that the first term on the right side of \eqref{eq33} decays sufficiently fast. By independence of $\mcl R'_{0,j}$ and $\mcl F_{t_0}$, the second sum on the right side of \eqref{eq33} is, conditioned on $\mcl F_{t_0}$, equal in law to $w^{-N/2}$ times the sum of $\preceq \Delta t w^{N/2}$ independent bounded centered random variables. We obtain the desired bound by a Chernoff bound.

Now we will prove that $A_6 \preceq \delta\Delta t$ with probability $>v(w)$. By first using $\|Y^{w}_m(\cdot,t_0)-Y_m(\cdot,t_0)\|_{L^\infty(\mcl S^N)}< \delta$ and \eqref{eq22}, and then using $Y_m-B_m \in \mcl L_M^{t}(\mcl S^N)$ and \eqref{eq23} for all $t\in[t_0,t_0+\Delta t]$, we get
\eqbn
\begin{split}
A_6 &\leq \Delta t \sum_{1\leq m,m'\leq M,m\neq m'}
\sup_{x\in\mcl S^N} \int_{|x-x'|\leq1} \bd 1_
{\left|\left(Y_{m}(t_0,x')-Y_{m'}(t_0,x')\right)\right|\leq 2\delta} \,dx'\\
&\leq \Delta t \sum_{1\leq m,m'\leq M,m\neq m'} \sup_{f\in{\mcl L}^{2t_0}(\mcl S^N)} \int_{\mcl S^N} \bd 1_{|(B_m(x)-B_{m'}(x))-f(x)|<2\delta}\,dx\\
&\preceq \delta \Delta t.
\end{split}
\eqen

Finally we will bound $A_7$. By Lemma \ref{prop1} and $\|Y_m(\cdot,t)-Y_m(\cdot,t_0)\|_{L^\infty(\mcl S^N)}\leq 2\Delta t$ for all $t\in[t_0,t_0+\Delta t]$, we have $A_7\preceq(\Delta t)^2$.
Combining the above estimates for $A_j$, $j=1,...,7$, we obtain the lemma by a union bound.
\end{proof}

\begin{proof}[Proof of Lemma \ref{prop21b}]
	For any $x\in\mcl S^N$ there are $\alpha_i(x)\in[0,1]$ and $\frk x_i(x)\in\frk S^N$ for $i=1,\dots,2^N$ such that $\|\frk x_i(x)-x\|_{\infty}\leq w^{-1}$ and $\sum_{i=1}^{2^N}\alpha_i(x)=1$, and such that for any $t\geq 0$, $Y^w(x,t)=\sum_{i=1}^{2^N} \alpha_i(x) Y^w(\frk x_i(x),t)$. For any $x\in\mcl S^N$ define $\Delta Y^w(x):=Y^w(x,t_0+\Delta t)-Y^w(x,t_0)$ and $\Delta Y(x):=Y(x,t_0+\Delta t)-Y(x,t_0)$.
	Observe that
	\eqb
	\Delta Y^w(x) - \Delta Y(x)
	= \sum_{i=1}^{2^N} \alpha_i(x)(\Delta Y^w(\frk x_i(x)) - \Delta Y(\frk x_i(x))) 
	+ \sum_{i=1}^{2^N} \alpha_i(x)(\Delta Y(\frk x_i(x)) - \Delta Y(x)).
	\label{eq75}
	\eqe
	By uniform continuity of $Y$, which follows from uniform continuity of $B$,
	\eqb
	\sup_{x\in\mcl S^N} \sum_i \alpha_i(x)(\Delta Y(\frk x_j(x)) - \Delta Y(x)) \rta 0\qquad\text{\,\,as\,\,}w\rta\infty,
	\label{eq54}
	\eqe 
	and the rate of convergence depends only on $T$ and the $\sigma$-algebra generated by $(B(x))_{x\in\mcl S^N}$. By Lemma \ref{prop21a},
	\eqb
	\1_{E_\delta^w}\BB P\left[\sup_{x\in\frk S^N} \|
	\Delta Y^w(x)- \Delta Y(x) 
	\| \leq  c\Delta t(\Delta t+\delta)\,|\,\mcl F_{t_0}\right] \geq \1_{E_\delta^w} v(w).
	\label{eq55}
	\eqe
	We obtain the desired bound for $\Delta Y^w(x) - \Delta Y(x)$ by combining \eqref{eq75}, \eqref{eq54} and \eqref{eq55}. 
\end{proof}

The following lemma will be needed to transfer the result of Proposition \ref{prop18} from $\mcl S$ to $\R$. It says that a discrete version of Lemma \ref{prop21c} holds with high probability for large $w$.
\begin{lem}
	Consider the Schelling model on $\frk S$ (resp.\ $\Z$) with (in the notation of Section \ref{sec:intro1}) $N=1$, $M\in\{2,3,\dots\}$ and $\mcl N=\mcl N_\infty$. Let $Y^w\in \mcl C_M({\mcl S}\times\R_+)$ (resp.\ $Y^w\in \mcl C_M({\R}\times\R_+)$) be given by \eqref{eq24}. Fix an interval $I\subset\R$ of length $>1$ and some $m\in\{1,\dots,M\}$. For any $\ep>0$ and $t\geq 0$ define the event $E_t^{\ep}$ by
	\eqbn
	E_t^\ep := \left\{ Y^{w}_m(x,t)> \sup_{m'\in\{1,\dots,M \}\setminus \{m\} } Y^{w}_{m'}(x,t) + \ep,\,\,\forall x\in I \right\}.
	\eqen
	For any $\ep>0$, $\lim_{w\rta\infty}\P\big[E_0^\ep \setminus \bigcup_{t\in[0,\ep^{-1}]} \left(E_t^0\right)^c \big]=0$.
	\label{prop28}
\end{lem}
\begin{proof}
	Define $\wt Y^{w}_m(x,t):=Y^{w}_m(x,t)- \sup_{m'\in\{1,\dots,M \}\setminus \{m\} } Y^{w}_{m'}(x,t)$.  For any interval $I\subset\R$ and $w\in\N$, let $wI=\{wx\,:\,x\in I \}$. Define stopping times $T$ and $T_i$ for $i\in wI$ by
	\eqbn
	T_i= \inf\{t\geq 0\,:\,\wt Y^{w}_m(w^{-1}i,t)\leq 0 \},\qquad
	T = \inf_{i\in(wI)} T_i.
	\eqen
	Since $\bigcup_{t\in[0,\ep^{-1}]} \left(E_t^0\right)^c \subset \{T\leq\ep^{-1} \}$ it is sufficient by a union bound to prove that for each fixed $i\in wI$
	\eqbn
	\log\P[E_0^\ep;\,T_i=T\leq \ep^{-1}] \preceq -w,
	\eqen
	where the implicit constant can depend on all parameters except $w$ and $i$. Letting $t_n:=\ep^{-1}w^{-0.1}n$ for $n\in\{0,1,\dots,\lceil w^{0.1}\rceil \}$, we observe that
	\eqbn
	\begin{split}
	\{E_0^\ep;\, T_i=T\in[t_n,t_{n+1}] \} \subset &
		\left\{ E_0^\ep;\,t_n\leq T_i=T;\,\wt Y^{w}_m(w^{-1}i,t_n) < \frac 12 \wt Y^{w}_m(w^{-1}i,0) \right\} \\ 
		&\cup
		\left\{ E_0^\ep;\,
			|\{ (j,w^{-1/2}t)\in\mcl R\,:\, j\in\frk N(i) ,\,t\in[t_{n},t_{n+1}]\}| w^{-1/2} 
			>  \frac{1}{100} \wt Y^{w}_m(w^{-1}i,0) \right\}.
	\end{split}
	\eqen
	It follows by a union bound that
	\eqbn
	\begin{split}
	\P[E_0^\ep;\,T_i=\,T\leq& \ep^{-1}]
	\leq \sum_{n=0}^{\lceil w^{0.1}\rceil-1} 
	\P\left[E_0^\ep;\,t_n\leq T_i=T;\,\wt Y^{w}_m(w^{-1}i,t_n) < \frac 12 \wt Y^{w}_m(w^{-1}i,0)\right]\\
	&+ \sum_{n=0}^{\lceil w^{0.1}\rceil-1}  \P\left[E_0^\ep;\,
	|\{ (j,w^{-1/2}t)\in\mcl R\,:\, j\in\frk N(i) ,\,t\in[t_{n},t_{n+1}]\}| w^{-1/2} 
	>  \frac{1}{100} \wt Y^{w}_m(w^{-1}i,0)\right].
	\end{split}
	\eqen
	Since $\wt Y^{w}_m(w^{-1}i,0)>\ep$ on the event $E^\ep_0$, the logarithm of the last sum is $\preceq -w$, so to conclude the proof of the lemma it is sufficient to show that for each fixed $n\in\{0,1,\dots,\lceil w^{-0.1}\rceil-1 \}$,
	\eqb
	\log\P\left[E_0^\ep;\,t_n\leq T_i=T;\,\wt Y^{w}_m(w^{-1}i,t_n) < \frac 12 \wt Y^{w}_m(w^{-1}i,0)\right] \preceq -w,
	\label{eq73}
	\eqe
	where the implicit constant can depend on all parameters except $w$, $i$, and $n$. Fix  $n\in\{0,1,\dots,\lceil w^{-0.1}\rceil-1 \}$, and define
	\eqbn
	\begin{split}
		&\frk N^+ := \{j\in\frk N(i)\,:\,X(j,0)\neq m,\,jw\in I \},\qquad
		\mcl R^+:=\{j\in\frk N^+\,:\,\exists t\in[0,t_n]\text{\,\,such\,\,that\,\,}  (j,tw^{-1/2})\in\mcl R\}, \\
		&\frk N^- := \{j\in\frk N(i)\,:\,X(j,0)=m,\,jw\not\in I\},\qquad
		\mcl R^-:=\{j\in\frk N^-\,:\,\exists t\in[0,t_n]\text{\,\,such\,\,that\,\,}  (j,tw^{-1/2})\in\mcl R\}.
	\end{split}
	\eqen
	By large deviation estimates for Bernoulli random variables,
	\eqb
	\log \P\big[ \wh E^c \big] \preceq -w,\qquad
	\wh E:= \left\{\big| |\mcl R^\pm| - (1-e^{-w^{-1/2}t_n}) |\frk N^\pm| \big|< w^{0.1}\right\}.
	\label{eq72}
	\eqe
	Furthermore, observe that
	\eqb
	\begin{split}
		|\frk N^+|-|\frk N^-|
		&= |\{j\in\frk N(i)\,:\,(jw)\in I \}| -
		|\{j\in\frk N(i)\,:\,X(j,0)= m \}|\\
		&\geq (w+1) - \left( \frac{2w+1}{M} + w^{1/2} Y^w_m(w^{-1}i,0) \right)\\
		&>-w^{1/2} Y^w_m(w^{-1}i,0),
	\end{split}
	\label{eq74}
	\eqe
	where the inequality on the second line follows by the definition of $Y^w$. By \eqref{eq74}, the definition of $\wh E$, $|1-e^{-w^{-1/2}t_n}|\leq 2\ep^{-1} w^{-1/2}$, and
	\eqbn
	\begin{split}
		\1_{t_n\leq T}\wt Y^{w}_m(w^{-1}i,t_n)
		&\geq \1_{t_n\leq T}\left(\wt Y^{w}_m(w^{-1}i,0) + 2w^{-1/2}|\mcl R^+| - 2w^{-1/2}|\mcl R^-|\right),\\
		\end{split}
	\eqen
	it follows that on the event $\{E_0^\ep;t_n\leq T_i=T;\wh E\}$, 
	\eqbn
	\wt Y^{w}_m(w^{-1}i,t_n)\geq \wt Y^w_m(w^{-1}i,0)-2\ep^{-1}w^{-1/2}Y^w_m(w^{-1}i,0)-4w^{-0.4}. 
	\eqen
	Since $\log\P[ 2\ep^{-1}w^{-1/2}Y^w_m(w^{-1}i,0)>\wt Y^w_m(w^{-1}i,0) ]\preceq -w$, this result and \eqref{eq72} implies \eqref{eq73}.
\end{proof}

Proposition \ref{prop18} now follows by iterating the estimate of Lemma \ref{prop21b}.
\begin{proof}[Proof of Proposition \ref{prop18}]
First consider the case $V=\mcl S$ and $N\in\N$. Couple the discrete and continuum Schelling model as in Lemma \ref{prop19}. Let $T,\Delta t>0$. Conditioned on $B$, let $c$ and $v$ be the (random) constant and function, respectively, of Lemma \ref{prop21b}. Recall that $c$ depends on $B$ and $T$, while $v$ depends on $B$, $T$, $\Delta t$, and the error $E(t_0)$ with $E$ as in Lemma \ref{prop21b}. By Lemma \ref{prop21b}, and with $\mcl F_0$ and $E^w_{c\Delta t^2}$ for $t_0=0$ as in that lemma, 
\eqbn
\1_{E^w_{c\Delta t^2}}\P\left[\|Y^w(\cdot,\Delta t)-Y(\cdot,\Delta t)\|_{L^\infty(\mcl S^N)}<c^2(\Delta t)^3+2c(\Delta t)^2\,|\,\mcl F_0\right]>\1_{E^w_{c\Delta t^2}}v(w).
\eqen
Iterating the result of Lemma \ref{prop21b}, we get further that for any $n\in\BB N$, 
\eqbn
\1_{E^w_{c\Delta t^2}}\P\left[\|Y^w(\cdot,\wt n\Delta t)-Y(\cdot,\wt n\Delta t)\|_{L^\infty(\mcl S^N)}<\Delta t(1+c\Delta t)^{\wt n+1}-\Delta t,\quad \wt n\in\{0,...,n\}\,|\,\mcl F_0\right] >
\1_{E^w_{c\Delta t^2}} v(w)^n.
\eqen
We need $n_0:=\lceil T/\Delta t\rceil$ time steps to reach time $T$, so conditioned on $\mcl F_0$ and on the event $E^w_{c\Delta t^2}$, with probability at least $v(w)^{n_0}$ and $\Delta t<1/(100 c)$,
\eqb
\|Y^w(\cdot,\wt n\Delta t)-Y(\cdot,\wt n\Delta t)\|_{L^\infty(\mcl S^N)}
<
\Delta t(1+c\Delta t)^{n_0+1}-\Delta t<(2e^{cT}-1)\Delta t,\qquad \forall \wt n\in\{0,\dots,n_0\}.
\label{eq57}
\eqe
With probability converging to 1 as $w\rta\infty$, for any interval $I=[\Delta t\wt n,\Delta t(\wt n+1)]$ and node $i\in\frk S^N$, the total number of times during $I$ at which the Poisson clock of a node in $\frk N(i)$ rings, is $\leq 10^N w^{N/2}\Delta t$. Therefore, with probability converging to 1 as $w\rta\infty$,
\eqbn
\sup_{i\in\frk S^N} \sup_{0\leq n\leq n_0} \sup_{d\in[0,\Delta t]} 
\|Y^w(iw^{-1},n\Delta t+d) - Y^w(iw^{-1},n\Delta t)\|_{L^\infty(\mcl S^N)} \leq 10^N \Delta t.
\eqen
Combining this estimate with \eqref{eq57}, for any given $\ep_0>0$ and for all $w$ sufficiently large as compared to $\ep_0$,
\eqbn
\1_{E^w_{c\Delta t^2}}
\P\left[\sup_{x\in{\mcl S^N}} \sup_{t\in[0,T]} 
\|Y(x,t) - Y^w(x,t)\|_{L^\infty(\mcl S^N)} \leq (2e^{cT}+10^N) \Delta t\,\Big|\,\mcl F_0\right]
>
\1_{E^w_{c\Delta t^2}}(v(w)^{n_0}-\ep_0).
\eqen
Since $\P[ E^w_{c\Delta t^2} ]\rta 1$ as $w\rta\infty$, for all sufficiently large $w$,
\eqbn
\P\left[\sup_{x\in{\mcl S^N}} \sup_{t\in[0,T]} 
\|Y(x,t) - Y^w(x,t)\|_{L^\infty(\mcl S^N)} \leq (2e^{cT}+10^N) \Delta t\right]
>
v(w)^{n_0}-2\ep_0.
\eqen
We first make $(2e^{cT}+10^N) \Delta t$ arbitrarily small by decreasing $\Delta t$, and then we make $v(w)^{n_0}-2\ep_0$ arbitrarily close to 1 by sending $\ep_0\rta 0$ and $w\rta\infty$. It follows that $\sup_{t\in[0,T]}\|Y^w(\cdot,t)-Y(\cdot,t)\|_{L^\infty(\mcl S^N)}\rta 0$ in probability. By the Skorokhod representation theorem we can couple the model for different values of $w$, such that we obtain almost sure convergence. This concludes the proof in the case $V=\mcl S$.

Now consider the case $V=\R$ and $N=1$. Let $\ep>0$. For $R>2(\ep^{-1}+2)$ define the event $E_R$ by
\eqbn
\begin{split}
	E_R = &\,\{ \exists a^-\in[-R/2+2,-\ep^{-1}], a^+\in[\ep^{-1},R/2-2]\,:\, 
	Y_1(x,0)>\sup_{m'\in\{2,\dots,M \} } Y_{m'}(x,0)+\ep,\\
	&\qquad\forall x\in [a^--2,a^-]\cup [a^+,a^+ + 2] \}.
\end{split}
\eqen
Choose $R$ sufficiently large such that $\P[E_R]>1-\ep/2$. Let $Y$ (resp.\ $\wh Y$) denote the solution of \eqref{eq1}, \eqref{eq2} on $\R$ (resp.\ $\mcl S=[-R/2,R/2]$), and let $Y^w$ (resp.\ $\wh Y^w$) be given by \eqref{eq24} for the Schelling model on $\Z$ (resp.\ $\frk S$). We will argue that we can couple $Y,\wh Y,Y^w$, and $\wh Y^w$ such that with probability at least $1-\ep$, $Y^w\rta Y$ uniformly on $[-\ep^{-1},\ep^{-1}]\times[0,\ep^{-1}]$. This will be sufficient to complete the proof of the proposition since $\ep$ was arbitrary. 

By the convergence result for the torus proved above, we can couple $\wh Y^w$ and $\wh Y$ such that $\wh Y^w|_{[-R/2,R/2]\times[0,\ep^{-1}]}$ converges uniformly to $\wh Y|_{[-R/2,R/2]\times[0,\ep^{-1}]}$. Furthermore, on $E_R$ we can couple $Y$ and $\wh Y$ such that $Y|_{[a^-,a^+]}=\wh Y|_{[a^-,a^+]}$, since the law of the initial conditions are the same, and since Lemma \ref{prop21c} implies that $p(Y(x,t))=p(\wh Y(x,t))=1$ for all $x\in [a^--1,a^-]\cup [a^+,a^+ + 1]$ and $t\geq 0$.  
To complete the proof of the proposition it is sufficient to prove that on $E_R$ we can couple $Y^w$ and $\wh Y^w$ such that $Y^w|_{[a^-,a^+]\times[1,\ep^{-1}]}=\wh Y^w|_{[a^-,a^+]\times[1,\ep^{-1}]}$ with probability at least $1-\ep/2$. 

Consider a coupling of $Y^w$ and $\wh Y^w$ such that the initial opinion of the nodes corresponding to the interval $[a^- -1,a^+ +1]$ is identical for the models on $\Z$ and $\frk S$, and such that the set of rings of Poisson clocks corresponding to this interval, i.e.\ the set $\{ (i,t)\in\mcl R\,:\,(a^- -1)w\leq i\leq (a^++1)w \}$, is the same for the models on $\Z$ and $\frk S$. We also assume that draws as described in (iii) of Section \ref{sec:intro1} are resolved in the same way. By Lemma \ref{prop28}, $p(Y^w(x,t))=p(\wh Y^w(x,t))=1$ for all $x\in [a^--2,a^-]\cup [a^+,a^+ + 2]$ and $t\geq 0$ with probability at least $1-\ep/2$ for sufficiently large $w$. On the event that this happens $Y^w|_{[a^-,a^+]\times [0,\ep^{-1}]}=\wh Y^w|_{[a^-,a^+]\times [0,\ep^{-1}]}$ for all $t\in[0,\ep^{-1}]$, so we have obtained an appropriate coupling. 
\end{proof}

\subsection{Limiting states for the one-dimensional discrete Schelling model}
\label{sec:discrete2}
In this section we will first conclude the proof of Theorem \ref{thm1}. Then we will prove that the opinion of each node in the Schelling model in any dimension converges almost surely, and we will present a result on stable configurations in the higher-dimensional Schelling model. 

The main inputs to our proof of Theorem \ref{thm1} are Propositions \ref{prop20a} and \ref{prop18}. We consider a coupling of the discrete and continuum Schelling model as in Proposition \ref{prop18}, and choose a sufficiently large $t\geq 0$ such that the limiting configuration of the continuum Schelling model described in Proposition \ref{prop20a} is almost obtained; more precisely, we choose $t$ sufficiently large such that with high probability 0 is contained in an interval of length strictly larger than 1 on which $p\circ Y(\cdot,t)$ is constant. Let $m\in\{1,\dots,M \}$ denote the value of $p\circ Y(\cdot,t)$ in this interval. Recall that by the scaling we used when defining $Y^w$ in \eqref{eq24}, a time $t$ for $Y$ corresponds to time $tw^{-1/2}$ for the discrete Schelling model. 

To conclude the proof it will be sufficient to prove that $p\circ Y^w=m$ in the interval of length $>1$ identified above until all nodes in this interval have changed opinion to $m$. We will first prove a lemma (Lemma \ref{prop20b}) which says, roughly speaking, that $p\circ Y=m$ in the interval for a macroscopic time with high probability, and then we prove (Lemma \ref{prop21}) that conditioned on the event of Lemma \ref{prop20b}, $p\circ Y=m$ in the interval throughout $\lceil w^{0.02}\rceil$ time intervals of length $w^{-0.01}$ with very high probability. 

In each step of the proof we allow the interval on which $p\circ Y=m$ to shrink slightly. For nodes $i$ bounded away from the boundary of the interval, we can guarantee that $p\circ Y=m$ by using (among other properties) that the fraction of nodes in $\frk N(i)$ which have a bias towards $m$ is strictly larger than $1/2$ for (almost) the full time interval we consider; therefore the bias of $i$ towards $m$ will have an upwards drift and never become negative. For nodes $i$ near the boundary of our interval, however, up to half of the nodes in $\frk N(i)$ may have a bias towards another opinion than $m$, so we do not necessarily have an upward drift, and the node may eventually get a bias towards another opinion. For such nodes we can guarantee that the bias will not become negative too fast, by using that the node typically has a strong bias towards $m$ at the beginning of the time interval we consider. We show that the interval on which  $p\circ Y=m$ shrinks sufficiently slowly, such that all nodes on a subinterval of length $>1$ get opinion $m$ before the interval vanishes.

Define
\eqbn
\begin{split}
\ol {\mcl Y}^{w}_m(i,t) &:= \left(\sum_{j\in\frk N(i)} \1_{X(j,t)=m}\right) - \sup_{m'\in\{1,\dots,M\}\setminus m} \left( \sum_{j\in\frk N(i)} \1_{X(j,t)= m'}\right)\\
&=w^{1/2} \left({ Y}^{w}_m(i/w,tw^{1/2}) - 
\sup_{m'\in\{1,\dots,M\}\setminus m} Y^{w}_{m'}(i/w,tw^{1/2})\right).
\end{split}
\eqen
Observe that if $\ol {\mcl Y}^{w}_m(i,t)>0$ then $m$ is the most common opinion in the neighborhood of node $i$ at time $t$. In the statement and proof of the following lemma $wI=\{wx\,:\,x\in I \}$ for any interval $I\subset\R$ and $w\in\N$.
\begin{lem}
Couple the discrete and continuum Schelling model on $V$ as described in Proposition \ref{prop18}, where $V=\R$ or $V=\mcl S$, and $N=1$, $M\in\{2,3,\dots\}$, and $\mcl N=\mcl N_\infty$. Let $\{A_1,\dots,A_M\}$ be as defined in Proposition \ref{prop20a}, and define $\wh E$ to be the event that the set $\R\setminus \cup_{1\leq m\leq M}A_m$ has measure zero, and that each set $A_m$ can be written as the union of intervals of length $>1$. If $\wh E$ occurs, choose $a\in V$ in a $\sigma(B)$-measurable way such that $a\in \cup_{m\in\{1,\dots,M\}}A_m$ almost surely, and let $m\in\{1,\dots,M\}$ be such that $a\in A_m$. Let $c_1,c_2\in(0,1/10)$, let $I'$ be the connected component of $A_m$ containing $a$, and let $I\subset I'$ be the open interval with left (resp.\ right) end-point at distance $c_2$ from the left (resp.\ right) end-point of $I'$. Let $E=E^w_{c_1,c_2}$ be the event that $\wh E$ occurs, that $I$ has a length between $1+c_2$ and $c_2^{-1}$, and $\ol {\mcl Y}^{w}_m(i,t)>0$ for all $i\in (wI)\cap \Z$ and $t\in [c_1^{-1/2}w^{-1/2},c_1]$. Then $\lim_{c_2\rta 0} \lim_{c_1\rta 0} \lim_{w\rta \infty}\P[E\cup (\wh E)^c]=1$.
\label{prop20b}
\end{lem}
\begin{proof}
	First we give a brief outline of the proof. For small $c_1$ and large $w$ it holds with high probability (by Proposition \ref{prop18}) that all nodes in $wI$ have a large bias towards opinion $m$ at time $c_1^{-1/2}w^{-1/2}$. In particular, $w^{-1/2}\ol{\mcl Y}_m^w(i,c_1^{-1/2}w^{-1/2})\gg 1$ for nodes $i$ in $wI$. We consider the system until (roughly speaking) the first time $\wh T>c_1^{-1/2}w^{-1/2}$ at which $\ol{\mcl Y}_m^w(i,\wh T)\leq  0$ for some node $i$ in $wI$; note that until this time occurs all nodes in $wI$ will have a bias towards $m$. We show that $\wh T>c_1$ with high probability by arguing that each individual node $i$ in $wI$ is unlikely to be the {\emph{first}} node in $wI$ for which $\ol{\mcl Y}_m^w(i,t)\leq  0$ if $t\in [c_1^{-1/2}w^{-1/2},c_1]$. Let the nodes $i_1^*$ and $i_2^*$ represent the two end-points of $(wI)\cap\Z$. If $i=i_1^*$ and for $t\in[c_1^{-1/2}w^{-1/2},\wh T]$, the evolution of $w^{-1/2}\ol{\mcl Y}_m^w(i,t)$ is (approximately) bounded below by a Brownian motion with a weak downward drift starting from a large positive value, since we know that at least half of the neighbors of $i_1^*$ have a bias towards $m$. This implies that $\ol{\mcl Y}_m^w(i,t)$ will not reach zero before time $c_1\ll 1$ with high probability. We argue similarly for $i=i_2^*$. If $i$ is contained in $wI$ and has distance $\Omega(w)$ from the boundary of $wI$, then $\ol{\mcl Y}_m^w(i,t)$ has an upward drift for $t\in[c_1^{-1/2}w^{-1/2},\wh T]$, since the fraction of neighbors of $i$ which have a bias towards $m$ is uniformly above $1/2$; therefore $\ol{\mcl Y}_m^w(i,t)$ is unlikely to get negative before time $c_1$. If $i$ is close to the boundaries of $wI$, but not equal to $i_1^*$ or $i_2^*$, we conclude that $\ol{\mcl Y}_m^w(i,t)$ is unlikely to get negative by comparing with $i_1^*$ or $i_2^*$.
	
	Note that the lemma clearly holds if $\P[\wh E]=0$, so we may assume $\P[\wh E]>0$. We will condition on the event $\wh E$ throughout the proof of the lemma. In other words, all objects we define are defined conditional on $\wh E$.	Let $\wt I\subset I'$ be the open interval with left (resp.\ right) end-point at distance $c_2/2$ from the left (resp.\ right) end-point of $I'$, and observe that $I\subset \wt I$. Let $i^*_1$ (resp.\ $i^*_2$) be the smallest (resp.\ largest) element of $(wI)\cap\Z$, and let $\mcl R$ denote the set of rings as defined in \eqref{eq47}.

	Define the following random variables $A_1,A_2\in\{0,1,2,\dots \}$, where $|\cdot|$ denotes the number of elements in a set
	\eqb
	\begin{split}
	A_1&:=|\{(j,t)\in\mcl R\,:\, c_1^{-1/2}w^{-1/2}\leq t\leq c_1,j\in\{i_1^*-w,\dots,i_2^*+w  \}|,\\
	A_2&:=|\{ (j,t)\in\mcl R\,:\, j\in \{i^*_1-w,\dots,i^*_2+w\} ,\,
		0\leq t\leq c_1^{-1/2}w^{-1/2},\,\\
		&\qquad\exists (j,t')\in\mcl R\text{\,\,such\,\,that\,\,}c_1^{-1/2}w^{-1/2}\leq t'\leq c_1
		\}|,
	\end{split}
	\label{eq40}
	\eqe
	Let $A_3\in\R$ be a random variable which is equal to the infimum in $[0,1]$ such that the following inequalities are satisfied 
	\eqb
	\begin{split}
		&|\{i\in \mcl N(i^*_k)\cap (wI)\,:\, X(i,0)\neq m \}| w^{-1} \geq \frac 12 - A_3w^{-1/2},\quad k=1,2,\\
		&|\{i\in \mcl N(i^*_k)\setminus (wI)\,:\, X(i,0)= m \}| w^{-1}\leq \frac 12 +A_3w^{-1/2},\quad k=1,2.
	\end{split}
	\label{eq41}
	\eqe
	Define $A_4\in\R$ by
	\eqbn
	A_4 := \inf_{x\in I} \left[ Y^{w}_m(x,c_1^{-1/2}) - \sum_{m'\neq m} Y^{w}_{m'}(x,c_1^{-1/2})\right].
	\eqen
	Let $\ol E$ be the event 
	\eqbn
	\begin{split}
		\ol E:=&\left\{ \forall i\in (wI),\, 
		|\{j\in\mcl N(i)\cap (wI)\,:\,X(i,0)\neq m \}|w^{-1}>\frac 12 - w^{-1/4}\right\}\\
		&\cap\left\{|\{j\in\mcl N(i)\setminus (wI)\,:\,X(i,0) = m \}|w^{-1}<\frac 12 + w^{-1/4}
		 \right\}.
	\end{split}
	\eqen
	Define the event $\wt E=\wt E^w_{c_1,c_2}$ by
	\eqb
	\wt E =
	\{A_1<2 c_1 c_2^{-1}w \} \cap 
	\{A_2< 2 c_1^{1/2} c_2^{-1} w^{1/2} \} \cap 
	\{A_3<c_1^{-1/10} \} \cap 
	\{A_4>2 \} \cap	
	\{1+c_2<\lambda(I)<c_2^{-1} \} \cap 
	\ol E. 
	\label{eq44}
	\eqe
	The probability of the first, second and sixth event on the right side of \eqref{eq44} converge to 1 as $w\rta\infty$ for any fixed $c_1,c_2\in(0,1/10)$ if we condition on the fifth event $\{ 1+c_2<\lambda(I)<c_2^{-1} \}$. For $M=2$ the probability of the third event on the right side of \eqref{eq44} converges to a constant as $w\rta\infty$, and it converges to 1 when first $w\rta\infty$ and then $c_1\rta 0$. For $M>2$ the probability of the third event on the right side of \eqref{eq44} converges to 1 as $w\rta\infty$.  It is immediate from \eqref{eq1} that $Y_m(\cdot,t)-\sum_{m'\neq m} Y_{m'}(\cdot,t)\rta \infty$ uniformly on $I$ as $t\rta\infty$. Therefore it follows from Proposition \ref{prop18} that the probability of the fourth event  on the right side of \eqref{eq44} converges to a constant as $w\rta\infty$, and it converges to 1 as first $w\rta\infty$ and then $c_1\rta 0$. The probability of the fifth event  on the right side of \eqref{eq44} is independent of $w$ and $c_1$, and converges to 1 as $c_2\rta 0$. Therefore 
	\eqb
	\lim_{c_2\rta 0} \lim_{c_1\rta 0} \lim_{w\rta \infty} \P[\wt E]=1.
	\label{eq39}
	\eqe
	
	For any $j\in\N$ let $t^1_j$ be the $j$th smallest element of $\{t\geq c_1^{-1/2}w^{-1/2}\,:\, \exists x\in\mcl N(i^*_1)\text{\,\,such\,\,that\,\,} (x,t)\in\mcl R \}$. Then define $i_j^1\in\mcl N(i^*_1)$ such that $(i^1_j,t^1_j)\in\mcl R$, and define $R^1_j:=(\ol {\mcl Y}^{w}_m(i^*_1,t_j)-\ol {\mcl Y}^{w}_m(i^*_1,t_j^-))\1_{(i^1_j>i^*_1)\vee (\ol {\mcl Y}^{w}_m(i^1_j,t^1_j)<0)}$. Define $(i^2_j,t^2_j)$ (resp.\ $R^2_j$) exactly as $(i^1_j,t^1_j)$ (resp.\ $R^1_j$), but with $i^*_2$ in place of $i^*_1$, and where we require $i_j<i^*_2$ instead of $i_j>i^*_1$ in the indicator in the definition of $R^2_j$. Now define the following stopping times $T_j$ for $j=1,2,3$
\eqb
\begin{split}
	T_1 &:= \inf\left\{t\geq c_1^{-1/2} w^{-1/2}\,:\, 
	\sum_{j\,:\, t^1_j\leq t} R_j^1
	\leq -w^{1/2} \right\},\\
	T_2 &:= \inf\left\{t\geq c_1^{-1/2} w^{-1/2}\,:\, 
	\sum_{j\,:\, t^2_j\leq t} R_j^2
	\leq -w^{1/2} \right\},\\
	T_3 &:= \inf\Big\{t\geq c_1^{-1/2} w^{-1/2}\,:\, \exists i\in wI\text{\,\,such\,\,that\,\,} i\in \{i^*_1+c_2w+1,i^*_1+c_2w+2,\dots,i^*_2-c_2w-1 \}\\
	&\qquad\qquad\text{\,\,and\,\,} \ol {\mcl Y}^{w}_m(i,t)\leq 0 \Big\}.
\end{split}
\label{eq36}
\eqe
We will now argue that 
\eqb
\ol {\mcl Y}^w(i,t)>0\quad\forall i\in wI \qquad\text{if}\qquad t<T_1\wedge T_2\wedge T_3\text{\,\,and\,\,}\wt E\text{\,occurs}.
\label{eq37} 
\eqe
To prove this it is sufficient to show that if $\wt E$ occurs and $t<T_1\wedge T_2\wedge T_3$, then $\ol {\mcl Y}(i,t)>0$ for any  $i\in\{i^*_1,\dots,i^*_1+c_2w \}\cup\{i^*_2-c_2w,\dots,i^*_2 \}$, since the inequality $\ol {\mcl Y}(i,t)>0$ clearly holds for other $i$ by the definition of $T_3$. For $i\in\{i^*_1,\dots,i^*_1+c_2w \}$ this follows by first observing that the bias is positive for all nodes in $\frk N(i)\setminus\frk N(i^*_1)$ throughout $[c_1^{-1/2}w^{-1/2},t]$, so 
\eqbn
\ol {\mcl Y}^{w}_m(i,t)-\ol {\mcl Y}^{w}_m(i,c_1^{-1/2}w^{-1/2})\geq \ol {\mcl Y}^{w}_m(i^*_1,t)-\ol {\mcl Y}^{w}_m(i^*_1,c_1^{-1/2}w^{-1/2})
\geq \sum_{j\,:\, t^1_j\leq t} R_j^1,
\eqen
and then using the definition of $T_1$ and $\ol {\mcl Y}^{w}_m(i^*_1,c_1^{-1/2}w^{-1/2})\geq A_4w^{1/2}>2w^{1/2}$ to conclude that $\ol {\mcl Y}^{w}_m(i,t)>0$.
We argue similarly for $i\in\{i^*_2-c_2w,\dots,i^*_2 \}$, and can conclude that \eqref{eq37} holds.

Since \eqref{eq37} implies that $\wt E\cap\{c_1<T_1\wedge T_2\wedge T_3\}\subset E$,
\eqb
\P[E^c] \leq \P[\wt E^c] + 
\P[\wt E; T_1\leq T_2\wedge T_3; T_1\leq c_1] +
\P[\wt E; T_2\leq T_1\wedge T_3; T_2\leq c_1] +
\P[\wt E; T_3\leq T_1\wedge T_2; T_3\leq c_1].
\label{eq38}
\eqe
We will bound the probability of each term on the right side separately. By \eqref{eq39}, and since the second and third term on the right side have the same probability, it is sufficient to bound the second term and the fourth term on the right side of \eqref{eq38}.

First we bound the second term on the right side of \eqref{eq38}. Recall the definition of $(i^1_j,t^1_j)_{j\in\N}$ above.
Define $\wt R_j$ to be equal to $2$ (resp.\ $-2$) iff both the conditions $X(i^1_j,0)\neq m$ (resp.\ $X(i^1_j,0)= m$) and $i^1_j\in (wI)$ (resp.\ $i^1_j\not\in (wI)$) are satisfied, and define $\wt R_j=0$ otherwise. Observe that $R^1_j,\wt R_j\in\{-2,0,2\}$ for all $j\in\N$, and that $\wt R_j\leq R^1_j$ on the event $\wt E$ if $t^1_j<T_1\wedge T_2\wedge T_3$ and if there is no $t'$ such that $(i^1_j,t')$ is contained in the set considered when defining $A_2$ in \eqref{eq40}. Therefore, using that $A_1<2 c_1 c_2^{-1}w$ and $A_2<2c_1^{1/2}c_2^{-1}w^{1/2}$ on $\wt E$,
\eqb
\begin{split}
\1_{\wt E;T_1\leq T_2\wedge T_3} \inf_{c_1^{-1/2}w^{-1/2}\leq t\leq c_1\wedge T_1} \left( \sum_{j\,:\, t^1_j\leq t} R_j^1 \right)
\geq \1_{\wt E;T_1\leq T_2\wedge T_3} \min_{1\leq k\leq 2c_1 c_2^{-1}w} \sum^{k}_{j=1} \wt R_j -2c_1^{1/2}c_2^{-1}w^{1/2}.
\end{split}
\label{eq42}
\eqe
On the event $\wt E$ (in particular, the requirement on $A_3$) and for $w$ sufficiently large, conditioned on $\sigma((X|_{t=0})$ the sequence $(\wt R_j)_{j\in\N}$ stochastically dominates a sequence $(\wh R_j)_{j\in\N}$ of i.i.d.\ random variables which are equal to 2 (resp.\ $-2$) with probability $\frac 12(\frac 12-c_1^{-1/10}w^{-1/2})$ (resp.\ $\frac 12(\frac 12+c_1^{-1/10}w^{-1/2})$), and equal to 0 otherwise. This follows since the sequence $(i^1_j,t^1_j)$ is independent of $\wt R_j$ conditioned on $i^*_1$. Letting $W=(W_t)_{t\geq 0}$ be a Brownian motion with drift $-2 c_1^{-1/10}$ and $\op{Var}[W_t]=2t$ the process $t\mapsto w^{-1/2}\sum_{j=1}^{\lceil tw\rceil}\wh R_j$ converges in law to $(W_t)_{t\geq 0}$. Therefore
\eqbn
\begin{split}
	\lim_{c_2\rta 0} &\lim_{c_1\rta 0} \lim_{w\rta \infty} 
	\P[\wt E; T_1\leq T_2\wedge T_3; T_1\leq c_1] \\
	&\leq \lim_{c_2\rta 0} \lim_{c_1\rta 0} \lim_{w\rta \infty} \P\left[\wt E;T_1\leq T_2\wedge T_3; \inf_{c_1^{-1/2}w^{-1/2}\leq t\leq c_1\wedge T_1} \left( \sum_{j\,:\, t^1_j\leq t} R_j^1 \right)<-\frac 12 w^{1/2} \right]\\
	&\leq \lim_{c_2\rta 0} \lim_{c_1\rta 0} \lim_{w\rta \infty} \P\left[\wt E;T_1\leq T_2\wedge T_3; \min_{1\leq k\leq 2c_1 c_2^{-1}w} \sum^{k}_{j=1} \wt R_j -2c_1^{1/2}c_2^{-1}w^{1/2}<-\frac 12 w^{1/2} \right]\\
	&\leq \lim_{c_2\rta 0} \lim_{c_1\rta 0} \P\left[\inf_{0\leq t\leq 2c_1 c_2^{-1}} W_t-2c_1^{1/2} c_2^{-1} <-\frac 12\right]\\
	&=0.
\end{split}
\eqen

In order to bound the fourth term on the right side of \eqref{eq38}, a union bound and the upper bound on the length of $I$ on the event $\wt E$, implies that it is sufficient to prove that for fixed $c_1,c_2\in(0,1/10)$ and fixed $i\in [-c_2w,c_2w]\cap\Z$,
\eqb
\begin{split}
	&\log \P[E'; T\leq c_1] \preceq -w,\\
	&E' := \wt E\cap \{T=T_3\leq T_1\wedge T_2 \} \cap \{i\in \{i^*_1+c_2w+1,i^*_1+c_2w+2,\dots,i^*_2-c_2w-1 \}\},\\
	&T := \inf\{t\geq c_1^{-1/2} w^{-1/2}\,:\, \ol {\mcl Y}^{w}_m(i,t)\leq 0\},
\end{split}
\label{eq43}
\eqe
where the implicit constant is independent of $w$, but may depend on $c_1,c_2$. We prove this by a similar approach as our bound for the second term on the right side of 
\eqref{eq38}, and will therefore only give a brief justification. 	For any $j\in\N$ let $t'_j$ be the $j$th smallest element of 
$\{t\geq c_1^{-1/2}w^{-1/2}\,:\, \exists x\in\frk N(i)\text{\,\,such\,\,that\,\,} (x,t)\in\mcl R \}$. 
Then define $i'_j\in\frk N(i)$ such that $(i'_j,t'_j)\in\mcl R$, and define 
$R'_j:=(\ol{\mcl Y}^{w}_m(i,t'_j)-\ol{\mcl Y}^{w}_m(i,(t'_j)^-))$.
Define $\wt R'_j$ to be equal to $2$ (resp. $-2$) iff both the conditions 
$X(i'_j,0)\neq m$ (resp.\ $X(i'_j,0)= m$) and $i'_j\in (wI)$ (resp.\ $i'_j\not\in (wI)$) are satisfied, and define $\wt R'_j=0$ otherwise. 
 As above, using $A_1<2 c_1 c_2^{-1}w$ and $A_2<2 c_1^{1/2}c_2^{-1}w^{1/2}$ on $E'$,
\eqbn
\1_{E'} \inf_{c_1^{-1/2}w^{-1/2}\leq t\leq c_1\wedge T }
\left(\ol{\mcl Y}^{w}_m(i,t)-\ol{\mcl Y}^{w}_m(i,c_1^{-1/2}w^{-1/2}) \right)
\geq \1_{E'} \min_{1\leq k\leq 2c_1c_2^{-1}w} \sum_{j=1}^{k} \wt R'_j - 2 c_1^{1/2}c_2^{-1}w^{1/2}.
\eqen
In this case the sequence $(\wt R'_j)_{j\in\N}$ on the event $E'$, stochastically dominates a sequence $(\wh R'_j)_{j\in\N}$ of i.i.d.\ random variables which are equal to 2 (resp.\ $-2$) with probability $\frac 12(1+c_2)(\frac 12-w^{-1/4})$ (resp.\ $\frac 12(1-c_2)(\frac 12+w^{-1/4})$), and equal to 0 otherwise, where the terms $(1\pm c_2)$ follow from the condition $\lambda(I)>1+c_2$ and the terms $(\frac 12\pm w^{-1/4})$ follow from the definition of the event $\ol E$. For sufficiently large $w$ the process $t\mapsto w^{-1/2}\sum_{j=1}^{\lceil tw\rceil}\wh R_j-\frac 12 c_2\lceil tw\rceil$ stochastically dominates a Brownian motion $W'=(W'_t)_{t\geq 0}$ satisfying $W'_0=0$ and $\op{Var}(W'_t)=2t$. Therefore
\eqbn
\begin{split}
	\lim_{w\rta\infty} \P[E'; T\leq c_1] 
&\leq \lim_{w\rta\infty} \P\left[ E';\, \inf_{c_1^{-1/2}w^{-1/2}\leq t\leq c_1\wedge T }
\left(\ol{\mcl Y}^{w}_m(i,t)-\ol{\mcl Y}^{w}_m(i,c_1^{-1/2}w^{-1/2}) \right)<-w^{-1/2} \right]\\
&\leq \lim_{w\rta\infty}\P\left[\inf_{0\leq t\leq 2c_1c_2^{-1}} \left(W_t-2c_1^{1/2}c_2^{-1}+\frac 12 c_2 tw^{1/2}\right)<-1\right]=0,
\end{split}
\eqen
which implies \eqref{eq43}. 
\end{proof}

\begin{lem}
	Consider the setting described in Lemma \ref{prop20b}, and let $I$, $c_1$, $c_2$, and $E$ be as in that lemma. Define open intervals $I_n$ for $n\in\{0,\dots,\lceil w^{0.02}\rceil \}$ inductively by $I_0:=I$, and by letting $I_n\subset I_{n-1}$ be the open interval such that the left (resp.\ right) end-point of $I_n$ has distance $w^{-0.1}$ from the left (resp.\ right) end-point of $I_{n-1}$. For $n\in \{1,\dots,\lceil w^{0.02}\rceil \}$ define $t_n:=n w^{-0.01}$, and let $E_n$ be the following event
	\eqbn
	E_n = \{ \ol{\mcl Y}^{w}_m(i,t)>0,\,
	\forall i\in (wI_n),\, \forall t \text{\,\,such\,\,that\,\,} c_1^{-1/2} w^{-1/2}\leq t\leq t_n \}.
	\eqen 
	Then $\log\P\big[E\cap E_{n-1} \cap E^c_{n}\big]\preceq -w$ for all $n\in\{1,\dots,\lceil w^{0.02}\rceil \}$, where the implicit constant is independent of $w$ and $n$, but may depend on all other constants, including $c_1$ and $c_2$.
	\label{prop21}
\end{lem}
\begin{proof} 
	We first give a brief outline of the proof. Note that in this lemma we consider only times of order 1, rather than times of order $w^{-1/2}$ as in the remainder of the paper; this means that a uniformly positive fraction of the Poisson clocks have been ringing. As in the proof of Lemma \ref{prop20b}, we consider some time $T$ which represents the first time at which $\ol{\mcl Y}^{w}_m(i,t)\leq 0$ for some $i$ in $I_{n}$, and we show that for each $i$ in $I_{n}$ it is very unlikely that $i$ is the \emph{first} node for which this happens for $T\leq t_n$.  We assume the event $E\cap E_{n-1}$ occurs. We divide $\frk N(i)$ into three parts (see $\frk N_1(i)$, $\frk N_2(i)$, $\frk N_3(i)$ below). For each $k\in\{1,2,3 \}$ we can obtain a good (lower) bound on the contribution to $\ol{\mcl Y}^w_m(i,t_{n-1})$ coming from nodes in $\frk N_k(i)$, since we know approximately how many Poisson clocks which have been ringing before time $t_{n-1}$, and since nodes in $I_{n-1}$ (corresponding to $\frk N_1(i)\cup \frk N_3(i)$, plus maybe part of $\frk N_2(i)$) have a bias towards $m$ during $[ c_1^{-1/2} w^{-1/2}, t_{n-1}]$ on the event $E\cap E_{n-1}$. Conditioning on the fraction of nodes in  $\frk N_k(i)$, $k\in\{1,2,3 \}$, of opinion $m$ at time $t_{n-1}$, we compare the evolution of $\ol{\mcl Y}^w_m(i,t)$ on $[t_{n-1},T]$ to a random walk of a certain step size distribution, and we argue that this random walk is unlikely to hit 0 before time $t_n$, which completes the proof. 
	
	As in the proof of Lemma \ref{prop20b} we may assume $\P[\wh E]>0$. We assume throughout the proof that $\wh E$ occurs; in particular, some variables we define may exist only conditioned on $\wh E$. Define the following stopping times for $i\in [a-c_2^{-1}w,a+c_2^{-1}w]\cap\Z$
	\eqbn
	T_i:=\inf\{t\geq t_{n-1}\,:\, \ol{\mcl Y}^{w}_m(i,t)\leq 0 \},\qquad
	T:= \inf\{T_i\,:\, i\in(wI_n) \}, 
	\eqen
	Fix $i\in [a-c_2^{-1}w,a+c_2^{-1}w]\cap\Z$. By a union bound it is sufficient to prove the following estimate, where the implicit constant is independent of $w$, $n$, and $i$	\eqb
	\log\P\left[E_n^i; T_i<t_n \right] \preceq -w,\qquad
	E_n^i:=E \cup E_{n-1}\cup \{i\in (wI_{n})\} \cup \{ T_i=T<t_n\}.
	\label{eq45} 
	\eqe
	Divide the neighborhood $\frk N(i)$ into three disjoint parts $\frk N_k(i)$ for $k=1,2,3$ satisfying the following requirements; existence of appropriate neighborhoods is immediate by using the definition of $I_n$ and $\lambda(I_n)>1$
	\eqbn
	\begin{split}
		&|\frk N_1(i)|=w,\qquad |\frk N_2(i)|=w-w^{0.9}+1,\qquad |\frk N_3(i)|=w^{0.9},\\
		&\frk N_1(i)\subset\{j\in\frk N(i)\,:\, j\in(w I_n) \},\qquad\frk N_3(i)\subset\{j\in\frk N(i)\,:\, j\in (w I_{n-1}) \}.
	\end{split}
	\eqen
	For any $t\geq 0$ define $\ol{\mcl Y}^{w,k}_m(i,t)$ for $k=1,2,3$ by
	\eqbn
	\ol{\mcl Y}^{w,k}_m(i,t) := \left(\sum_{j\in\frk N_k(i)} \1_{X(j,t)=m} - \frac 1M |\frk N_k(i)|\right) - \left(\sum_{j\in\frk N_k(i)} \1_{X(j,t)\neq m}-\frac{M-1}{M} |\frk N_k(i)|\right).
	\eqen
	On $E\cap E_{n-1}$ the nodes in $\frk N_1(i)$ have bias $m$ throughout $[c_1^{-1/2}w^{-1/2},t_{n-1}]$. For each node the probability that its clock rings during $[0,t_{n-1}]$ is $(1-e^{-t_{n-1}})$, and when this happens for the first time for some $t\in[c_1^{-1/2}w^{-1/2},t_{n-1}]$ the node updates its opinion to $m$ iff its current opinion is different. When this happens $\ol{\mcl Y}^{w,1}_m$ increases by 2. Except on an event of exponentially small probability, for large $w$ the number of nodes with an opinion different from $m$ at time $c_1^{-1/2}w^{-1/2}$ differ from $\frac{M-1}{M}w$ by at most $\frac{1}{100} w^{1/2+1/100}$, and the number of nodes in $\frk N_1(i)$ whose Poisson clock rings during $t\in[c_1^{-1/2}w^{-1/2},t_{n-1}]$ differ from $(1-e^{-t_{n-1}})w$ by at most $\frac{1}{100}w^{1/2+1/100}$. Therefore 
	\eqb
	\begin{split}
		&\log\P\left[E;\,E_{n-1};\,\left|\ol{\mcl Y}^{w,1}_m(i,t_{n-1}) - 2\frac{M-1}{M}(1-e^{-t_{n-1}})w\right| > w^{1/2+1/100} \right] \preceq -w,
	\end{split}
	\label{eq48}
	\eqe
	By a similar argument the following inequalities hold. Note that we only get a lower bound for $\ol {\mcl Y}^{w,2}_m(i,t_{n-1})$ since we do not know the bias of the nodes in $\frk N_2(i)$ throughout $[0,t_{n-1}]$
	\eqb
	\begin{split}	
		&\log\P\left[E;\,E_{n-1};\,\ol{\mcl Y}^{w,2}_m(i,t_{n-1}) < -2\frac{1}{M}(1-e^{-t_{n-1}})(w-w^{0.9})-w^{1/2+1/100} \right] \preceq -w,\\
		&\log\P\left[E;\,E_{n-1};\,\left|\ol{\mcl Y}^{w,3}_m(i,t_{n-1}) - 2\frac{M-1}{M}(1-e^{-t_{n-1}})w^{0.9}\right| > w^{1/2+1/100} \right] \preceq -w.
	\end{split}
	\label{eq49}
	\eqe
	
	First assume $M>2$. By \eqref{eq48} and \eqref{eq49}, and by using
	$	\ol{\mcl Y}^{w}_m(i,t_{n-1})
	= \sum_{k=1}^{3} \ol{\mcl Y}^{w,k}_m(i,t_{n-1})$,
	\eqb
	\P\left[E;\,E_{n-1}; i\in(wI_{n}); \ol{\mcl Y}^{w}_m(i,t_{n-1})<\frac{M-2}{M}(1-e^{-t_{n-1}})w\right] \preceq -w.
	\label{eq50}
	\eqe
	Defining
	\eqbn
	A_1 := |\{(j,t)\in\mcl R\,:\, j\in\frk N(i),\,t\in[t_{n-1},t_n] \}|
	\eqen
	we have $\log\P[A_1>\frac{M-2}{4M}(1-e^{-t_{n-1}})w]\preceq -w$. Using this estimate, \eqref{eq50} and $\ol{\mcl Y}^{w}_m(i,t_{n})\geq \ol{\mcl Y}^{w}_m(i,t_{n-1})-2A_1$ we obtain \eqref{eq45} for the case $M>2$.
	
	In the remainder of the proof assume $M=2$.
	Defining $\mcl T^w:= \{t_{n-1}+w^{-0.3},t_{n-1}+2w^{-0.3},\dots, t_{n-1}+\lceil w^{0.3-0.01} \rceil w^{-0.3} \}$, observe by a union bound that for all sufficiently large $w$
	\eqbn
	\begin{split}
		\P\left(E_n^i; \inf_{t_{n-1}\leq t\leq T_i\wedge t_n} \ol {\mcl Y}^{w}_m(i,t)\leq 0 \right)
		\leq&\, 
		\sum_{s\in \mcl T^w} 
		\P\left(
		E_n^i; s<T_i; \ol {\mcl Y}^{w}_m(i,s) < \frac{1}{2} \ol {\mcl Y}^{w}_m(i,t_{n-1});\,
		\ol {\mcl Y}^{w}_m(i,t_{n-1})>w^{0.85}
		\right) \\
		&+\sum_{s\in \mcl T^w} 
		\P(E_n^i;\,|\{ (j,t)\in\mcl R\,:\,j\in\frk N(i), t\in[s-w^{-0.3},s] \}| > w^{0.8})\\
		&+\P\left(E_n^i;\, \ol {\mcl Y}^{w}_m(i,t_{n-1})\leq w^{0.85} \right).
	\end{split}
	\eqen
	The logarithm of the second sum on the right side is $\preceq -w$. The logarithm of the third term on the right side is also $\preceq -w$  by \eqref{eq48}-\eqref{eq49}, so to prove \eqref{eq45}, and thereby complete the proof of the lemma, it is sufficient to prove the following estimate for any $s\in \mcl T^w$
	\eqb
	\log \P\left(
	E_n^i; s<T_i; \ol {\mcl Y}^{w}_m(i,s) < \frac{1}{2} \ol {\mcl Y}^{w}_m(i,t_{n-1});\,
	\ol {\mcl Y}^{w}_m(i,t_{n-1})>w^{0.85}
	\right) \preceq -w.
	\label{eq52}
	\eqe
	Fix $s\in\mcl T^w$, define the following random variables
	\eqbn
	\begin{split}
	&R^+ := |\{ j\in\frk N^+(i)\,:\, \exists t\in(t_{n-1},s] \text{\,\,such\,\,that\,\,}(j,t)\in\mcl R   \}|,\quad\frk N^+(i):=\{j \in\frk N_1(i) \,:\, X(j,t_{m-1})\neq m\},\\
	&R^- := |\{ j\in\frk N^-(i)\,:\, \exists t\in(t_{n-1},s] \text{\,\,such\,\,that\,\,}(j,t)\in\mcl R   \}|,\quad \frk N^-(i):=\{j \in\frk N_2(i)\cup\frk N_3(i) \,:\, X(j,t_{m-1}) = m\},
	\end{split}
	\eqen
	and observe that 
	\eqb
	E_n^i\cap \{s<T_i \} \cap \left\{\ol {\mcl Y}^{w}_m(i,s) < \frac{1}{2} \ol {\mcl Y}^{w}_m(i,t_{n-1}) \right\}
	\subset
	E_n^i\cap \{s<T_i \} \cap \left\{R^+-R^-<-\frac{1}{4} \ol {\mcl Y}^{w}_m(i,s) \right\}.
	\label{eq53}
	\eqe
	By the definition of $\ol {\mcl Y}^{w,k}_m(i,t_{n-1})$,
	\eqb
	|\frk N^+(i)|=\frac{1}{2}(w - \ol {\mcl Y}^{w,1}_m(i,t_{n-1})),\qquad
	|\frk N^-(i)|=\frac{1}{2}(w + \ol {\mcl Y}^{w,2}_m(i,t_{n-1})+\ol {\mcl Y}^{w,3}_m(i,t_{n-1})).
	\label{eq51}
	\eqe
	By large deviation estimates for Bernoulli random variables
	\eqbn
	\begin{split}
		&\log \P[R^+<(1-e^{-(s-t_{n-1})})|\frk N^+(i)|-w^{1/2+1/100}] \preceq -w,\\
		&\log \P[R^->(1-e^{-(s-t_{n-1})})|\frk N^-(i)|+ w^{1/2+1/100}] \preceq -w,
	\end{split}
	\eqen
	so using \eqref{eq51} and $1-e^{-(s-t_{n-1})}<2 w^{-0.01}$, for all sufficiently large $w$,
	\eqbn
	\log \P\left[E_n^i; s<T_i; R^+-R^-< -\frac 14 \ol {\mcl Y}^w_m(i,s) \right] \preceq -w.
	\eqen
	We obtain \eqref{eq52} by using this estimate and \eqref{eq53}.
\end{proof}

\begin{proof}[Proof of Theorem \ref{thm1}]
Couple the discrete and continuum Schelling model as in Proposition \ref{prop18}. Almost surely there is an $m^*\in\{1,2\}$ such that $0\in A_{m^*}$. Let $I'\subset A_{m^*}$ (resp.\ $I^w\subset A^w_{m^*}$) be the connected component of $A_{m^*}$ (resp.\ $A^w_{m^*}$) containing to origin, where $I^w$ is empty if $0\not\in A_{m^*}^w$. Let $\ep>0$. Define $I\subset I'$ by $I:=\{x\in I'\,:\,\op{dist}(x,(I')^c)>\ep \}$. By translation invariance in law of both the discrete and continuum Schelling model it is sufficient to prove that $\P[I\subset I^w]>1-\ep$ for all sufficiently large $w\in\N$.

Consider the objects defined in the statement of Lemma \ref{prop20b} for $a=0$. Since $M=2$ we know by Proposition \ref{prop20a} that $\P[\wh E ]=1$, so $\lim_{c_2\rta 0}\lim_{c_1\rta 0} \lim_{w\rta\infty}\P[E]=1$. Let $c_1,c_2\in(0,1/2)$ be such that $\lim_{w\rta\infty}\P[E]>\ep/100$ and such that $c_2\ll\ep$. Observe that if 
\eqbn
E' := \left\{\forall i\in (wI)\cap\Z,\,\,\exists t\in (c_1^{-1/2}w^{-1/2},w^{0.01}) \text{\,\,such\,\,that\,\,} (i,t)\in\mcl R \right\}
\eqen
and the events $E_n$ for $n\in\{0,\dots, \lceil w^{0.02} \rceil \}$ are defined as in Lemma \ref{prop21}, then $E'\cap E_{\lceil w^{0.02} \rceil} \subset \{I\subset I^w\}$. It follows by a union bound that
\eqb
\P[I\not\subset  I^w ] \leq
\P[(E')^c] + \P[E^c] + \P[E;E_0^c] + \sum_{n=1}^{\lceil w^{0.02} \rceil} \P[E; E_{n-1}; E_n^c].
\label{eq77}
\eqe
The first term on the right side of \eqref{eq77} converges to 0 as $w\rta\infty$, and the second term on the right side of \eqref{eq77} is smaller than $\ep/3$ for all sufficiently large $w$ by our choice of $c_1$ and $c_2$. The third term on the right side of \eqref{eq77} is identically equal to zero. The last term on the right side of \eqref{eq77} is smaller than $\ep/3$ for all sufficiently large $w$ by Lemma \ref{prop21}. Therefore $\P[I\subset  I^w ]>1-\ep$ for all sufficiently large $w$, which concludes the proof.
\end{proof}

The following proposition says that the opinion of the nodes in the discrete Schelling model converge almost surely. For the case $M=2$ it was proved in \cite[Theorem 1.1]{tt14}, and our proof proceeds similarly as the proof for this case. See e.g.\ \cite{moran95,gh00} for other related results. Observe that this result is not used in our proof of our main results, and is included only as a statement of independent interest.
\begin{prop}
Consider the Schelling model on either $\BB Z^N$ or on the torus $\frk S^N$ where (in the notation on Section \ref{sec:intro1}) $\mcl N$ is invariant upon reflection through the origin, $N\in\N$, and $M\in\{2,3,\dots\}$. For each node $i$ the opinion $X(i,t)$ converges almost surely as $t\rightarrow\infty$, i.e., for each node $i$ there is a random time $T>0$ such that $X(i,t)=X(i,T)$ for all $t\geq T$.
\label{prop30}
\end{prop} 
\begin{proof}
The proof is identical for the two cases $\BB Z^N$ and $\frk S^N$. We will only present it for the case of $\BB Z^N$, but the result for $\frk S^N$ follows by replacing $\BB Z^N$ by $\frk S^N$ throughout the proof. Let $E$ be the set of undirected edges of the graph on which the Schelling model takes place, i.e.\ $(i,j)=(j,i)\in E$ for $i,j\in \BB Z^N$ iff $j\in \frk N(i)$ (equivalently, since $\mcl N$ is invariant upon reflection through the origin, $i\in \frk N(j)$). For each $i,j\in \BB Z^N$ such that $(i,j)\in E$,  associate a positive real number $w_{ij}$, such that for any $i,j,k$ for which $(i,j),(i,k)\in E$
\eqb
\frac{w_{ij}}{w_{ik}} < \frac{(2w+1)^N+1}{(2w+1)^N-1}.
\label{eq18}
\eqe
Choose the $w_{ij}$'s such that $\sum_{(i,j)\in E} w_{ij}<\infty.$ The existence of appropriate $w_{ij}$ satisfying these properties follows by \cite[Proposition 3.4]{tt14}, since the degree of each node is bounded by $(2w+1)^N$, and since our graph satisfies the growth criterion considered in \cite{tt14}. For each $i\in\BB Z^N$ and $t\geq 0$ define $J_t^i$ by 
\eqbn
J_t^i := \sum_{j\in \frk N(i)} w_{ij} \bd 1_{X(j,t)\neq X(i,t)} - \sum_{j\in \frk N(i)} w_{ij}\bd 1_{X(j,t^-)\neq X(i,t^-)}.
\eqen
If $t$ is a time at which the Poisson clock of node $i$ rings, the first (resp.\ second) term on the right side expresses how many neighbors of $i$ that disagree with $i$ after (resp.\ before) $i$ updates its opinion at time $t$. Assume the clock of node $i$ rings at time $t$, and consider the three rules (i)-(iii) from Section \ref{sec:intro1} that $i$ follows when updating its opinion. By our constraint \eqref{eq18}, $J_t^i<0$ if one of the rules (i) or (iii) apply, while $J_t^i=0$ if $i$ does not change its opinion, which is the case when (ii) applies. Given a node $i$ and the value of $w_{ij}$ for all $j\in \frk N(i)$, the number of possible values for $J_t^i$ is finite. It follows that we can find a real number $\ep_i>0$, such that we have either $J_t^i<-\ep_i$ or $J_t^i=0$ for all times $t\geq 0$ for which the clock of node $i$ rings.

Next define the Lyapunov function $L:[0,\infty)\rightarrow\R_+$ by
\eqbn
L_t = \sum_{(i,j)\in E} w_{ij} \bd 1_{X(i,t)\neq X(j,t)}.
\eqen
Note that $L_t<\infty$ for all $t\geq 0$ by our assumption of summability of $w_{ij}$. Since $L_t-L_{t^-} = J_t^i$ if the clock of node $i$ rings at time $t$, $L_t$ is decreasing in $t$. It follows that there exists some $L\geq 0$ such that $\lim_{t\rightarrow\infty}L_t = L$. Fix $i\in\BB Z^N$, and let $T>0$ be such that $L_t-L<\ep_i$ for all $t\geq T$. Since $L_t-L_{t^-} = J_t^i$ for any time $t$ at which the clock of node $i$ rings, we see that $J_t^i=0$ for all times $t>T$. It follows that $i$ never updates its opinion after time $T$, which completes the proof of the lemma.
\end{proof}

We end the section with a result which may be related to the limiting opinions of the Schelling model on $\Z^N$ for $N\geq 2$. We say that $A\subset\BB Z^N$ is \emph{connected} if, for any two $i,j\in A$, there is an $n\in\BB N$ and a sequence $\{i^k\}_{0\leq k\leq n}$ such that $i^0=i$, $i^n=j$, $i^k\in A$ and $\|i^{k}-i^{k-1}\|_{1}\leq 1$ for all $k\in\{1,...,n\}$. We say that $A\subset\BB Z^N$ is \emph{stable} if all nodes of $\Z^N$ agree with the most common opinion in their neighborhood when all nodes in $A$ have opinion $1$ and all nodes in $\BB Z^N\backslash A$ have opinion 2. Note that the definition of stability depends on $\mcl N$ and $w$. The \emph{diameter} of a set $A\subset \BB Z^N$ is defined by $\sup_{i,j\in A}\|i-j\|_{\infty}$. We say that $A$ is a \emph{smallest stable shape} for the Schelling model if it is a connected stable subset of $\BB Z^N$ of minimal diameter. 

We observed before the statement of Lemma \ref{prop23} that in the one-dimensional Schelling model on $\Z$ the final configuration of opinions consists of monochromatic intervals of length at least $w+1$. We also observe that the smallest stable shapes for the Schelling model on $\Z$ are sets $A\subset\Z$ consisting of $w+1$ consecutive integers. In particular, this means that all nodes are part of a monochromatic stable shape in the final configuration. By Theorem \ref{thm1} the blocks with constant opinion in the final configuration of the Schelling model on $\Z$ or $\frk S$ have length of order $w$. 

One might guess that the smallest stable shape is related to the diameter of a typical cluster in the limiting configuration of opinions also in higher dimensions. There exist stable configurations where the cluster sizes are smaller than the diameter of a stable shape (e.g.\ a checkerboard configuration when $\mcl N$ has the shape of a cube), but these seem unlikely to occur since they are typically unstable, in the sense that changing the opinion of a small number of nodes may cause a large cluster of nodes to obtain the same opinion. 

We thank Omer Tamuz for suggesting the approach used in the upper bound of the following proposition.

\begin{prop} 
	Let $N>1$, $p\in[1,\infty]$, and assume $\mcl N = \mcl N_p:=\{x\in\R^N\,:\,\|x\|_{p}<1 \}$. The diameter $d$ of a smallest stable shape for the Schelling model satisfies $w^{2}\preceq d \preceq w^{N+1}$, where the implicit constants can be chosen independently of $p$ and $w$, but the constant in the upper bound may depend on $N$.
	\label{prop33}
\end{prop}
\begin{proof}
	We will prove the lower bound $w^2 \preceq d$ by induction on the dimension $N$. We will only do the case $p<\infty$, but the case $p=\infty$ can be done in exactly the same way. We start with the case $N=2$. See Figure \ref{fig-stableshape} for an illustration. 
	
	\begin{figure}[ht!]
		\begin{center}
			\includegraphics[scale=1]{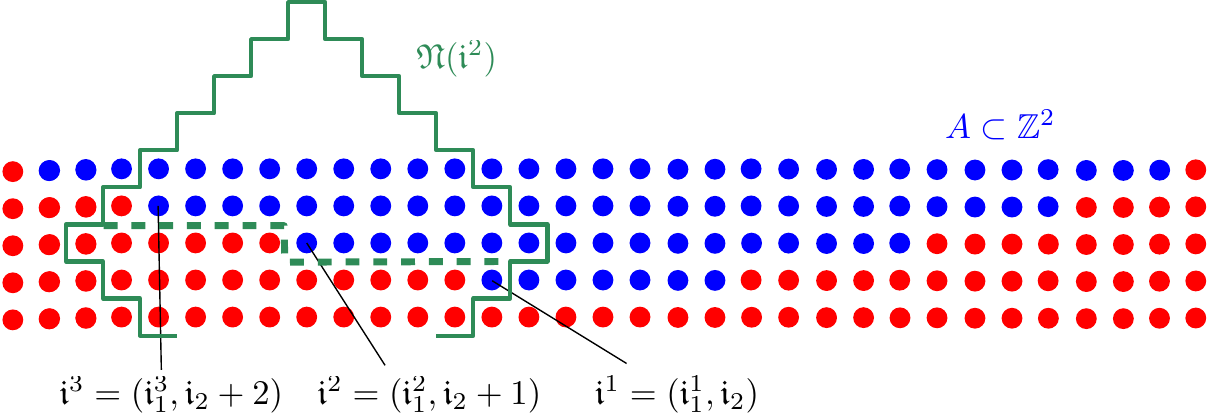} 
			\caption{Illustration of the lower bound in Proposition \ref{prop33} for $N=2$, $w=6$, and $p=1$. The points of $\Z^2$ marked in blue is a subset of a stable shape $A$, while elements of $\BB Z\setminus A$ are shown in red. The green dotted line separates $\frk N^-(\frk i^2)$ (lower part) and $\frk N^+(\frk i^2)$ (upper part). 
			} 
			\label{fig-stableshape}
		\end{center}
	\end{figure}
	Assume $A$ is a stable shape, and define $\frk i_1^k$ and $\frk i_2$ as follows for $k=\{1,\dots,\lfloor w/2 \rfloor \}$ 
	\eqbn
	\begin{split}
		\frk i_2 := \min\{i_2\in\Z\,:\,\exists i_1\in\Z \text{\,\,such\,\,that\,\,}(i_1,i_2)\in A \},\qquad
		\frk i_1^k := \min\{i_1\in\Z\,:\, (i_1,\frk i_2+k-1)\in A \}.
	\end{split}
	\eqen
	Then define $\frk i^k:=(\frk i_1^k,\frk i_2+k-1)\in A$.
	We will prove by induction on $k$ that $\frk i_1^{k}-\frk i_1^{k+1}\geq w-k$ for all $k\in \{1,\dots,\lfloor w/2 \rfloor \}$, which is sufficient to obtain the lower bound $w^2 \preceq d$ since it implies that $\frk i_1^{1}-\frk i_1^{\lfloor w/2\rfloor}\succeq w^2$. 
	For any $i=(i_1,i_2)\in\Z^2$ define
	\eqbn
	\begin{split}
		\frk N^-(i) = \{ (i'_1,i'_2)\in\frk N(i)\,:\,i'_2<i_2 \vee ((i'_2=i_2)\wedge (i'_1<i_1)) \},\qquad
		\frk N^+(i) = \frk N(i) \setminus \frk N^-(i).
	\end{split}
	\eqen
	First let $k=1$. By definition of $\frk i_2$ and $\frk i_1^1$, $A\cap\frk N^-(\frk i^1)=\emptyset$. Since $\frk i^1$ agrees with the most common opinion in its neighborhood this implies $A\cap\frk N(\frk i^1)=\frk N^+(\frk i^1)$. In particular, this implies by the definition of $\frk i_1^2$ and since $(\frk i_1^1-w+1,\frk i^2+1)\in \frk N^+(\frk i^1)$ that $\frk i_1^{1}-\frk i_1^2\geq w-1$. 
	
	Now assume $k>1$ and that $\frk i_1^{\ell}-\frk i_1^{\ell+1}\geq w-\ell$ for $\ell\in\{1,\dots,k-1\}$. This assumption implies by the definition of $\frk i_1^\ell$ that 
	\eqbn
	A\cap\frk N^-(\frk i^k) \subset \{(i_1^k+w-k+1, i_2+k-2), 
	(i_1^k+w-k+2, i_2+k-2),
	\dots,
	(i_1^k+w-1, i_2+k-2) \},
	\eqen
	in particular, $|A\cap\frk N^-(\frk i^k)|\leq k$. Since $\frk i^{k}\in A$ agrees with the most common opinion in its neighborhood by stability of $A$, we must have $|\frk N^+(\frk i^k) \setminus A|\leq k-1$. By $\mcl N=\mcl N_p$ for $p\in[1,\infty)$,
	\eqbn
	\{
	(\frk i^k_1-w+1, \frk i_2+k),
	(\frk i^k_1-w+2, \frk i_2+k),
	\dots,
	(\frk i^k_1-w+k, \frk i_2+k)
	\}\subset \frk N^+(\frk i^k),
	\eqen
	where we note that the set on the left side has $k$ elements. Since $|\frk N^+(\frk i^k) \setminus A|\leq k-1$ this implies $\frk i_1^{k}-\frk i_1^{k+1}\geq w-k$. This completes our proof by induction, and hence completes the proof of the lower bound for $d$ in the case when $N=2$.
	
	Now assume the lower bound $w^2\preceq d$ has been proved for dimension $2,\dots,N-1$ for some $N>2$. We want to show that it also holds in dimension $N$. Define 
	\eqbn
	\frk i_N:= \min\{i_N\in\Z\,:\,\exists i_1,\dots i_{N-1}\in\Z \text{\,\,such\,\,that\,\,}(i_1,i_2,\dots,i_N)\in A \}.
	\eqen
	Then 
	\eqbn
	\{i=(i_1,\dots,i_N)\in A\,:\, i_N<\frk i_N \}=\emptyset, 
	\eqen
	so since all elements of $A\cap\{ i=(i_1,\dots,i_N)\in\Z^N\,:\,i_N=\frk i_N \}$ agree with the most common opinion in its neighborhood and by, for any $i'=(i'_1,\dots,i'_N)\in\Z^N$, symmetry of $\frk N(i')$ upon reflection through the plane $i_N=i'_N$, the following $(N-1)$-dimensional set must be stable if the neighborhood of any $j\in\Z^{N-1}$ is given by $\{j'\in\Z^{N-1}\,:\,\|j-j'\|_{p}<w\}$
	\eqbn
	\{(i_1,\dots,i_{N-1})\in \Z^{N-1} \,:\,  (i_1,\dots,i_{N-1},\frk i_N)\in A \}.
	\eqen
	By the induction hypothesis this set has diameter $\succeq w^2$, so $A$ also has diameter $\succeq w^2$. This concludes the proof of the lower bound by induction.
	
	Now we will prove the upper bound $d\preceq w^{N+1}$. Let $r=2^{2N}w^{N+1}$, and define $A_0=\{i\in\BB Z^N\,:\,\|i\|_\infty\leq r$. Let all nodes in $A_0$ have opinion 1, and let all nodes in $\BB Z^N\backslash A_0$ have opinion 2. We define decreasing sets $A_n\subset\Z^N$, $n\in\N$, by induction as follows. For each $n\in\BB N$ we choose one element of $i\in A_{n-1}$ which does not agree with the most common opinion in its neighborhood, we change the opinion of this node to 2, and we define $A_n=A_{n-1}\backslash\{i\}$. We continue this procedure until $A_n=\emptyset$ or until all nodes in $A_n$ agree with the most common opinion in its neighborhood. Let $\wt n$ denote the time at which the process terminates. We will prove by contradiction that $A_{\wt n}\neq\emptyset$. This will imply the existence of a stable shape of diameter $\preceq r\asymp w^{N+1}$.
	
	Define 
	\eqbn
	E_n = \{(i,j)\,:\,i\in A_n,\,j\in \BB Z^N\backslash A_n\}.
	\eqen
	By our choice of the node $i$ in each step, the sequence $(|E_n|)_{1\leq n\leq \wt n}$ is strictly decreasing. Assuming $A_{\wt n}=\emptyset$ this implies $|A_0|\leq|E_0|$. We have $|A_0|=(2r+1)^{N}$. We can find a constant $C_N>0$ depending on $N$, such that there are $<C_N(2r+1)^{N-1}w$ nodes in $A_0$ which have a neighbor in $\BB Z^N\backslash A_0$. By using this and that $|\frk N(i)|\leq (2w+1)^N$ for any $i\in\frk S^N$, we get $|E_0|\leq C_N(2r+1)^{N-1}w(2w+1)^N$. Using $|A_0|\leq|E_0|$ these estimates imply $(2r+1)\leq C_N w(2w+1)^N$. This is a contradiction to our definition of $r$, and we conclude that $A_{\wt n}\neq\emptyset$. 
\end{proof}

\section{Open problems}
\label{sec:open}
This paper explains the limiting configuration of opinions in the Schelling model when $N=1$ and $M=2$ (see Theorem \ref{thm1}). One open problem is to understand the limiting configuration of opinions in cases where $N\geq 2$ and/or $M>2$.  In particular, it remains an open question to understand the following situations: 
$(i)$ $\frk S^N$ for $N\geq 2$, 
$(ii)$ $\Z^N$ for $N\geq 2$, and 
$(iii)$ $\frk S$ or $\Z$ for $M>2$.
	
Case $(i)$ could possibly be understood by studying the long-time behavior of solutions $Y$ of the initial value problem \eqref{eq1}, \eqref{eq2} for $N\geq 2$. If we knew that the limit $Y(x):=\lim_{t\rta\infty}p(Y(x,t))$ exists for almost every $x\in\mcl S^N$ a.s., the field $( Y(x) )_{x\in\mcl S^N}$ would likely describe law of the limiting opinions in the discrete model. Observe that non-trivial limiting configurations (i.e., limiting configurations with more than one limiting opinion) happen with positive probability, e.g.\ if the torus width $R$ is at least 3, and the initial data are such that $p(Y((x_1,\dots,x_N),0))$ equals 1 (resp.\ 2) for $x_1\in[0,1]$ (resp.\ $x_1\in[R-5/4,R-1/4]$). For such initial data we will have $p(Y((x_1,\dots,x_N),t))$ equal to 1 (resp.\ 2) for $x_1\in[0,1]$ (resp.\ $x_1\in[R-5/4,R-1/4]$) and all $t\geq 0$.

The continuum Schelling model \eqref{eq1}, \eqref{eq2} may be less helpful for understanding case $(ii)$. Even if we had established existence and uniqueness of solutions of \eqref{eq1}, \eqref{eq2} on $\R^N$ for $N\geq 2$ (see Theorem \ref{prop3}), the solution $Y$ may be of only limited help for understanding the final configurations of opinions in the discrete model. Observe that there are no bounded continuum stable shapes, where a continuum stable shape is a set $D\subset\R^N$ which is such that if $m\in\{1,\dots,M\}$, $t_0\geq 0$ and $p(Y(x,t_0))=m$ for all $x\in D$ then $p(Y(x,t))=m$ for all $t\geq t_0$ (see above Proposition \ref{prop33} for the discrete definition). Since there are no bounded continuum stable shapes, we expect that the limit $\lim_{t\rta\infty}p(Y(x,t))$ a.s.\ does not exist for any fixed $x\in\R^N$, at least when $M=2$. The existence of this limit is necessary in order for the continuum Schelling model to describe the limiting configurations of opinions in the discrete model.

The continuum Schelling model \eqref{eq1}, \eqref{eq2} is related to the discrete Schelling model upon rescaling the lattice by $w^{-1}$. If we proved a scaling limit result for the discrete model by proving convergence of the opinions in the continuum model, the diameter of a typical cluster in the discrete model would therefore be of order $w$. Figure \ref{fig-limit} suggests that the diameter of the limiting clusters on $\Z^2$ grow superlinearly in $w$. The typical cluster size may be related to the size of the smallest (discrete) stable shape for the model; see Proposition \ref{prop33} for upper and lower bounds on the diameter of the smallest stable shape. An independently interesting problem (which involves no probability) is to resolve the sizable discrepancy between the upper and lower bounds in Proposition \ref{prop33}. One could try to explicitly construct the minimal stable shape for each given $w$, and compute its size.

Case $(iii)$ could be understood by studying the long-time behavior of the solutions of \eqref{eq1}, \eqref{eq2} for $N=1$. We believe Theorem \ref{thm1} also holds for $M>2$, i.e., the limiting opinions in the Schelling model have a scaling limit upon rescaling the lattice by $w^{-1}$, and the limiting law can be described by an $M$-tuple $(A_1,\dots,A_M)$, where the sets $A_m$ have a.s.\ disjoint interior and can be written as the union of intervals each of length larger than 1 a.s. This version of Theorem \ref{thm1} with $M>2$ would be immediate from the approach in Section \ref{sec:discrete2} if we had established the corresponding version of Proposition \ref{prop20a} (see Remark \ref{remark1}).

\bibliography{mybib} 
\bibliographystyle{alpha}

\end{document}